\documentclass{amsart}
\usepackage{amsmath}
\usepackage{amsfonts}
\usepackage{graphics}
\usepackage{epsfig}
\usepackage{amssymb}
\usepackage{amscd}
\usepackage[all]{xy}
\usepackage{latexsym}
\usepackage{graphicx}
\usepackage{multirow}
\usepackage{geometry}
\usepackage{color}
\usepackage[usenames,dvipsnames]{pstricks}
\usepackage{epsfig}
\usepackage{pst-grad} 
\usepackage{pst-plot} 
\usepackage{wasysym}

\usepackage[utf8]{inputenc}
\usepackage[english]{babel}
\usepackage{hyperref}

 \usepackage{textcomp}

 \usepackage{stmaryrd}

\newtheorem{thm}{Theorem}[section]

\newtheorem{cor}[thm]{Corollary}
\newtheorem{lem}[thm]{Lemma}
\newtheorem{prop}[thm]{Proposition}

\newtheorem{quest}[thm]{Question}
\newtheorem{prob}[thm]{Problem}
\newtheorem{prin}[thm]{Principle}
\theoremstyle{definition}
\newtheorem{Def}[thm]{Definition}
\newtheorem{rem}[thm]{Remark}
\newtheorem*{ack}{Acknowledgement}

\numberwithin{equation}{section}
\numberwithin{figure}{section}

\def\Hom{{\text{\rm{Hom}}}}
\def\End{{\text{\rm{End}}}}

\def\rchi{{\hbox{\raise1.5pt\hbox{$\chi$}}}}
\def\Aut{{\text{\rm{Aut}}}}

\def\isom{\cong}

\def\tensor{\otimes}
\def\dsum{\oplus}
\def\Ker{{\text{\rm{Ker}}}}
\def\Coker{{\text{\rm{Coker}}}}

\def\const{{\text{\rm{const}}}}
\def\deg{{\text{\rm{deg}}}}
\def\rank{{\text{\rm{rank}}}}

\def\a{\alpha}
\def\b{\beta}
\def\o{\omega}

\def\sig{\sigma}
\def\lam{\lambda}
\def\Lam{\Lambda}
\def\gam{\gamma}
\def\Gam{\Gamma}
\def\e{\epsilon}

\def\Jac{{\text{\rm{Jac}}}}
\def\Spec{{\text{\rm{Spec}}}}

\def\Proj{{\text{\rm{Proj}}}}

\def\Pic{{\text{\rm{Pic}}}}

\def\Ext{{\text{\rm{Ext}}}}
\def\Jac{{\text{\rm{Jac}}}}

\def\lcm{{\text{\it{LCM\,}}}}

\newcommand{\bea}{\begin{eqnarray}}
\newcommand{\eea}{\end{eqnarray}}
\newcommand{\be}{\begin{equation}}
\newcommand{\ee}{\end{equation}}

\newcommand{\Mbar}{{\overline{\mathcal{M}}}}

\newcommand{\bP}{{\mathbb{P}}}
\newcommand{\bC}{{\mathbb{C}}}

\newcommand{\bE}{{\mathbb{E}}}
\newcommand{\bF}{{\mathbb{F}}}
\newcommand{\bL}{{\mathbb{L}}}

\newcommand{\bQ}{{\mathbb{Q}}}
\newcommand{\bR}{{\mathbb{R}}}

\newcommand{\bZ}{{\mathbb{Z}}}

\newcommand{\cM}{{\mathcal{M}}}

\newcommand{\cD}{{\mathcal{D}}}
\newcommand{\cE}{{\mathcal{E}}}
\newcommand{\cF}{{\mathcal{F}}}

\newcommand{\cL}{{\mathcal{L}}}
\newcommand{\cO}{{\mathcal{O}}}

\newcommand{\cS}{{\mathcal{S}}}

\newcommand{\cU}{{\mathcal{U}}}

\newcommand{\la}{{\langle}}
\newcommand{\ra}{{\rangle}}
\newcommand{\half}{{\frac{1}{2}}}

\newcommand{\bq}{{\mathbf{q}}}

\newcommand{\rar}{\rightarrow}
\newcommand{\lrar}{\longrightarrow}

\newcommand{\hxi}{{\hat{\xi}}}

\newcommand{\bite}{\begin{itemize}}
\newcommand{\eite}{\end{itemize}}

\newcommand{\benu}{\begin{enumerate}}
\newcommand{\eenu}{\end{enumerate}}

\begin{document}
\large
\setcounter{section}{-1}


\title[In Search of a Hidden Curve]
{In Search of  a Hidden Curve}

\author[M.~Mulase]{Motohico Mulase}
\address{Motohico Mulase:
Department of Mathematics\\
University of California\\
Davis, CA 95616--8633}
\email{mulase@math.ucdavis.edu}
\address{and Kavli Institute for the Physics and Mathematics of the Universe (WPI)\\
The University of Tokyo\\
 Kashiwa 277-8583, Japan}

\begin{abstract} 
It has been noticed since around 2007 that certain enumeration problems
can be solved when an analytic or algebraic curve is identified. 
This curve is the key to 
the problem. In these lectures, 
a few such examples are presented. One is a detailed account on counting
simple Hurwitz numbers, explaining
how the problem was solved by discovering this key  curve. The formula for
the curve allows us to write
the generating functions of Hurwitz numbers in terms of
polynomials. This unexpected polynomiality  produces, as a byproduct, 
straightforward and short
proofs of the Witten-Kontsevich theorem and the 
$\lambda_g$-theorem of Faber-Pandharipande. 
An analogous enumeration problem associated with Catalan 
numbers is also presented, which has a simpler feature in terms of analysis. 
The asymptotic behavior of  counting leads this time 
to the Euler characteristic of the moduli spaces of smooth curves.  
We then discuss another enumeration problem,  
the
Ap\'ery sequences. The  quest of
identifying the   hidden  curve for this case   remains open.

These curves, also  known as  \emph{spectral curves},
are discovered  via solving
ordinary differential equations. 
The counting problem of geometric origin associated with the
\emph{genus $0$, one marked point} case is encoded in the spectral curve. 
It is explained that going from the $(g,n)=(0,1)$-case to arbitrary 
$(g,n)$ is a process of \emph{quantization}
of the spectral curve.

This 
perspective of quantization is   discussed in a geometric setting,
when the differential equations are linear with  holomorphic
coefficients, in terms  of Higgs bundles, opers, and Gaiotto's conformal 
limit construction. In this context, however, there are no counting problems behind the 
scene. 
 \end{abstract}



\thanks{Illustrations are created by
the author.}

\thanks{The research 
of the author has been supported by NSF-FRG grant DMS-2152257.}

\maketitle

\tableofcontents

\allowdisplaybreaks

\section{Introduction}

\subsection{The story begins with Catalan numbers and their Laplace transform}

Let us start with two differential equations
\begin{align}
\label{PF}
&\left((x^2-4)\;\frac{d^2}{dx^2} + x\; \frac{d}{dx} -1
\right) z(x) = 0,
\\
\label{HW}
&\left(\hbar^2 \frac{d^2}{dx^2} + \hbar x\frac{d}{dx}+  1\right)
\Psi(x,\hbar)
=0.
\end{align}

\begin{quest}
What is the common algebraic curve hidden behind these two 
equations?
\end{quest}

The first equation is easy to solve. Obviously $x$ itself is a solution,
and $\sqrt{x^2-4}$ also solves it. We can thus choose
$$
z(x) = \frac{x\pm \sqrt{x^2-4}}{2}
$$
as a basis for all solutions. It then reminds us of the quadratic formula for
$z^2-xz+1=0$,
and its possible relation to the second equation through the \emph{Weyl
quantization}
$
\begin{cases}
z\longmapsto -\hbar\frac{d}{dx},\\
x\longmapsto x.
\end{cases}
$
Equivalently, we can find the same  polynomial from 
the \emph{semi-classical limit} of \eqref{HW}:
$$
\lim_{\hbar\rar 0}e^{-\frac{S_0(x)}{\hbar}}
\left(\hbar^2 \frac{d^2}{dx^2} + \hbar x\frac{d}{dx}+  1\right)
e^{\frac{S_0(x)}{\hbar}} = \big(S_0(x)'\big)^2 + x S_0(x)' + 1, \quad 
z=-S_0(x)'.
$$

\begin{figure}[htb]
\includegraphics[height=1.5in]{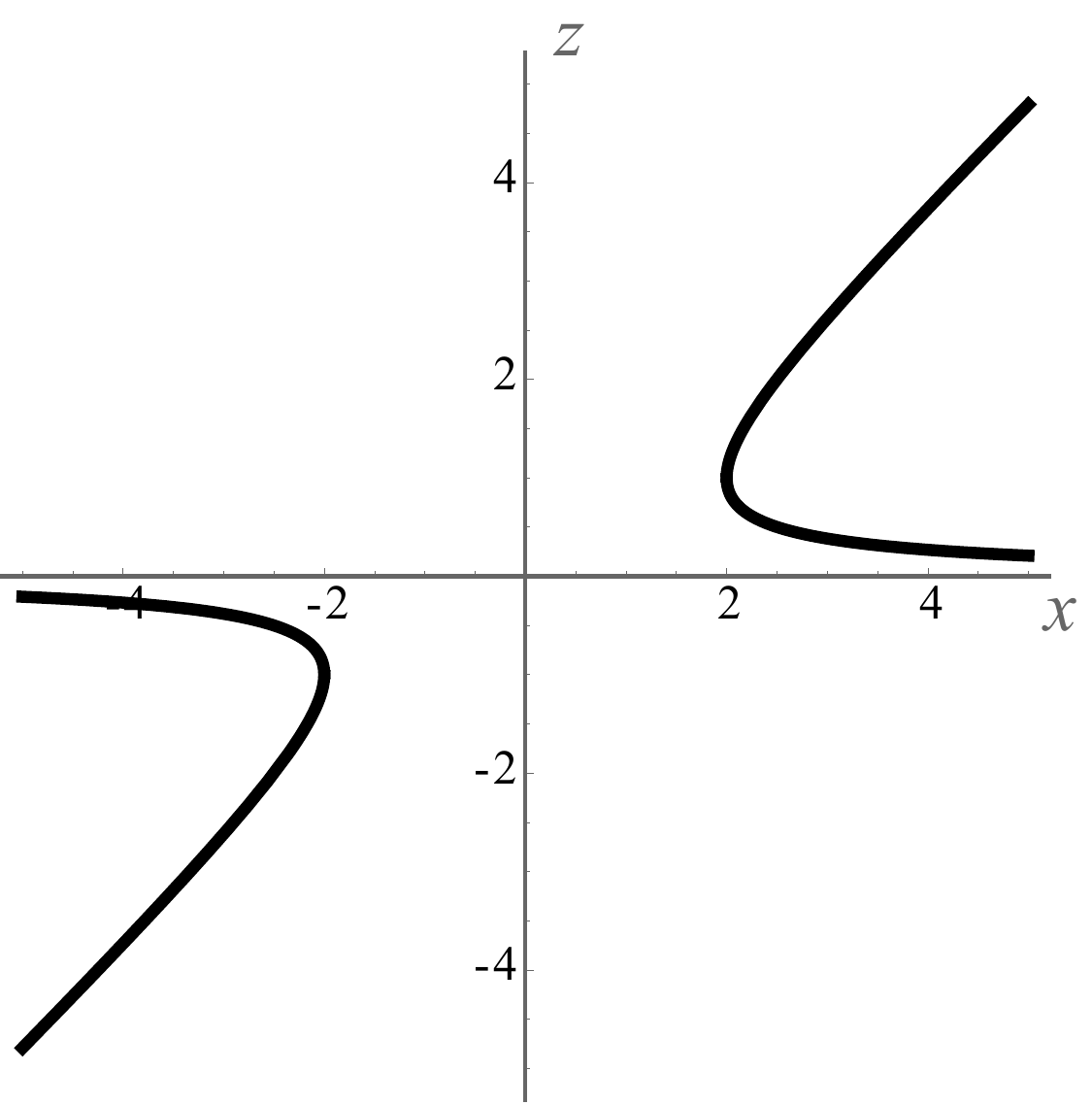}
\caption{The mirror dual spectral curve $\Sigma$ of the Catalan numbers.
}
\label{fig:xz}
\end{figure}

The curve in common appears to be  
$
x = z+\frac{1}{z}.
$
Is there anything significant here? We will see that this curve tells  
 us a rich story of geometry that is \emph{not}  obvious from the 
shape of these differential equations.
As explained in the main text \eqref{z(x)},  $z(x)$ 
is an unconventional 
generating function of Catalan numbers $C_m: 1, 2, 5, 14, 42, 132, 429,
1430, 4862,16796\dots$ One interpretation of these numbers 
is the count of \emph{cell-decompositions} of a two-dimensional
sphere $S^2$ with one $0$-cell and $m$ $1$-cells, while not allowing rotation
symmetry   to avoid complication coming from 
automorphism count (see Remark~\ref{categorical}).

\begin{figure}[htb]
\label{01graph}
\includegraphics[width=1.5in]{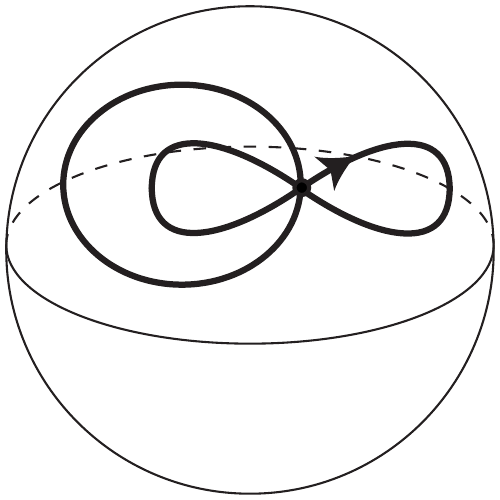}
\caption{A cell-decomposition of a curve of genus $0$ with
one $0$-cell. The arrow on a half-edge indicates no rotation symmetry is allowed.}
\end{figure}

From this point of view, the counting can be generalized
to cell-decompositions on an
arbitrary compact oriented surface of genus $g$ with
$n$ $0$-cells.  This is the story of Section~\ref{subsect:Catalan},
based on  \cite{CMS, OM4, DMSS}.
An unexpected fact is that
the generating functions $F_{g,n}(t_1,\dots,t_n)$ of the 
numbers of these cell-decompositions
are \emph{\textbf{Laurent polynomials}} of degree
$6g-6+3n$, when $2g-2+n>0$. 
This \emph{polynomiality} is
surprising, and happens only in 
the right choice of variables  defined by
 the above $z(x)$ as in Theorem~\ref{thmCatalan}. 
It gives  a key to relate the story with many different 
subjects of the geometry of moduli spaces. 
We will show that these Laurent polynomials know,
as their special values,
the Euler characteristic $\rchi(\cM_{g,n})$ of the moduli spaces
of smooth pointed curves first
calculated by Harer-Zagier \cite{HZ}, and  the 
Witten-Kontsevich formula for intersections of the 
 tautological cotangent classes on $\Mbar_{g,n}$
\cite{DVV,K1992,W1991} through their asymptotic behavior.

What do these geometric quantities of moduli spaces
 have anything to do with the 
differential equation  \eqref{HW} above?
We will see that 
its solution $\psi(x,\hbar)$   is the \emph{generating function} of these
generating functions $F_{g,n}$ for all $g\ge0$ and $n>0$.
In this context, since the curve $x=z+1/z$ is the \textbf{semi-classical limit}
of  \eqref{HW},  the differential equation  is
called the \textbf{quantum curve} \cite{ADKMV, DHS, Holl}
associated with the  classical curve $x=z+1/z$.
The Catalan numbers are the count of cell-decompositions of the
$(g,n)=(0,1)$ geometry. Hence going from $z(x)$ to 
$F_{g,n}(t_1,\dots,t_n)$ for arbitrary $(g,n)$ is the process of 
\emph{quantization}, in  analogy of counting graphs of higher
genera as \emph{Feynman diagrams} representing interactions. 

At the same time, the fact that generating functions 
$F_{g,n}(t_1,\dots,t_n)$ know the $\psi$-class
intersection numbers of  $\Mbar_{g,n}$ indicates that 
they give the Gromov-Witten invariants $GW_{g,n}(\bullet)$
of a point. Therefore, from the point of view of 
Mari\~no \cite{Mar} and Bouchard-Klemm-Mari\~no-Pasquetti
\cite{BKMP}, we understand that the spectral curve $x=z+1/z$ is
the \emph{mirror} $B$-model, corresponding to the 
 $A$-model on a   \emph{trivial} symplectic manifold, i.e., a point.
From the Catalan numbers to this point of view of mirror symmetry
is the theme of Section~\ref{subsect:Catalan}.

The polynomiality of some quantities appearing in 
enumerative problems associated with the moduli spaces $\Mbar_{g,n}$
was  discovered in 
an earlier paper \cite{MZ},  which was the key to the solution
\cite{EMS} of  the 
Hurwitz number conjecture of Bouchard and Mari\~no \cite{BM}.
The polynomiality of the generating functions of 
simple Hurwitz numbers established in \cite{MZ} presented 
another surprise: it gives  simple few-line proofs of Witten's conjecture
\cite{W1991} and the $\lambda_g$-conjecture of Faber-Pandharipande
\cite{FP1, FP2}. 

The stories coming out of Hurwitz numbers 
are weaved in Section~\ref{sect:Hurwitz}.
The hidden curve in this context is the \textbf{Lambert curve} \eqref{eq:Lambert},
whose role in discovering the polynomiality is explained in detail.
Bouchard and Mari\~no \cite{BM} identified the Lambert curve through the 
limiting process of the \emph{mirror curves} of 
toric Calabi-Yau $3$-folds \cite{BKMP}, and using Lagrange's Inversion Formula \cite{WW1927}.
Appendix gives an elementary analysis that leads to this identification of 
the Lambert curve.
In Section~\ref{Coda}, we will show that the Lambert curve
is  a simple consequence of \emph{tree counting}, 
and the formula for the curve, i.e., the \emph{Lambert
function}, directly follows from a combinatorial identity
of trees.

Our point is that all these known results are obtained by finding the
hidden curve behind the scene, the \textbf{spectral curve}
of the counting problem. 
And the spectral curve is always the generating function 
of the type $(0,1)$-invariants of the problem.
The \textbf{quantum curve} is a family of $\hbar$-deformations of 
a differential operator whose 
\emph{semi-classical limit} is the spectral curve. 
We will explain this relation for the case of  Catalan numbers 
and Hurwitz numbers in detail in the main text.

\begin{rem} We will use the terminology \textbf{Laplace Transform}
in a slightly more general context.
\bite

\item The process from the sequence $\{C_m\}_{m=0}^\infty$
of Catalan numbers
to their \emph{generating function} and its inverse function
$$
z(x) = \sum_{m=0}^\infty \frac{C_m}{x^{2m+1}}, \qquad x=z+\frac{1}{z}
$$
is considered as the \textbf{Laplace transform} in these lectures. 
This is an idea developed over the years (see for example, \cite{CMS,DMSS, MP2012,MS2008,MS2015,MZ}).
The rationale behind it 
is that when a sequence satisfies a combinatorial relation,  we can take
the Laplace transform of the relation. Often the result becomes
a system of differential equations. Therefore, the \textbf{Laplace transform
changes combinatorics to geometry}. 

\item This effect is explained  in Section~\ref{sect:Hurwitz}
using Hurwitz numbers.
The starting sequence is the number of \emph{rooted trees} on 
$k$ nodes and their generating function, 
$$
y(x) = \sum_{k=1}^\infty \frac{k^{k-1}}{k!} x^k. 
$$
We will show that a simple combinatorial identity of tree counting 
allows us to identify its inverse function, the Lambert curve, $x=y e^{-y}$,
in Section~\ref{Coda}.

\item Both $y(x)$ and $z(x)$ above land on $\bP^1$, so we use an automorphism of
$\bP^1$ to bring these variables to the ``right'' coordinate, which we write
as $t\in \bP^1$. The key of Section~\ref{subsect:Catalan} is that
the Laplace transform of the genus $g$, $n$ marked point version of the 
Catalan numbers $C_{g,n}(\mu_1,\dots,\mu_n)$ 
for $2g-2+n>0$ is a Laurent polynomial
in the $t$-variables. The special values of these Laurent polynomials 
are identified as the Euler characteristic $\rchi(\Mbar_{g,n})$ of the
moduli space of smooth pointed curves \cite{MP2012}, and
the asymptotic behavior of these Laurent polynomials recovers the 
$\psi$-class intersection numbers on the moduli space of 
stable curves $\Mbar_{g,n}$ \cite{CMS}.

\item In Section~\ref{Fugue2} it is explained
that the Laplace transform of the $(g,n)$-Hurwitz number in the 
stable region $2g-2+n>0$ is a polynomial in the $t$-variables. 
 We will prove that the Laplace transform of the combinatorial
formula, the \emph{cut-and-join} equation of \cite{GJ,V}, 
is a differential equation, and that it automatically
proves the Witten-Kontsevich theorem
on the $\psi$-class intersection numbers  and the Faber-Pandharipande theorem
on the $\lambda_g$-conjecture \cite{MZ}, through the ELSV formula
\cite{ELSV}.

\item So what is the Laplace transform, after all? Our thesis
in Section~\ref{subsect:Laplace}  is that the
\textbf{Laplace transform is the mirror symmetry}. 
\eite
\end{rem}


\subsection{A miraculous integer sequence}

In the final Section~\ref{zeta3}, we explore  stories from  uncharted territories. Let 
$H_m = 1+\half + \frac{1}{3}+\cdots +\frac{1}{m}$ be the
$m$-th harmonic number. We use the convention $H_0=0$. 
For every $n\ge 0$,  define
$$
\widetilde{GW}_{0,1}(n):= (-1)^n (n!)^2 \sum_{\substack{\ell+m=n\\
\ell,m\ge 0}}
\frac{(2\ell+m)! (\ell+2m)!}
{(\ell !)^5  (m!)^5}\bigg(1+(m-\ell)\big(H_{2\ell+m} 
+2H_{\ell + 2 m} -5 H_m\big)\bigg)
$$ and
$$
A_n:= \sum_{\ell=0}^n\binom{n}{\ell}^2 \binom{n+\ell}{\ell}^2.
$$
Then we have (see \cite{CCGK,GZ,Zagier}):
\be
\label{mysteryintro}
A_n=\widetilde{GW}_{0,1}(n) \qquad \text{for all }\;n.
\ee
The first few terms are: $1, 5, 73, 1445, 33001, 819005,21460825,
584307365,16367912425,\dots$. What is unbelievable 
about this equality is that $\widetilde{GW}_{0,1}(n)$ as defined 
is a \textbf{positive integer}, and has a simple expression
as in the second line.
What we know about \eqref{mysteryintro} so far includes the following:
\bite

\item The formula comes from the \textbf{mirror symmetry}, 
a geometric relation between the Gromov-Witten invariants of a 
complex $3$-dimensional variety
and the Gauss-Manin connection associated with its mirror partner. 
In this particular context, we refer to Golyshev-Zagier \cite{GZ}
and
 Zagier \cite{Zagier} for an inspiring account of this interplay.
 
\item The quantity  $\widetilde{GW}_{0,1}(n)$, the \textbf{quantum
period},  was obtained by 
Coates, Corti, Galkin, and Kasprzyk \cite{CCGK}, which is
essentially 
the same  as the
genus $0$, $1$-marked point case of the degree $n$
 Gromov-Witten invariant
of a Fano $3$-fold known as $V_{12}$. The generating function of
these numbers, after appropriate modifications including
the Borel-Laplace transform,  satisfies a \textbf{Quantum Differential Equation}. A brief background  will be explained  in 
 Section~\ref{zeta3}.
 
\item The sequence $\{A_n\}_{n=0}^\infty$ was discovered by 
Ap\'ery in 1978 \cite{A} in his \textbf{proof of the irrationality} of 
a special value 
$\zeta(3)$ of the 
Riemann zeta function. A quick review of his proof,
including the significance of the integer sequence 
$\{A_n\}_{n=0}^\infty$, is presented in Section~\ref{subsect:Apery}.

\item The differential equation that determines  the generating function of 
$\{A_n\}_{n=0}^\infty$   was discovered to be the \textbf{Picard-Fuchs equation} of
a particular $1$-parameter family of K3 surfaces by 
Beukers and Peters \cite{BP}.
\eite

What we can hope to see happens, yet still not established, is the following:
\bite
\item There should be a  \textbf{spectral curve} hidden behind the scene,
determined by the type $(0,1)$ Gromov-Witten invariants of $V_{12}$. 
This should essentially be the same object as the Picard-Fuchs
equation mentioned above, being on the $B$-model side.
\item There should be an extension of  the \textbf{quantization procedure}, explained in
Section~\ref{subsect:oper} for 
differential equations with holomorphic coefficients, to differential equations with 
\emph{irregular singular points}. This is the story of \textbf{quantum 
curves}.
\eite

\subsection{The unique quantization  for the holomorphic cases}
The correspondence between classical systems and quantum systems
is never one-to-one. If we imagine the classical system
to be the limit of a quantum system 
as
the Planck constant
$\hbar\rar 0$, then we can add 
anything multiplied with $\hbar$  to the quantum equation. 
This addition vanishes in this limit. 

It is therefore rather counter intuitive that there is a bijective,
and even biholomorphic, correspondence between spectral curves as classical systems
and quantum curves as quantum systems, discovered
in  \cite{DFKMMN,OM5},
prompted by the idea of Gaiotto \cite{Gai2014}.
Suppose we have a  holomorphic 
$n$-th order ($n\ge 1$) linear ordinary differential
operator $P$ \emph{globally defined} on a smooth complete 
complex algebraic curve $C$ 
of genus $g(C)>1$, whose leading coefficient does not
vanish anywhere on $C$. One interpretation of the above 
result   proves that 
it determines  
a unique algebraic curve $\pi: \Sigma \lrar C$, known 
as a \emph{spectral curve} of a \emph{Higgs} bundle, that recovers $P$ through
the \emph{conformal limit construction} of Gaiotto \cite{Gai2014}.
The passage from $P$ to $\Sigma$ includes: 
\bite
\item Identification
of the canonical (i.e.,  unique) $\hbar$-deformation family $P^\hbar$
of $P$ such that $P$  appears in this family at $\hbar=1$; and 
\item Calculating its \emph{semi-classical limit} as $\hbar\rar 0$ that 
selects $\Sigma$ as the classical geometric object corresponding to this
family of deformations.
\eite
This is counter intuitive because a family $P^\hbar$ can determine $P$ 
at $\hbar=1$, but not in the other way around. Also, from classical $\Sigma$
to a quantum $P^\hbar$ is usually never unique.

In terms of a local coordinate $z$ of a local neighborhood
 $U\subset C$, $P$ is a noncommutative
polynomial in $d/dz$ of degree $n$ with coefficients in 
holomorphic functions in $\cO_C(U)$. The $\hbar$-deformation changes it
to a $1$-dimensional Schr\"odinger operator on $U$, which patches
together to a 
globally  defined operator $P^\hbar$ on $C$. The process of semi-classical limit is 
equivalent to taking the ``total symbol'' of this differential operator.
After a suitable choice of a local 
coordinate and a conjugation action of the operator
with a locally non-vanishing function, we can write 
\be
\label{nth order}
P=\left(\frac{d}{dz}\right)^n + \sum_{k=2} ^n a_k(z) 
\left(\frac{d}{dz}\right)^{n-k}, \qquad a_k\in \cO_C(U)
\ee
on $U$.
Then the equation  $P\psi = 0$ is equivalent to $\nabla_z \Psi = 0$,
where
\be
\label{introconn}
\nabla_z := \frac{d}{dz} + \begin{bmatrix}
0&a_2&\cdots&a_{n-1}&a_n\\
-1&\\
&-1&&&\vdots\\
&&\ddots\\
&&&-1&0
\end{bmatrix} \quad \text{and}\quad
\Psi
=
\begin{bmatrix}
\psi^{(n-1)}\\
\psi^{(n-2)}\\
\vdots\\
\psi'\\
\psi
\end{bmatrix}.
\ee
Therefore, locally an $n$-th order differential equation $P\psi=0$
is always equivalent to the flatness equation $\nabla_z\Psi=0$ with respect
to a holomorphic connection $\nabla_z$. In Section~\ref{sect:var},
we will translate the condition that $P$ is globally defined on $C$ into
a set of properties of this connection. The requirements include that
$P$ acts on the line bundle $K_C^{-\frac{n-1}{2}}$, and that
the connection $\nabla$ acts on a vector bundle 
that has a full-flag filtration with \textbf{Griffiths transversality}.
Such a connection is known as an \textbf{oper} \cite{BD, BZF}.

The statement we mentioned above is thus equivalent to the assertion that 
\emph{every oper on $C$ corresponds in a one-to-one manner
to  a spectral curve $\Sigma$ covering $C$.}
The $\hbar$-deformation of $\nabla$ is the $\hbar$-connection 
$\nabla^\hbar$ of
Deligne, and the parameter $\hbar$ is identified as an element
$\hbar\in H^1(C,K_C)$ (see \cite{OM5}), which determines the \emph{extension classes}
of line bundles associated with the full-flag filtration.
In \cite{DFKMMN}, 
a globally defined linear differential operator on a curve $C$, or an oper, 
is constructed from an
arbitrary spectral curve of a $G$-Hitchin system on a base curve
$C$, where $G$ is
a simple Lie group of adjoint type. This gives a biholomorphic 
map from the moduli space of spectral curves to the moduli space
of opers. Bijection means that the original spectral curve is
uniquely determined by the corresponding oper. And  the
construction also implies that a higher order
differential operator uniquely identifies a \emph{quantization}
of the spectral curve, i.e., the $\hbar$-family of deformations
of the starting operator, or equivalently, a Deligne's
$\hbar$-connection. 

There is an important difference between  holomorphic 
differential operators $P$ on
a curve $C$ of genus $g(C)>1$, and the examples 
coming from enumeration problems mentioned
above. 
When  $P$, and its corresponding oper $\nabla$,  is  holomorphic, 
any \emph{solution} 
to the equation $P\psi = 0$ is  everywhere holomorphic.  So we do not expect
it to contain any new geometric information of something that goes
beyond the given context,
such as topology of $\Mbar_{g,n}$ as in the examples. 
We expect that when we consider
$P$ with  \emph{irregular singularities},  a whole new story   begins.
This is an active area of research in geometry. We refer to 
\cite{ABF, CFW, CW2018,GMN}.

\begin{quest}
Is there  an analogous  correspondence between differential operators
with irregular singular points and singular spectral curves, both
defined over $C$?
\end{quest}

Even for $C=\bP^1$, if such a correspondence is established, then it
should tell us a lot of stories behind some deep mysteries, 
such as the geometry behind the irrationality of $\zeta(3)$.

\subsection{A quick guide of the contents}
\bite
\item
The story given in Section~\ref{subsect:Catalan}
illustrates the model of the theory: The expansion of a solution 
of a differential equation ($=$ quantum curve) around
its irregular singular point  contains a profound amount of 
geometric information 
not obvious from the given setting. 

\item A story of Hurwitz numbers is presented in Section~\ref{sect:Hurwitz}
with some details.  In this case
the quantum curve corresponding to the spectral curve is a 
difference-differential operator,
or a differential operator of an infinite-order. This is due to the fact that
the spectral curve,  the Lambert curve in this case, is an analytic curve, not an algebraic curve.
So far we do not have any counterpart generalization of the 
theorem of \cite{DFKMMN}
for difference operators.

\item What we mean by a globally defined high order linear differential
operator on a curve $C$ of genus $g(C)>1$  is explained in Section~\ref{subsect:higherorder}.
There and in the following Section~\ref{subsect:oper}, we will encounter the geometric meaning of the parameter
$\hbar$, how it determines the \emph{unique} quantization from 
the starting classical spectral curve, and how the 
\textbf{projective coordinate system} \cite{Gun} appears in the global construction.

\item An analogy of two differential equations 
described in \eqref{PF}, a Picard-Fuchs equation,
and  an Hermite-Weber equation \eqref{HW}, 
appears in a context of Gromov-Witten invariants of a
Fano $3$-fold. This is also deeply and mysteriously related to the 
integer sequence playing a key role in the \emph{irrationality proof} 
of  $\zeta(3)$. 
This open-ended story is presented at the end of these 
lectures in Section~\ref{zeta3}.
\eite

These are the  notes  based on the author's series of 
talks delivered in K\"oln, Hamburg, Osaka, Oxford, Z\"urich, 
Riverside, Kyoto, Hiroshima, Kobe,
Les Diablerets,  Madrid, and Seattle 
 in the last two years or so. They are
designed to tell a story of the 
exploration: ``\emph{In Search of a Hidden Curve}.''

\begin{ack}
The author would like to express his
hearty gratitude to  the organizer, Laura P.\;Schaposnik, of the
\emph{Workshop and School on Complex Lagrangians, Integrable Systems, and Quantization}, which was 
held at the Mathematical Institute of the University of Oxford
 in summer 2023,  for her invitation and kind hospitality, and her continued support and encouragement
 to him over the past decade.
 His thanks are  due to Schaposnik's 
co-organizer and co-editor of the volume  Mengxue Yang
  for her
effort to make the workshop and the volume successful. 
He thanks Zachary Ibarra for many  discussions on the topics of 
$\zeta(3)$, which were very useful to compile 
Section~\ref{zeta3} of these notes.

The author  would also  like to express his  gratitude to Don Zagier for 
his kind invitation to spend a sabbatical term of Fall 2022 at 
Max-Planck-Institut f\"ur Mathematik, Bonn, and to 
Takuro Mochizuki and Masa-Hiko Saito for their generous invitation 
to spend the Spring-Summer term of 2024
at the Research Institute for Mathematical Sciences, Kyoto University,
in part of the \emph{RIMS Research Project 2024: Development in Algebraic Geometry related to Integrable Systems and Mathematical Physics}. 
The final version of these notes were drafted  
during the author's stay in Madrid. He thanks 
Monica Jinwoo Kang, Marina Logares, and Piergiulio Tempesta 
for their kind invitation and hospitality. During the preparation of the present work, the author received financial
support from Max-Planck-Institut f\"ur Mathematik-Bonn,
Kyoto University, and Universidad Complutense de Madrid,
which are gratefully acknowledged. 

Last but not least, the author's 
special thanks are due to Olivia M.~Dumitrescu for her over a decade
long discussions and collaboration with the author that led to numerous joint 
publications, including \cite{DMSS} featured in Section~\ref{subsect:Catalan} and \cite{OM5} in Section~\ref{subsect:oper}.

The research 
of the author has been supported by NSF-FRG grant DMS-2152257.
\end{ack}

\section{Prelude}
\label{sect:prelude}

\subsection{Theme $1$: Mirror symmetry of Catalan numbers}
\label{subsect:Catalan}

Mathematics thrives on mysteries. \emph{Mirror symmetry}
has been a great mystery for a long time,
and has served as a driving force in many areas of
mathematics. Even after three and a half decades
since its conception in physics,
it still produces new challenges to mathematicians.
One of the  starting points of this forever expanding universe of 
research frontier is
the 1991 discovery of Candelas, De La Ossa, Green and Parkes
\cite{CdlOGP}. From this paper's scope, we
learn  that the following four subjects of mathematics,
\begin{itemize}
\item combinatorial counting problems,
\item  algebraic geometry over $\bC$, and more lately over  $\bF_q$
and $p$-adic fields,
\item Picard-Fuchs  differential 
equations, and
\item nonlinear integrable systems
\end{itemize}
are deeply intertwined in a manner beyond the realm of 
classical mathematics, and that an interplay of these different subjects
 leads us to new
insight and further understanding of mathematics. Some of the developments of
mathematics stemming out of mirror symmetry are compactly 
characterized as \emph{quantum} mathematics.

To illustrate how these  items listed above appear and interact together, 
let us begin  with a na\"ive question: 
\begin{quest}
What is the mirror symmetric dual of the Catalan  numbers?
\end{quest}
\noindent
Here our usage of the terminology 
\emph{mirror symmetry} is not conventional.
 At least at this moment, our question  
is not directly interpreted from the point of veiw of the
\emph{homological mirror symmetry} of \cite{K1994}.
Catalan numbers $C_m = \frac{1}{(m+1)} \begin{pmatrix}2m\\m
\end{pmatrix}$, a sequence of positive integers, 
exhibit  solutions to many different combinatorial problems.

The first aim of these lectures is to present a few problems that have been 
solved when a hidden \emph{curve} behind the scenes is identified.
These curves are commonly called \emph{spectral curves}. 
As we see below, as soon as the spectral curve is identified,
it acts as a catalyst to let the subject in question weave the whole story
in front of us. And the important common feature of these curves
is that they are \emph{complex Lagrangians} in a complex symplectic
surface. 

Spectral curves have appeared in many different contexts in the past,
including integrable systems of nonlinear PDEs such as the KdV equations, 
Hitchin's work on Higgs bundles, and analysis of random matrices.
In each situation, there is a concrete definition of what people call a
spectral curve. In these lectures, however, we do not attempt to give a universally 
valid definition of spectral curves, except that we say,
\textbf{a spectral curve is a collection of eigenvalues of 
an operator that takes the shape of a curve}. The fact that we have a spectral
curve for our counting problem suggests that it has a
 hidden connection to integrable systems, algebraic geometry,  random
matrix theory, and beyond.

In \cite{DMSS}, 
we  proposed  that a 
\emph{spectral curve}
\be
\label{xz}
\Sigma = \left\{ (x,z)\;\left|\;  x = z + \frac{1}{z}\right.\right\}
\ee
is the mirror dual to the Catalan numbers.
It is not hard to see why this curve 
has something to do with  Catalan numbers:
The inverse function $z=z(x)$ of the equation $x=z+1/z$  that satisfies 
$\lim_{x\rar \infty} z(x) = 0$ is an unconventional generating function of 
Catalan numbers
\be
\label{z(x)}
z(x) = \sum_{m=0}^\infty C_m \; \frac{1}{x^{2m+1}}.
\ee
This is an absolutely convergent series for $|x|>2$. First, we observe
that this curve immediately
defines a differential equation. 
Since $x$ and $z(x)$ satisfy a polynomial relation 
$z^2-xz + 1=0$, the $x$-derivatives $z'(x)$, $z''(x)$, $z'''(x),\dots$
are all in the extension field $\bC(x,z)$ of degree $2$
over the field  of rational functions $\bC(x)$.
Therefore, $z, z'$, and $z''$ are linearly dependent over 
the polynomial ring $\bC[x]$. The simplest linear relation is
\be
\label{xzdiff}
\left((x^2-4)\;\frac{d^2}{dx^2} + x\; \frac{d}{dx} -1
\right) z(x) = 0.
\ee
This differential equation  is equivalent to
the recursion formula
$$
C_m = \frac{2(2m-1)}{m+1} C_{m-1}, \qquad C_0 = 1,
$$
with respect to the generating function $z(x)$.
The differential equation \eqref{xzdiff} has three regular singular points
at $x = -2, 2, \infty$. (The definition of this terminology is given below in
Definition~\ref{regular and irregular}.)
The analytic continuations of $z(x)$  from the neighborhood 
of $\infty\in\bP^1$ are
$$
z(x)_\pm = \frac{x \pm \sqrt{x^2-4}}{2} = 
\left(\frac{x+\sqrt{x^2-4}}{2}\right)^{\pm 1},
$$
which produce obvious solutions  $x=z(x)_+ + z(x)_-$ 
and $\sqrt{x^2-4}$ to \eqref{xzdiff}.  Thus the monodromy property of 
the differential equation \eqref{xzdiff} is very simple at each singular point,
which is the consequence of the regular singularity at these points.

Even though this example is trivial in many sense, the point here
is that \eqref{xzdiff} is a \emph{Picard-Fuchs} equation
associated with the projection $\pi: \Sigma\lrar \bP^1$ defined by 
the algebraic equation $x = z + 1/z$. This is a trivial 
\emph{Landau-Ginzburg} model. A coordinate transformation $x = 4t-2$ brings this differential equation to 
one of the Euler-Gau\ss\ 
hypergeometric differential equations,
\be
\label{hyper}
\left(t(1-t) \frac{d^2}{dt^2} + (-\half +t) \frac{d}{dt} -1
\right) y(t)=0
\ee
with regular singular points at $0, 1, \infty$. 
The cohomology $H^0(\pi^{-1}(x),\bZ)$
defines a Gauss-Manin connection in a local system over $\bP^1$.

Another differential equation that is determined by 
the spectral curve \eqref{xz} is 
a Schr\"odinger equation (cf.\ the \emph{quantum curve} of \cite{OM1, OM2})
\be
\label{qxz}
\left(\hbar^2 \frac{d^2}{dx^2} + \hbar x\frac{d}{dx}+  1\right)
\Psi(x,\hbar)
=0.
\ee
This equation has only one irregular singular point at 
$\infty$, and no regular singular points. Hence a solution is
an entire function on $\bC$ with an essential singularity at $\infty$.

Let us briefly review the mechanism of semi-classical limit here. The WKB method 
allows us to find an asymptotic solution of \eqref{qxz} 
in terms of the exponential of a  \emph{Laurent series} expansion in $\hbar$.
We impose that 
$$
\psi(x,\hbar) = \exp\left(\frac{s_0(x)}{\hbar}\right) 
\exp\left(\sum_{m=1}^\infty \hbar^{m-1}s_m(x)\right)
$$
is a solution to \eqref{qxz} and derive differential equations for each 
$s_m(x)$. The idea here is that as $\hbar\rar 0$, the function 
$s_0(x)$ has a dominant importance, which should recover the
classical behavior of the quantized equation \eqref{qxz}. 
Since the above expression has no meaning as a series in $\hbar$
because each term $\hbar^n$ is an infinite sum. So we rewrite
the equation as
$$
\left[\exp\left(-\frac{s_0(x)}{\hbar}\right) 
\left(\hbar^2 \frac{d^2}{dx^2} + \hbar x\frac{d}{dx}+  1\right)
\exp\left(\frac{s_0(x)}{\hbar}\right) \right] 
\exp\left(\sum_{m=1}^\infty \hbar^{m-1}s_m(x)\right)=0.
$$
Then the operator acting on $\exp\left(\sum_{m=1}^\infty \hbar^{m-1}s_m(x)\right)$
is 
$$
\left(\hbar^2 \frac{d^2}{dx^2} + \hbar x\frac{d}{dx}+  1\right)
+2\hbar s_0'(x) \frac{d}{dx}+\hbar s_0''(x)+ \big(s_0'(x)\big)^2
+xs_0(x)' ,
$$ 
hence the equation produces only non-negative powers of $\hbar$.
The semi-classical limit is the limit of $\hbar\rar 0$ at this stage. 
Clearly, the operator converges to $\big(s_0'(x)\big)^2
+xs_0(x)' + 1$, which is multiplied to $e^{s_1(x)}$. 
Therefore, the equation becomes an algebraic equation
$\big(s_0'(x)\big)^2
+xs_0(x)' + 1=0$, which is the same as $z^2-xz+1=0$ by
substituting $s_0(x)'=-z$. 

What we find is that the semi-classical limit
in this context is equivalent to replacing 
$
\begin{cases}
\hbar \frac{d}{dx} \lrar -z,\\
x\lrar x
\end{cases}
$
in \eqref{qxz}. 
Hence the spectral curve \eqref{xz} is recovered from the quantum curve
\eqref{qxz}. In other words, the quantum curve \eqref{qxz} is 
the result of \emph{Weyl quantization} of the spectral curve
$z^2-xz+1=0$.

Since we are already using the terminology of regular singular and
irregular singular points of
differential equations, let us  explain what they are. 

\begin{Def}
\label{regular and irregular} 
Let
\begin{equation}
\label{second}
\left(\frac{d^2}{dx^2}+a_1(x)\frac{d}{dx}+a_2(x)
\right)\Psi(x) = 0
\end{equation}
be a second order differential equation
defined around a neighborhood of $x=0$ on a
small disc $|x|< \epsilon$ with meromorphic 
coefficients $a_1(x)$ and $a_2(x)$ with poles  
at $x=0$. Denote by $k$ (reps.\ $\ell$) the 
order of the pole of 
$a_1(x)$ (resp.\ $a_2(x)$) at $x=0$. 
If $k \le 1$  and $\ell\le 2$, then \eqref{second}
has a \textbf{regular singular point} at $x=0$.
Otherwise, consider the \emph{Newton polygon}
 of the order of poles of the coefficients of
 \eqref{second}. It is the upper part of
 the convex hull of three
 points $(0,0), (1, k), (2,\ell)$. As a convention,
 if $a_j(x)$ is identically $0$, then
 we assign $-\infty$ as its pole order. Let $(1,r)$
 be the intersection point of
 the Newton polygon and the line $x=1$,
$
r= \begin{cases}
 k \qquad 2k\ge \ell,\\
\frac{\ell}{2} \qquad 2k\le \ell.
 \end{cases}
$
 If $r>1$, \eqref{second} 
 has an \textbf{irregular singular point of class}
 $r-1$. 
 \end{Def}
\begin{rem}
Jacob \cite{CJ} recently discovered a $2$-parameter family of \eqref{xz}
associated with different combinatorial objects. How the story of this section
changes with this new family is a subject of future investigation. 
\end{rem}

The significance of \eqref{qxz} is in the expression of its solution,
asymptotically expanded
at its essential singularity:
\begin{equation}
\label{CatalanPsi}
\Psi(x,\hbar)=
\exp\left(
\sum_{g\ge 0, n>0}\frac{1}{n!}\hbar^{2g-2+n}
F_{g,n}(x,\dots,x)
\right)\;,
\end{equation}
where $F_{g,n}(x,\dots,x)$ is the \emph{principal specialization}
of the symmetric functions
\begin{align}
F_{0,1} (x)&= -\half z(x)^2 + \log z(x),
\\
F_{0,2} (x_1,x_2)&= -\log\big(1-z(x_1)z(x_2)\big),
\\
\label{Fgn Catalan}
F_{g,n}(x_1,\dots,x_n)
&=
\sum_{\mu_1,\dots,\mu_n>0}
\frac{C_{g,n}(\mu_1,\dots,\mu_n)}
{\mu_1\cdots\mu_n}
\prod_{i=1}^n x_i^{-\mu_i}, \quad 2g-2+n>0.
\end{align}

\begin{figure}[htb]
\includegraphics[width=2.7in]{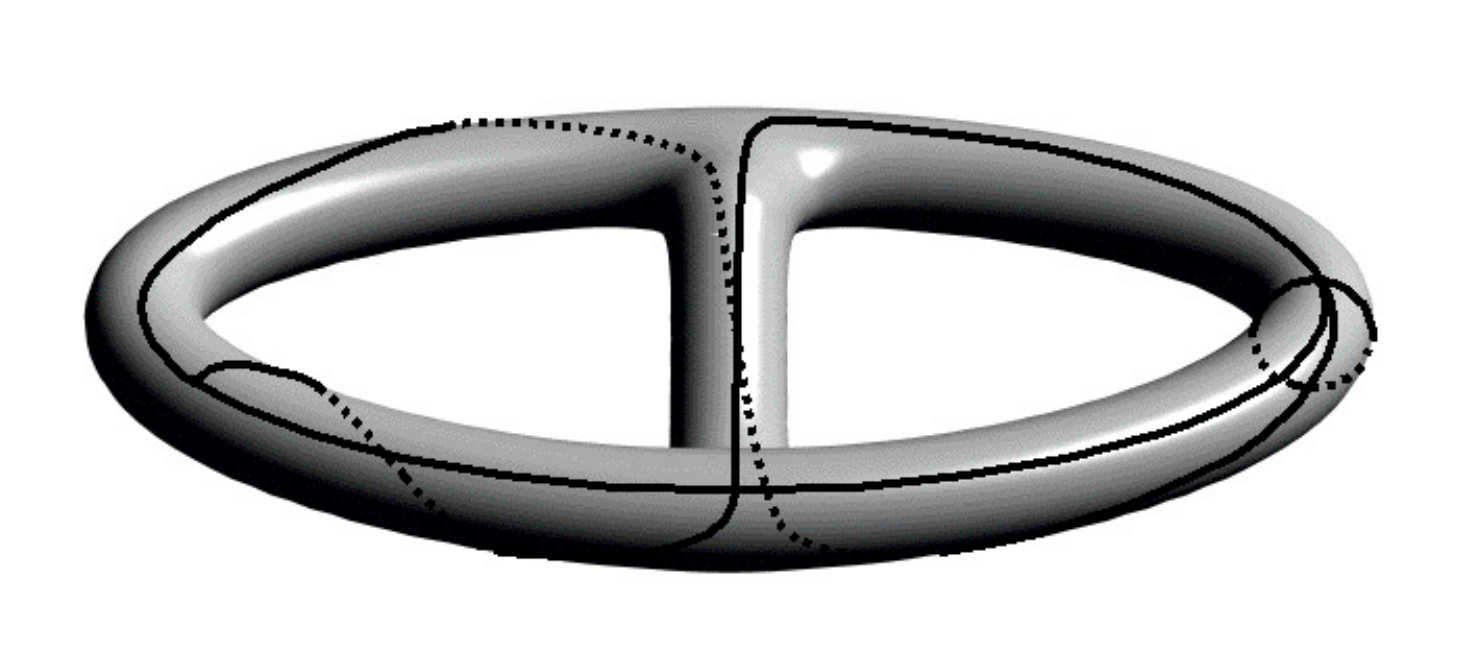}\hskip0.3in
\includegraphics[width=2.7in]{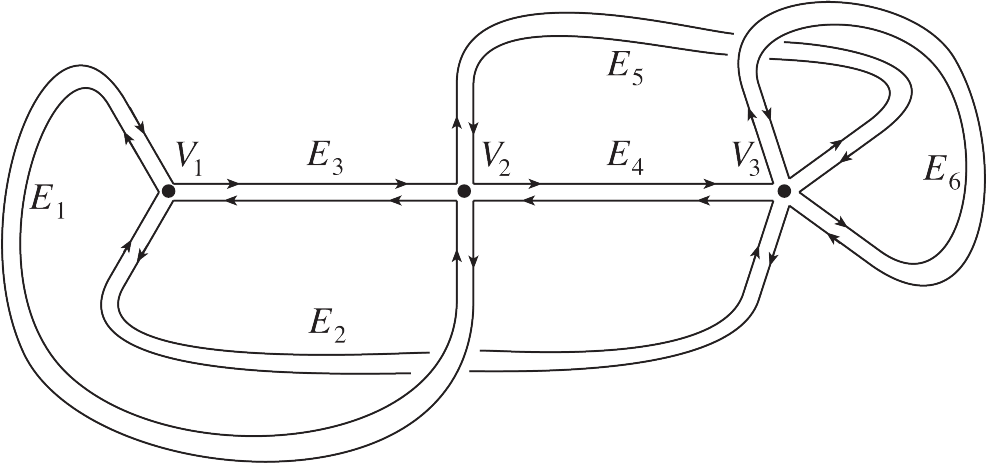}
\caption{$1$-Skeleton of a Cell Decomposition and a Cell Graph}
\end{figure}

Here comes the relation to enumeration. 
The coefficients of the expansion $C_{g,n}(\mu_1,\dots,\mu_n)$ 
of \eqref{Fgn Catalan} are
the generalized Catalan numbers of type $(g,n)$ 
that count the numbers
of \textbf{cell graphs} of genus $g$ and
$n$ labeled vertices of degrees
$(\mu_1,\dots,\mu_n)$
\cite{ OM3, OM4, DMSS, MP1998, MP2012}. A cell graph is the $1$-skeleton of
a cell decomposition of a connected oriented surface of
genus $g$ with $n$ labeled $0$-cells. Its \emph{dual graph} is
commonly called a \emph{ribbon graph} on which
faces are labeled. To avoid the difficulties of 
counting automorphisms of graphs, we impose that
no cyclic rotations of incident edges at any vertex are allowed,
as in Figure~\ref{01graph}. This explains the  denominator $\mu_1\cdots\mu_n$
of \eqref{Fgn Catalan} (see Remark~\ref{categorical}).
 For $(g,n) = (0,1)$, 
the unique vertex has to have an  even degree, and
$C_{0,1}(2m) = C_m$ is the $m$-th Catalan number.
The surprising properties of the functions $F_{g,n}$
are the following (see for example, \cite[Theorem 2.7]{OM4}).

\begin{thm}[\cite{CMS, OM4, OM5,DMSS, MP2012}]
\label{thmCatalan}
For the case of $2g=2+n>0$,
substitute each $x_i$ in \eqref{Fgn Catalan}  with
$
\begin{cases}
x_i = z_i + \frac{1}{z_i},\\
z_i= \frac{t_i + 1}{t_i-1},
\end{cases}
$
and write the result as $F_{g,n}(t_1,\dots,t_n)$ incorporating these substitutions.
Then for every $(g,n)$ in the range of $2g-2+n>0$, the following holds:
\begin{itemize}
\item $F_{g,n}(t_1,\dots,t_n)$ is a Laurent polynomial in 
the $t$-variables of  degree $6g-6+3n$.
\item $F_{g,n}(t_1,\dots,t_n) = F_{g,n}(1/t_1,\dots,1/t_n)$.
\item $F_{g,n}(1,\dots,1) = (-1)^n \rchi(\cM_{g,n})$.
\item The restriction to the highest degree terms of $F_{g,n}(t_1,\dots,t_n)$
is a homogeneous polynomial
\be
\label{WK}
 F_{g,n}^{\text{top}}(t_1,\dots,t_n)
=\frac{(-1)^n}{2^{2g-2+n}}
\sum_{\substack{d_1+\cdots+d_n\\
=3g-3+n}}\la\tau_{d_1}\cdots \tau_{d_n}
\ra_{g,n}
\prod_{i=1}^n (2d_i-1)!!
\left(\frac{t_i}{2}\right)^{2d_i+1}.
\ee
\end{itemize}
\end{thm}

Here, $\cM_{g,n}$ is the moduli space of smooth $n$-pointed
algebraic curves of genus $g$, $\Mbar_{g,n}$ its compactification, 
i.e., 
the Deligne-Mumford stack of
stable curves of finite type $(g,n)$, and
$$
\la\tau_{d_1}\cdots \tau_{d_n}\ra_{g,n} 
=\int_{\Mbar_{g,n}} c_1(\bL_1)^{d_1}\cdots c_1(\bL_n)^{d_n}
$$
is the intersection number of the first Chern classes of the tautological
line bundles $\bL_i\lrar \Mbar_{g,n}$   that is determined by the 
$i$-th marked point on stable curves (see \eqref{tautologicalL} below).

The contrast between the Picard-Fuchs equation \eqref{xzdiff} and
the irregular singular equation \eqref{qxz} is the \emph{quantization}
effect: the former only determines the original Catalan numbers, while the latter
has the information of all $(g,n)$-Catalan numbers $C_{g,n}(\mu_1,\dots,\mu_n)$. This effect is often 
observed in the \emph{Borel-Laplace} transform of 
differential equations. 
The passage from the spectral curve $\Sigma$ of \eqref{xz} to the quantum curve
\eqref{qxz} is the inverse operation of taking the {semi-classical limit}. 
As noted 
in \cite{OM2}, the conic $\Sigma$, a nonsingular plane
curve in $\bP^2$,  should be placed in the cotangent bundle $T^*\bP^1$
to consider its quantum curve.
The irregular singularity at $\infty\in\bP^1$ of \eqref{qxz}
comes from the fact that the embedding of $\Sigma$ into the compactified
cotangent bundle  $\overline{T^*\bP^1}$ of $\bP^1$, 
or the Hirzebruch surface $\bF^2$, has a singularity
at $\infty\in\bP^1$.
\begin{equation}
\label{spectral}
\xymatrix{
\Sigma \ar[dr]_{\pi}\ar[r]^{i} 
&{\overline{T^*\bP^1}}\ar[d]^{\pi}\ar[r]^\isom&\bF^2
\\
&\bP^1	}
\end{equation} 
While the Picard-Fuchs equation \eqref{xzdiff} is associated with the
the conic $\Sigma\subset \bP^2$ and the
ramified double-sheeted covering  $\pi: \Sigma \lrar \bP^1$,
 the irregular singular equation \eqref{qxz}
represents the symplectic geometry of $T^*\bP^1$ and the singularity of
$\Sigma$ in the compactified fiber of $\pi: \overline{T^*\bP^1}=\bF^2 \lrar \bP^1$
at $\infty\in\bP^1$. Precisely this geometric difference is reflected in the
\emph{quantization} process, from the $(0,1)$ invariants
in $z(x)$ of \eqref{xz} to a function $\Psi(x,\hbar)$ 
of \eqref{CatalanPsi} that has the information of invariants for all 
values of $(g,n)$.

Although we do not touch in these lectures, there is yet another coordinate
change from $(t_1,t_2,t_3,\dots)$ appearing in the Catalan case to the
time evolution parameters $(t_1,t_2,t_3,\dots)$  in 
the Lax equations \eqref{KP} defined below. 
Our use of the  same notations is due to the (confusing) conventions 
of different previous publications.  In reality,
these are not the same variables. After the choice of  the
right  coordinate transformation, 
\begin{equation}
\label{Catalantau}
\Psi(\mathbf{t},\hbar)=
\exp\left(
\sum_{g\ge 0, n>0}\frac{1}{n!}\hbar^{2g-2+n}
F_{g,n}(t_1,t_2,t_3,\dots)
\right)
\end{equation}
becomes a $\tau$-function for the KP equations. When we restrict
$F_{g,n}$ to the highest degree terms and apply the same 
coordinate change, this $\tau$-function then becomes that of the
KdV equations of Witten \cite{W1991} and Kontsevich \cite{K1992}.

\subsection{Theme $2$: When is a space a moduli space?}
\label{theme2}
A space, or a \textbf{manifold} is a democratic space. 
Every point of a manifold is 
treated equal. No matter where you go on a manifold,  no matter 
which point you choose, its neighborhood is \emph{isomorphic}
to that of any other point. There are no special points, or \emph{singular} 
points. No point dominates other points, and no point shows its individuality. 
Everywhere is the same and things work smoothly on a manifold.

A \textbf{moduli space} is different. On this space, every point has 
its \emph{unique} name. Each point is distinguished by its individuality. 
No points on a moduli space are the same. A neighborhood of each point is
different from place to place. There may  also be \emph{singularities}
on a moduli space. Each point has its own characteristic importance, yet it 
is a member of a  family.
Since every point is different, the moduli space itself
presents its miracle: it can exist, as a harmonious \emph{space}! 

We are not asking when a moduli space is a manifold. Our purpose is 
different. We are wondering when a given space 
is actually a moduli space. 

A simplest example we can consider is a vector space $\bC^n$. 
Can $\bC^n$ be a moduli space? Hitchin \cite{H1} told us: \emph{Yes}!
But for some special values of the dimension $n$. He starts with a smooth 
algebraic curve $C$ of genus $g>1$, and defines a vector space
\be
\label{HitchinBase}
B: = \bigoplus_{i=2}^r H^0\big(C,K_C^{\tensor i}\big)
\isom \bC^{(r^2-1)(g-1)}
\ee
consisting of \emph{multi differentials} on $C$. Here, $K_C=\Omega_C^1$ is
the canonical sheaf of $C$ consisting of holomorphic $1$-forms. This vector space is the 
\emph{moduli space of spectral curves} for $SL_r(\bC)$-Higgs bundles.
This example also explains the title of these lecture notes: 
\emph{In search of a hidden curve}. The vector space $B$ becomes a
moduli space as soon as we find a curve $C$, and the dimension $n$ 
factors as $n=(r^2-1)(g-1)$ for some $r>1$, which is associated with 
the group $SL_r(\bC)$.

As another simple example, let us start, again,  with the complex vector space 
$\bC^n$ of dimension $n>0$, but this time $n$ is arbitrary,
and we choose a set of 
$2n$ vectors in it linearly independent over $\bR$.  
These vectors generate a $\bZ$-submodule $\Gam\subset \bC^n$
of rank $2n$, a full-rank $\bZ$-\emph{lattice}. The quotient space
$T:=\bC^n/\Gam$ is a compact complex manifold whose underlying
space is the real $2n$-dimensional torus $(S^1)^{2n}$.
Now the question: When is this quotient space a \emph{moduli space}?

We give a linear coordinate system on $\bC^n$ and represent
the $2n$ generators of $\Gam$ by a set of $2n$ column vectors
$w_1,\dots, w_{2n}$ of size $n$. They form a complex matrix $\Pi$
of size $n\times 2n$. $GL_n(\bC)$ acts on $\Pi$ from the left,
representing the coordinate change. Since the column vectors of the
matrix $\Pi$ generate the $\bZ$-module $\Gam$, 
$SL_{2n}(\bZ)$ acts on $\Pi$ from the right, representing 
change of generators of the same lattice $\Gam$. 
Since $\{w_1,\dots, w_{2n}\}$ is linearly independent over $\bR$, 
there is a $\bC$-linearly independent subset consisting of
$n$ vectors. Thus we can use a left multiplication of $GL_n(\bC)$ and 
a right multiplication of $SL_{2n}(\bZ)$ on $\Pi$ to bring
it to the special shape  $[I\,|\,\Omega]$. The geometry of the 
quotient space $\bC^n/\Gam$ is encoded in this $n\times n$ matrix
$\Omega$, which we call the \emph{period matrix} of the
torus $T=\bC^n/\Gam$.

A classical result shows that if the period matrix satisfies the
\emph{Riemann period condition} $^t \Omega = \Omega$ and 
$Im(\Omega)>0$, i.e., $\Omega$ is symmetric and its imaginary part is
positive definite, then the torus admits a holomorphic embedding
$\bC^n/\Gam \subset \bP^N$ into a complex projective space
of a large dimension $N$. When it happens,  the quotient
$\bC^n/\Gam$ can be defined by a set of homogeneous polynomial equations 
in $N+1$ variables,  and it becomes an \emph{Abelian variety}.
The choice of the shape $\Pi=[I\,|\,\Omega]$ brings to 
the quotient a \emph{principal
polarization}. For those complex tori with periods satisfying the Riemann
period condition, let us use the notation $\mathbf{A} = \bC^n/\Gam$.

The space of all period matrices satisfying the Riemann period condition
is therefore a moduli space, known as the moduli space of Abelian varieties. 
But our question is at a different level: we are still asking, when is an Abelian 
variety $\mathbf{A}$ a moduli space by itself?

An immediate answer comes, again, 
from the geometry of smooth algebraic curves
over $\bC$. On such a curve $C$ of genus $g>0$, we have $g$ linearly
independent holomorphic $1$-forms. Thus
$H^0(C,K_C) = \bC ^g$, which is the  space of holomorphic $1$-forms. The first 
homology group of $C$ is $H_1(C,\bZ) = Z^{2g}$. Let us choose a $\bC$-basis
$\{\omega_1,\dots,\omega_g\}$ for $H^0(C,K_C)$ and homology 
generators $\{\gam_1,\dots,\gam_{2g}\}$. We arrange it so that the
homology basis satisfies the \emph{symplectic} intersection property
\be
\label{homology}
\la \gam_i, \gam_j\ra = \begin{cases}
1\qquad j=i+g,\;\;\; i=1, \dots, g\\
0 \qquad \text{otherwise}.
\end{cases}
\ee
Now define a complex $g\times 2g$ matrix 
$$
\Pi = \begin{bmatrix}
\oint_{\gam_1}\omega_1 &\cdots& \oint_{\gam_{2g}}\omega_1\\
\vdots &\ddots&\vdots\\ 
\oint_{\gam_1}\omega_g &\cdots& \oint_{\gam_{2g}}\omega_g
\end{bmatrix}.
$$
Using the actions of $GL_g(\bC)$ from the left and $Sp_{2g}(\bZ)$
from the right, we can rearrange $\Pi$ in the shape $\Pi = [I\,|\,\Omega]$
again. Riemann discovered that $\Omega$ satisfies 
$^t \Omega = \Omega$ and 
$Im(\Omega)>0$. We can thus construct an Abelian variety
$$
\Jac(C) := H^0(C, K_C)\big/H_1(C,\bZ),
$$
for which $\Omega$ is the period matrix.

Therefore, a complex torus $T=\bC^g/\Gam$ is a moduli space
if its period matrix $\Omega$ comes from the integration \emph{periods}
of holomorphic $1$-forms on a smooth algebraic curve $C$
along its homology basis, because
$$
\Jac(C) \isom H^1(C,\cO_C)\big/H^1(C,\bZ) =: \Pic^0(C)
$$ 
is the moduli space
of holomorphic line bundles on $C$ of degree $0$.
For each fixed $g>1$, there are $3g-3$ dimensional family of Jacobian 
varieties, while Abelian varieties form a family of $g(g+1)/2$ dimensions.
Therefore, Jacobians are very special among Abelian varieties.

So one can ask a question, which Riemann himself did: \emph{When is 
a principally polarized Abelian variety a Jacobian variety}? 
And this question leads us  again to:  \emph{In search of a hidden curve}.

\subsection{Theme $3$: Finding a curve from a period matrix}
The holomorphic embedding of the Abelian variety $\mathbf{A} = \bC^g/\Gam$
into a projective space is constructed by the \emph{Riemann theta function}
\be
\label{theta}
\vartheta(z,\Omega) = \sum_{n\in \bZ^n} \exp
\left(2\pi i \la n,z\ra + \pi i \la n,\Omega n\ra\right), \qquad z\in \bC^n,
\ee
where $\la\;\;,\;\;\ra$ is the standard symmetric form on $\bC^n$.
The Riemann period condition makes
the above infinite series absolutely convergent on $\bC^n$ and satisfy
quasi-periodic conditions. 

Over a four decade-old work  \cite{M1984}
shows the following:

\begin{thm}[\cite{M1984}]
A principally polarized
Abelian variety of period $\Omega$ is a Jacobian variety if and only if
the Riemann theta function \eqref{theta} satisfies the 
nonlinear completely integrable system of Kadomtsev-Petviashvili (KP) equations.
\end{thm}

The nonlinear integrable system of KP equations is compactly  formulated 
using the formalism of Peter Lax as follows:
\be
\label{KP}
\frac{\partial L}{\partial t_n} = [B_n, L], \qquad B_n = (L^n)_+,\quad n=1, 2, 3, 
\dots,
\ee
where the \emph{Lax operator}
\be
\label{Lax}
L=\partial + \sum_{i=1}^\infty u_{i+1}(x; \mathbf{t})\partial^{-i}, \qquad 
\partial = \frac{\partial}{\partial x}, \quad \mathbf{t} = (t_1, t_2, t_3,\dots)
\ee
is a normalized (i.e., there is no constant term in $L$) 
first order pseudodifferential operator in $x$ 
depending on  an 
infinitely many parameters $(t_1, t_2, t_3,\dots)$. The notation $(L^n)_+$ indicates the
differential operator part of $L^n$, i.e., suppressing all 
negative powers of $\partial$. If we write $u = u_2, y = t_2$ and $t = t_3$, 
then the first equation from the system \eqref{KP} is of the form
$$
3u_{yy} = \big(4 u_t -u_{xxx}-12u u_x\big)_x,
$$
where the subscript indicates partial differentiation. This equation
is discovered by Kadomtsev and Petviashlivi in plasma physics. In particular,
a solution to the KdV equation
\be
\label{KdV}
u_t = \frac{1}{4} u_{xxx}+ 3u u_x
\ee
which does not depend on $y$ gives a solution to the KP equation.
A relation between the KdV equation \eqref{KdV} and 
algebraic curves is easily seen by considering a \emph{traveling}
wave solution with speed $-c$ constructed from the Weierstra\ss \;
elliptic function $\wp(z)$: 
\be
\label{p-solution}
u(x,t) = -\wp(x+ct) + \frac{1}{3} c.
\ee

The starting point of the work \cite{M1984} is 
Issai Schur's 1905 theorem which says  that if two differential operators
$P$ and $Q$ in $x$ satisfy the commutativity relation 
$[P,L]=[Q,L]=0$ with a
given pseudodifferential operator $L$, then $P$ and $Q$
automatically commute: $[P,Q]=0$. Now, suppose 
a solution 
$L$ of the KP system depends only on finitely many time variables so that
$L$ deforms  only along a $g$-dimensional direction of the
parameter space $\mathbf{t}\in \bC^\infty$. Then for all other directions,
the KP system \eqref{KP} produces infinitely many differential operators $P$, 
which are linear combinations of $B_n$'s, that commute with $L$.
Denote by $A_L$ the  algebra generated by all these 
$P$'s over $\bC$. Schur's argument  shows that $A_L$ is a commutative 
ring, and by Euclidean algorithm, we can show that the transcendence
degree of $A_L$ over $\bC$ is $1$. Give a natural grading in $A_L$ by
the order of differential operators, and  define 
\be
\label{specL}
C = \Proj(gr(A_L)) = \overline{\Spec(A_L)}= \Spec(A_L)\sqcup \{\infty\}.
\ee
It is an algebraic curve! It is called the \textbf{spectral curve} associated with the 
solution $L=L(x; \mathbf{t})$. Moreover, it is proved 
in \cite{M1984} that $C$ is of (arithmetic) genus $g$, 
the ring of differential operators
$\cD$
in $x$ is an $A_L$-module determining
 a line bundle $\cL$ over $C$, and that the KP equations define a linear
flow on the cohomology $H^1(C,\cO_C)\isom \bC^g$.
A geometric interpretation of the KP equations in this context is that 
they form a
 compatible linear deformation family of the line 
bundle $\cL$ on $C$ along $H^1(C,\cO_C)$
(cf. \cite{LiM1997, M1984,M1994}).

A generalization of the formula \eqref{p-solution} gives the
Lax operator $L$ from a Riemann theta function $\vartheta(z,\Omega)$. 
If is solves the KP system, then it  deforms only along a finite dimensional
direction. Therefore, the argument above shows that the period matrix
$\Omega$ must come from an algebraic curve, based on the unique solvability
of the initial value problem of the evolution equation \eqref{KP} established
in \cite{M1984a, M1988}. The analysis of KP equations
in these earlier papers were done in the formal power series level, 
which was sufficient for the purpose of identifying the finite-dimensional
solutions (a precise argument can be found in \cite{LiM1997}). 
Recently, a powerful theorem is established by 
Magnot and Reyes \cite{MagnotR} that deals with differentiable 
(non-formal) solutions
of KP equations. 

This is a success story of finding a hidden curve in the jungle.

\subsection{Theme $4$: The $\tau$-functions}

In the mid 1970s, Ryogo Hirota was developing a novel mechanism to
calculate exact \emph{soliton} solutions to nonlinear PDEs
such as  the KdV equation, introducing
the \emph{Hirota bilinear differentiation}
$$
D_x \; f(x)\bullet g(x): = \left.\frac{\partial}{\partial s} \big(f(x+s)g(x-s)\big)
\right|_{s=0}
$$
and a new dependent variable $\tau$.
For a polynomial $P(D)$ in $D$, the Hirota bilinear differentiation is
defined to be
$$
P(D_x) \; f(x)\bullet g(x): = \left. P\left(\frac{\partial}{\partial s}\right) \big(f(x+s)g(x-s)\big)
\right|_{s=0}.
$$
Then the KdV equation \eqref{KdV}  becomes
$$
\left(D_x^4 - 4D_xD_t\right)\; \tau\bullet \tau 
$$
with the substitution 
\be
\label{utau}
u(x,t) = \frac{\partial^2}{\partial x^2} \; \log \tau(x,t).
\ee
One can easily check the effectiveness of Hirota's method by 
calculating examples, such as
$$
\tau(x,t) = 1+ c\, e^{2(\lam x + \lam^3 t)}\;,
$$
which gives a $2$-parameter family of $1$-soliton solutions to the KdV equation via \eqref{utau}.

In the late 1970s, Mikio Sato noticed the algebraic mechanism behind
Hirota's method, and discovered  the \emph{Sato Grassmannian}. He then 
identified the Hirota's dependent variable $\tau$, now known as
\emph{Sato's} $\tau$-function,  as the canonical section 
of the determinant line bundle of the Sato Grassmannian.
As a set, the Sato Grassmannian of index $\mu$ is modeled over
$$
Gr(\mu) = \left\{ W\subset \bC((z))\;\left|\; \gam: W\lrar \bC[z^{-1}]
\isom \bC((z))\big/\big(\bC[[z]]\cdot z\big) \text{ is Fredholm of index } \mu
\right\}\right. .
$$
The time evolution determined by the Lax formalism \eqref{KP}
becomes, in this language, the action on the unipotent element
$$
\exp\left(\sum_{n=1}^\infty t_n \Lam^n\right)
\in GL(V)
$$
on the Sato Grassmannian, where $V=\bC((z))$ and
$\Lam\in \mathfrak{gl}_\infty$ is the
maximal nilpotent element corresponding to  the
multiplication of $z^{-1}$ on $V$
(see for example, \cite[Chapter 7]{M1994}). 
The $\tau$-function records this action through the vantage point of 
the determinant line bundle. This also explains why so many
soliton equations are exactly solved in terms of \emph{determinant}.
Since the initial value problem of the KP equations is uniquely
solvable (\cite{M1984a, M1988}), the space of solutions to the 
KP equations is identified with $Gr(\mu=0)$. 

We can immediately appreciate the similarity of \eqref{utau} and
the differential relation between elliptic functions and theta functions. 
Indeed, the Riemann theta functions are $\tau$-functions when the
period matrix comes from a Jacobian. Of course these $\tau$-functions are
very special ones.  They correspond to points $W\in Gr(0)$ 
determined by
$$
W = H^0\big(C, \cL(*p)\big)\subset \bC((z)),
$$
the space of meromorphic sections of $\cL$ with arbitrary poles at
$p$,
where $p\in C$ is a non-singular point of $C$, $z$ is a local 
parameter of $C$ around $p$, and $\cL$ is a line bundle on $C$ of 
degree $g(C)-1$. We then have (see \cite[Lemma 3.7]{M1984a})
$$
\begin{cases}
H^0(C,\cL) \isom \Ker\; \gam\\
H^1(C, \cL) \isom \Coker\; \gam.
\end{cases}
$$

Generically a solution to \eqref{KP} 
produces an infinite-dimensional orbit. 
Among infinite-dimensional orbits, there are many solutions identified
in geometry. One is associated with the Catalan numbers mentioned earlier.
We  review another example, coming from Hurwitz theory, in the subsequent section.

\section{The spectral curve for Hurwitz numbers}
\label{sect:Hurwitz}


\subsection{The Laplace transform}
\label{subsect:Laplace}
The concept of spectral curves as the $B$-model mirror to an $A$-model
counting problem appeared 
in the \emph{remodeling
conjecture}
for Gromov-Witten invariants of toric Calabi-Yau threefolds. This idea has been developed 
 by Mari\~no \cite{Mar}, 
Bouchard-Klemm-Mari\~no-Pasquetti
\cite{BKMP}, and Bouchard-Mari\~no
\cite{BM}, based on the theory of topological recursion 
formulas of Eynard
and Orantin \cite{EO1}.
The remodeling conjecture 
states that the open and closed 
Gromov-Witten invariants of a toric Calabi-Yau
threefold can be  captured by the Eynard-Orantin
topological recursion as a $B$-model that is 
constructed on the mirror curve. This conjecture has been
completely solved, even beyond the original scope of the conjectures,
 by Fang-Liu-Zong
\cite{FLZ}.

Let us now examine the idea
that \textbf{mirror symmetry is the Laplace transform}
in some cases,
by going through
the concrete example of simple Hurwitz numbers \cite{H}.  
Thus our  question is the following:

\begin{quest}
What is the mirror dual of 
simple Hurwitz numbers?
\end{quest}

\noindent
This is the same question we asked for the
Catalan numbers. Indeed, 
around 2010, and before the author started to 
work on the Catalan number case, Boris Dubrovin and he
 had the following
conversation.
\medskip

\noindent
Dubrovin: \textit{Good to see you, Motohico!} 

\noindent
M: \textit{Hi Boris, good to see you, too!
At last I think I am coming close to 
understanding what mirror symmetry is.}

\noindent
Dubrovin: \textit{Oh, yes? 
All right, what do you
think about mirror symmetry?}

\noindent
M: \textit{It is the Laplace transform!}

\noindent
Dubrovin: \textit{Do you think so, too? But I have
been saying so for the last 15 years!} (He was referring to \cite{D,DZ}.)

\noindent
M: \textit{Oh, have you? But I'm not talking about
the Fourier-Mukai transform or the $T$-duality. It's the 
Laplace transform in the classical complex analysis sense.}

\noindent
Dubrovin: \textit{Of course I mean the same way.}

\noindent
M: \textit{All right, then let's
check if we have the same understanding. 
Question: What is the mirror symmetric dual of a point?}

\noindent
Dubrovin: \textit{It is the Lax operator of 
the KdV equations that was 
identified by Kontsevich.}

\noindent
M: \textit{The Lax operator is the mirror symmetric dual of a point!? Hmm. Ah! I 
think you mean $x=y^2$, don't you?}

\noindent
Dubrovin: \textit{What? Wait a second. Oh, yes, indeed! Now it is my turn
to ask you a question. 
What is the mirror symmetric dual of 
the Weil-Petersson volume of the moduli space
of bordered hyperbolic surfaces discovered by
Mirzakhani?}

\noindent
M: \textit{The sine function
$x=\sin y$.}

\noindent
Dubrovin: \textit{Exactly!}

\noindent
M: \textit{How about simple Hurwitz numbers?}

\noindent
Dubrovin: \textit{Oh, this one you know well: The Lambert 
curve!}

\noindent
M: \textit{Yes, it is. And in all these cases, the mirror
symmetry is the Laplace transform.}

\noindent
Dubrovin: \textit{Of course it is.}

\noindent
M: \textit{A, ha! Then we seem to have
the same understanding of the mirror symmetry.}

\noindent
Dubrovin: \textit{Apparently we do!}

\medskip

\noindent
This conversation took place in the lobby of MSRI, Berkeley, after the
end of the day. 
We shook hands tightly and happily, as we started to depart. 
We then noticed that there was a young mathematician listening to our conversation. 
At the end, he shouted: 
\medskip

\noindent
Bystander: \textit{What, what, what? With this exchange, are
you saying that you understood one  another? Unbelievable!}

\medskip

In the spirit of the above dialogue, the answer
to our  question should be:

\begin{thm}[\cite{BM, EMS, MZ}]
The mirror dual to the  simple Hurwitz numbers
is the Lambert curve
$x=y e^{-y}$.
\end{thm}

A  mathematical picture  has emerged in 
the last two decades since the discoveries
of Eynard-Orantin \cite{EO1},
Mari\~no \cite{Mar}, 
 Bouchard-Klemm-Mari\~no-Pasquetti
\cite{BKMP}
and Bouchard-Mari\~no \cite{BM} in physics, and 
 many mathematical efforts including
\cite{BHLM,CMS,  DMSS, DMNPS, EMS, FLZ,LM, MP2012,
MS2008, MS2015, MZ, N1,N3, N4,NS1,NS2,SW}. As a working hypothesis, we phrase it
in the form of a principle.

\begin{prin}
\label{prin}
For a number of interesting
cases, we have  the following general structure.
\begin{itemize}
\item
On the   $A$-model side of  topological string
theory, we have a class of mathematical problems
arising from combinatorics, geometry, and topology. 
The common feature of these problems is that they
are somehow related to a lattice point counting
of a collection of polytopes.
\item
On the $B$-model side, we have a universal theory
due to Eynard and Orantin \cite{EO1}. It is a
framework of the recursion formula of a particular
kind that is based on a \textbf{spectral curve} and
 two analytic  functions (often with singularities) on it that
immerse  the curve
into a complex symplectic surface as a complex Lagrangian.
\item
The  passage from  $A$-model
to  $B$-model, i.e., the mirror symmetry operation
of the class of problems that we are concerned,
is given by the \textbf{Laplace transform}.  The
spectral curve on the $B$-model side is
defined  as the \textbf{Riemann surface}
of the Laplace transform, which means that it is the
domain of holomorphy of the Laplace transformed
function.
\end{itemize}
\end{prin}

There are many examples of 
mathematical
problems that fall in to this principle.
Among them is the theory of simple
Hurwitz numbers \cite{EMS,MZ} that we are going to present now. 
Besides Hurwitz numbers, numerous mathematical results have been 
established. They include
counting of Grothendieck's dessins d'enfants
\cite{CMS, MP2012, N1,NS1},
higher genus Catalan numbers \cite{OM4,DMSS},  
orbifold Hurwitz numbers \cite{BHLM}, double Hurwitz numbers
and higher spin structures \cite{MSS},
and the stationary Gromov-Witten invariants of $\bP^1$
\cite{DMNPS, NS2}. One of the most important results in 
the Gromov-Witten setting is 
the  remodeling conjecture 
\cite{BKMP} 
and its solution \cite{FLZ} mentioned above.
A new direction was suggested in Kontsevich-Soibelman \cite{KS2017}.
It is now impossible to list all research in this direction,
simply called \emph{topological recursion},
so we refer to recent papers by many
authors and the papers they cite: Andersen, Borot, Bouchard, Chidambaram,
Do, Dunin-Barkowski, Eynard, 
Garcia-Failde, Iwaki,
Lewa\'nski, Norbury, Orantin, Osuga, Shadrin, ...  A comprehensive
and the most recent
introduction is given by Bouchard \cite{B2024}, in which one can see
how the focus of the topological recursion community has changed into
new developments in the last decade.

The study of simple Hurwitz numbers
is one of the earliest results in this direction and  exhibits 
all important ingredients found in the later 
papers.

\subsection{Simple Hurwitz numbers}
In this subsection, we define simple Hurwitz 
numbers, and a combinatorial 
equation that they satisfy, the \emph{cut-and-join
equation} \cite{GJ, V},  is proved.
In the next subsection, the Laplace transform of the
Hurwitz numbers is presented. The Laplace
transformed holomorphic functions live on 
the mirror $B$-side of the model,
according to Princeple~\ref{prin}. 
The Lambert curve is defined as the
domain of holomorphy of these holomorphic
functions.
Then in the following subsections, we give the Laplace transform
of the cut-and-join equation. The result is a
simple \textbf{polynomial recursion formula}
and is equivalent to the Eynard-Orantin topological recursion
for the Lambert curve. We also give a 
straightforward and simple
derivation \cite{MZ} of the Witten-Kontsevich theorem
on the $\psi$-class intersection numbers
\cite{DVV, K1992, W1991},
and the $\lambda_g$-formula of Faber and Pandharipande
\cite{FP1,FP2}, using the recursion formula we establish
in \cite{MZ}.

 In all examples of Principle~\ref{prin}
we know so far, the $A$-model side always has a series of 
combinatorial equations that should uniquely determine 
the quantities in question, at least
theoretically. But in practice solving these
equations is quite complicated. 
As we develop in these lectures, the Laplace
transform changes these equations to a
\emph{topological recursion} in the $B$-model 
side, which is an inductive formula based 
 on the absolute value of the Euler 
characteristic of  punctured surfaces.

\begin{rem}
Recently there have been  totally unexpected spectacular developments 
in the topology
of moduli spaces $\cM_g, \cM_{g,n}$, and $\Mbar_{g,n}$ 
(see for example, \cite{CLP}).  
From the point of view of \emph{In Search of a Hidden Curve}, 
and from the perspective of Norbury discovering the Norbury classes
in $H^*(\Mbar_{g,n}, \bQ)$ using topological recursion
\cite{N4}, which
were later identified as the Euler classes associated with a 
supersymmetric generalization of the work of Mirzakhani
by Stanford-Witten \cite{SW}, we wonder if there could be a
\emph{spectral curve} behind the  recent
developments. 
\end{rem}

A \emph{simple Hurwitz number} represents the number
of a particular type of meromorphic functions
defined on an algebraic curve $C$ of genus $g$.
Let $\mu=(\mu_1,\dots,\mu_\ell)\in\bZ_+ ^\ell$
be a \emph{partition} of a positive integer
$d$ of length $\ell$. This means
that $|\mu|\overset{\text{def}}{=}
\mu_1+\cdots+\mu_\ell = d$. Instead of ordering 
 parts of $\mu$ in the decreasing order,
 we consider them as a vector consisting of
 $\ell$ positive integers. 
By a \emph{Hurwitz covering} of 
type $(g,\mu)$ we mean a
meromorphic function $h:C\rightarrow \bC$
that has $\ell$ labelled poles $\{x_i,\dots,x_\ell\}$,
such that the pole order of $h$ at $x_i$ is 
 $\mu_i$ for every $i=1,\dots,\ell$,
 and that except for these poles, the holomorphic
 $1$-form $dh$ has simple zeros on
 $C\setminus \{x_i,\dots,x_\ell\}$ with 
 distinct critical values of $h$.
 A meromorphic function of $C$ is a holomorphic
 map of $C$ onto $\bP^1$. 
 In algebraic geometry, the situation described above
 is summarized as follows: $h:C\rightarrow \bP^1$ is
 a \emph{ramified covering} of $\bP^1$, 
 simply ramified except for $\infty\in\bP^1$.
 We identify two Hurwitz coverings $h_1:C_1\rightarrow
 \bP^1$ with poles at
 $\{x_1,\dots,x_\ell\}$ and $h_2:C_2\rightarrow
 \bP^1$ with poles at 
 $\{y_1,\dots,y_\ell\}$
 if there is a biholomorphic map 
 $\phi:C_1\overset{\sim}{\rightarrow} C_2$
 such that $\phi(x_i) = y_i$, $i=1,\dots,\ell$, and 
 \begin{equation*}
		\xymatrix{
		C_1\ar[dr]_{h_1}  
		\ar[rr]^\phi _\sim &&C_2\ar[dl]^{h_2}\\
		&\bP ^1&.
		}
\end{equation*} 
When $C_1=C_2$, $x_i=y_i$,
 and $h_1=h_2=h$, such a
biholomorphic map $\phi$ is called an 
\emph{automorphism}
of a Hurwitz covering $h$.
Since biholomorphic Hurwitz coverings are
identified,  following the stack theoretic principle, we  
count the number of  Hurwitz coverings with the 
automorphism factor $1/|\Aut(h)|$.
And when the above $\phi:C_1{\rightarrow} C_2$
is merely a homeomorphism,  we say 
$h_1$ and $h_2$ have the same 
\emph{topological type}.

We are calling a meromorphic function 
a \emph{covering}. This is because if we
remove the critical values of $h$ (including $\infty$) from
$\bP^1$, then on this open set $h$ becomes
 a topological  covering. More precisely,
 let $B=\{z_1,\dots,z_r,\infty\}$ denote the
set of distinct critical values of $h$. 
Then 
$$h:C\setminus h^{-1}(B)
\longrightarrow \bP^1\setminus B$$ 
is a topological covering
of degree $d$. Each $z_k\in B$
 is a \emph{branched
point} of $h$, and a critical point (i.e., a zero of 
$dh$) on $C$ is called a \emph{ramification point}
of $h$. Since $dh$ has only simple zeros with
distinct critical values of $h$, the number of
ramification points of $h$, except for 
the poles, is equal to the number of branched points,
which we denote by $r$. Therefore, $h^{-1}(z_k)$
consists of $d-1$ points. Then by comparing
the Euler characteristic of the covering space and
its base space, we obtain
$$
d(2-r-1)=
d \cdot \rchi(\bP^1 \setminus B)
=\rchi(C\setminus h^{-1}(B))
\\
= 2-2g-r(d-1)-\ell.
$$
Thus we establish the  \textbf{Riemann-Hurwitz formula}
\begin{equation}
\label{eq:RH}
r=2g-2+|\mu|+\ell.
\end{equation}
What we wish to enumerate is:

\begin{Def}
The simple Hurwitz number of type $(g,\mu)$
for $g\ge 0$ and $\mu\in\bZ_+ ^\ell$
that we consider in these lectures is 
\begin{equation}
\label{eq:Hgmu}
H_g(\mu) = \frac{1}{r(g,\mu)!}\; \sum_{[h] \text{ type }
(g,\mu)}
\frac{1}{|\Aut(h)|},
\end{equation}
where the sum runs all topological equivalence classes $[h]$
of Hurwitz coverings $h$ of type $(g,\mu)$. 
Here,
$$
r=r(g,\mu) = 2g-2+(\mu_1+\cdots+\mu_\ell) + \ell
$$
is the number of simple ramification points
of $h$.
\end{Def}

\begin{rem}
Our definition of simple Hurwitz numbers
differs from the standard definition
by two automorphism factors. The quantity 
$h_{g,\mu}$ of \cite{ELSV} and $H_g(\mu)$
are related by
$$
H_g(\mu) = \frac{|\Aut(\mu)|}{r!} \; h_{g,\mu},
$$
where $\Aut(\mu)$ is the  group 
of permutations that permutes equal parts
of $\mu$ considered as a partition. 
This is due to the convention
that we label the poles of $h$ and consider
$\mu\in\bZ_+ ^\ell$ as a vector, 
while we do not label simple ramification points.
\end{rem}

\begin{rem}
Note that interchanging the entries of $\mu$ 
 means permutation of the label of
the poles $\{x_1,\dots,x_\ell\}$ of $h$.
Thus it does not affect the count of simple
Hurwitz numbers. 
Therefore, as a function in $\mu\in\bZ_+ ^\ell$,
$H_g(\mu)$ is a symmetric function.
\end{rem}

Simple Hurwitz numbers satisfy a 
simple equation, known as the \emph{cut-and-join
equation} \cite{GJ, V}. Here we give it in the
format used in \cite{MZ}.

\begin{prop}[Cut-and-join equation \cite{MZ}]
Simple Hurwitz numbers satisfy
\begin{multline}
\label{eq:caj}
r(g,\mu) H_g(\mu)
=
\sum_{i<j} (\mu_i+\mu_j)H_g(\mu(\hat{i},\hat{j})
,\mu_i+\mu_j)
\\
+
\half \sum_{i=1} ^\ell \sum_{\alpha+\beta=\mu_i}
\alpha\beta
\left[
H_{g-1}(\mu(\hat{i}),\alpha,\beta)
+
\sum_{\substack{g_1+g_2=g
\\
I\sqcup J=\mu(\hat{i})}}
H_{g_1}(I,\alpha) H_{g_2}(J,\beta)
\right].
\end{multline}
Here we use the following notations.
\begin{itemize}
\item
$\mu(\hat{i})$ is the vector of $\ell -1$
entries obtained by deleting the $i$-th entry
$\mu_i$.
\item
$(\mu(\hat{i}),\alpha,\beta)$ is the 
 vector of $\ell +1$
entries obtained by appending two new
entries $\alpha$ and $\beta$ to $\mu(\hat{i})$.
\item
$\mu(\hat{i},\hat{j})$ is the vector of $\ell -2$
entries obtained by deleting the $i$-th and the $j$-th 
entries
$\mu_i$ and $\mu_j$.
\item
$(\mu(\hat{i},\hat{j})
,\mu_i+\mu_j)$ is the 
 vector of $\ell -1$
entries obtained by appending a new
entry $\mu_i+\mu_j$ to $\mu(\hat{i},\hat{j})$.
\end{itemize}
The final sum is over all partitions of $g$
into non-negative integers $g_1$ and $g_2$,
and a disjoint union decomposition (or a set partition) of 
entries of $\mu(\hat{i})$ as a set, allowing 
the empty set.
\end{prop}

\begin{rem}
Since $H_g(\mu)$ is a symmetric function,
the way we append a new entry to a vector does not
affect the function value.
\end{rem}

The idea to prove the formula is reducing the number
$r$ of simple ramification points. Note that
$$
h:C\setminus h^{-1}(B) \lrar
\bP^1\setminus B
$$
is a topological covering of degree $d$. 
Therefore, it is obtained by a representation
$$
\rho:\pi_1(\bP^1\setminus B)
\lrar S_d,
$$
where $S_d$ is the permutation group of $d$ letters.
The covering space $X_\rho$ of $\bP^1\setminus B$
is obtained by the quotient construction
$$
X_\rho = 
\widetilde{X}\times_{\pi_1(\bP^1\setminus B)}
[d],
$$
where 
$\widetilde{X}$ is the universal covering space
of $\bP^1\setminus B$, and $[d]=\{1,2,\dots,d\}$ is
the index set on which $\pi_1(\bP^1\setminus B)$
acts via the representation $\rho$.

To make $X_\rho$ a Hurwitz covering, we need
to specify the monodromy of the representation
at each branch point of $B$.
Let $\{\gam_1,\dots,\gam_r,\gam_\infty\}$
denote the collection of non-intersecting loops 
on $\bP^1$, starting from $0\in\bP^1$, 
rotating around $z_k$ counter-clockwise, and 
coming back to $0$, for each $k = 1,\dots,r$.
The last loop $\gam_\infty$
does the same for  $\infty\in\bP^1$.
Since $\bP^1$ is simply connected, we have
$$
\pi_1(\bP^1\setminus B) = 
\la \gam_1,\dots,\gam_r,\gam_\infty|
\gam_1\cdots\gam_r\cdot\gam_\infty=1\ra.
$$
Since $X_\rho$ must have $r$ simple ramification
points over $\{z_1,\dots,z_r\}$, the monodromy
at $z_k$ is given by a transposition
$$
\rho(\gam_k) = (a_kb_k)\in S_d,
$$
where $a_k,b_k\in [d]$ and all other indices are fixed
by $\rho(\gam_k)$. To impose the condition on
poles $\{x_1,\dots,x_\ell\}$ of $h$, we need
$$
\rho(\gam_\infty) = c_1 c_2\cdots c_\ell,
$$
where $c_1,\dots,c_\ell$ are disjoint cycles of
$S_d$ of length $\mu_1,\dots,\mu_\ell$, respectively.

We want to reduce the number $r$ by one. To do so,
we simply merge $z_r$ with $\infty$. The monodromy
at $\infty$ then changes from $c_1 c_2\cdots c_\ell$
to $(ab)\cdot c_1 c_2\cdots c_\ell$, where
$(ab)=(a_rb_r)$ is the transposition corresponding to
$\gamma_r$. There are two cases we have now:
\begin{enumerate}
\item Join case: 
$a$ and $b$ belong to two disjoint cycles,
say $a\in c_i$ and $b\in c_j$;
\item Cut case: 
both $a$ and $b$ belong to the same cycle,
say $c_i$.
\end{enumerate}
An elementary computation shows that for Case (1), the product
$(ab)c_ic_j$ is a single cycle of length $\mu_i+\mu_j$.
For the second case, the result depends on how far
$a$ and $b$ are apart in cycle $c_i$. If $b$ appears
$\a$ slots after $a$ with respect to the 
cyclic ordering, then 
\begin{equation}
\label{eq:cut}
(ab)c_i = c_\a c_\b,
\end{equation}
where $\a+\b = \mu_i$. The cycles
$c_\a$ and $c_\b$ are disjoint  of length
$\a$ and $\b$, respectively, and  $a\in c_\a$,
and $b\in c_\b$. Note that everything is symmetric
with respect to interchanging $a$ and $b$.
A couple of simple examples makes sense here.
\begin{align*}
\text{Join Case:}\quad &(12)(3517)(46829) = (468173529)
\\
\text{Cut Case:}\quad &(12)(351746829) = (3529)(17468).
\end{align*}

With this preparation, we can now give the proof
of (\ref{eq:caj}).
The right-hand side of (\ref{eq:caj}) represents 
the set of all monodromy representations 
obtained by merging one of the branch points
$z_k$ with $\infty$. The factor $r$ on the
left-hand side represents 
the choice of $z_k$.

\begin{proof} Since we are reducing $r=2g-2+d+\ell$ by one
without changing $d$, there are
three different ways of reduction:
\begin{align}
\label{eq:join}
(g,\ell)&\longmapsto (g,\ell-1)
\\
\label{eq:cut1}
(g,\ell)&\longmapsto (g-1,\ell+1)
\\
\label{eq:cut2}
(g,\ell)&\longmapsto (g_1,\ell_1+1)+(g_2,\ell_2+1),
\text{ where }
\begin{cases}
g_1+g_2=g
\\
\ell_1+\ell_2 = \ell -1.
\end{cases}
\end{align}

\begin{itemize}
\item
The first reduction (\ref{eq:join}) is exactly the
first line of the right-hand side
of (\ref{eq:caj}), which corresponds to
the join case. Two cycles of length 
$\mu_i$ and $\mu_j$ are joined to form
a longer cycle of length $\mu_i+\mu_j$. 
Note that the number $a$ has to be recorded somewhere
in this long cycle. This explains the factor
$\mu_i+\mu_j$. Then the number $b$ is automatically
recorded,
because it is  the entry appearing exactly $\mu_i$
slots after $a$ in this 
long cycle.

\item
The second line of the right-hand side 
of (\ref{eq:caj}) represents
the cut cases. In (\ref{eq:cut}), we have
$\a$ choices for $a$ and $\b$ choices for $b$. 
The symmetry of interchanging $a$ and $b$ 
explains the factor $\half$.
The first term of the second line of 
(\ref{eq:caj}) corresponds to (\ref{eq:cut1}).

\item
The second term of the second line
corresponde to (\ref{eq:cut2}).
Note that in this situation, merging a branched point
with $\infty$ breaks the connectivity of 
the Hurwitz covering. We have two ramified 
coverings $h_1:C_1\rar \bP^1$ of 
degree $d_1$ and genus $g_1$
with $\ell_1+1$ poles, and 
$h_2:C_2\rar \bP^1$ of degree $d_2$
and genus $g_2$ with
$\ell_2+1$ poles. If we denote by $r_i$ the
number of simple ramifications points of 
$h_i$ for $i=1,2$, then we have
$$
\begin{matrix}
&r_1 &= &2g_1-2+d_1+\ell_1+1\\
+)&r_2&= &2g_2-2+d_2+\ell_2+1\\
\hline
&r-1 &= &\;2\,g-2\,+\,d\,+\,\ell\,-\,1.
\end{matrix}
$$
\end{itemize}
This completes the proof of (\ref{eq:caj}).
\end{proof}

\begin{rem}
The reduction of the number $r$ of simple ramification 
points by one is exactly reflecting the 
reduction of the Euler characteristic of 
the punctured surface $X_\rho = C\setminus h^{-1}(B)$
appearing in our consideration
by one. Since we do not change
the degree $d$ of the covering, the reduction of $r$ 
is simply reducing $2g-2+\ell$ by one.
\end{rem}

\subsection{The Laplace transform of the
simple Hurwitz numbers}
\label{Fugue2}

Let us now compute the Laplace transform
of the simple Hurwitz number $H_g(\mu)$,
considered as a function in $\mu\in\bZ_+ ^\ell$.
According to Principle~\ref{prin},
the result should give us  the mirror dual 
of simple Hurwitz numbers.
We explain where the \textbf{Lambert curve}
$x=y e^{1-y}$ comes from, and its
essential role in
computing the Laplace transform.
The most surprising feature is that the result
of \textbf{the Laplace transform of $H_g(\mu)$ is
a polynomial} if $2g-2+\ell >0$. 
This polynomiality produces powerful
consequences, which is the main 
subject of the following subsections.

The most important reason for our  interest
in Hurwitz numbers in this summer school on \emph{complex Lagrangians}
 lies in the  
theorem due to Ekedahl, Lando, Shapiro and 
Vainshtein \cite{ELSV} that relates the simple Hurwitz numbers
with the intersection numbers of tautological 
classes on the moduli spaces of curves.
From analyzing their formula, we see the emergence of
the spectral curve, which we consider as a complex Lagrangian 
in the holomorphic symplectic surface.

Let us recall the necessary notations here. 
Our main object is the moduli stack $\Mbar_{g,\ell}$
consisting of stable algebraic curves of genus $g\ge 0$
with $\ell\ge 1$ distinct smooth labeled points. 
Forgetting the last labeled point on a curve 
gives a canonical projection 
$$
\pi:\Mbar_{g,\ell+1}\lrar 
\Mbar_{g,\ell}.
$$
Since the last labeled point moves on the curve,
the projection $\pi$ can be considered as a
\textbf{universal family} of $\ell$-pointed curves. This is 
because for each point $[C,(x_1,\dots,x_\ell)]
\in \Mbar_{g,\ell}$, the fiber of $\pi$ is 
indeed $C$ itself. 

If $x_{\ell+1}\in C$ is a smooth
point of $C$ other than $\{x_1,\dots,x_\ell\}$, 
then this point represents an element
$[C,(x_1,\dots,x_{\ell+1})]\in \Mbar_{g,\ell+1}$.
If $x_{\ell+1}=x_i$ for some $i=1,\dots,\ell$, 
then this point represents a stable curve
obtained by attaching a rational curve $\bP^1$ 
to $C$ at the original location of $x_i$, 
while carrying three special points on it. One is 
the singular point at which $C$ and $\bP^1$ 
intersects. The other two points are labeled
as $x_i$ and $x_{\ell+1}$. And if 
$x_{\ell+1}\in C$ coincides with one of the nodal points
of $C$, say $x\in C$, then this point represents 
another stable curve. This time,
consider the local normalization 
$\widetilde{C}\rar C$
about the singular point $x\in C$, and let $x_+$ and
$x_-$ be the two points in the fiber.
The stable curve we have is the curve $\widetilde{C}
\cup \bP^1$, where the two curves intersect at 
$x_+$ and $x_-$. The labeled point $x_{\ell+1}$ is
placed on the attached $\bP^1$ different from these
two singular points. Because $\bP^1$ with three distinct labeled
points form a moduli space consisting of a single point, the processes
above (called \emph{stabilization}) leave no ambiguity.

\begin{figure}[htb]
\includegraphics[width=2.9in]{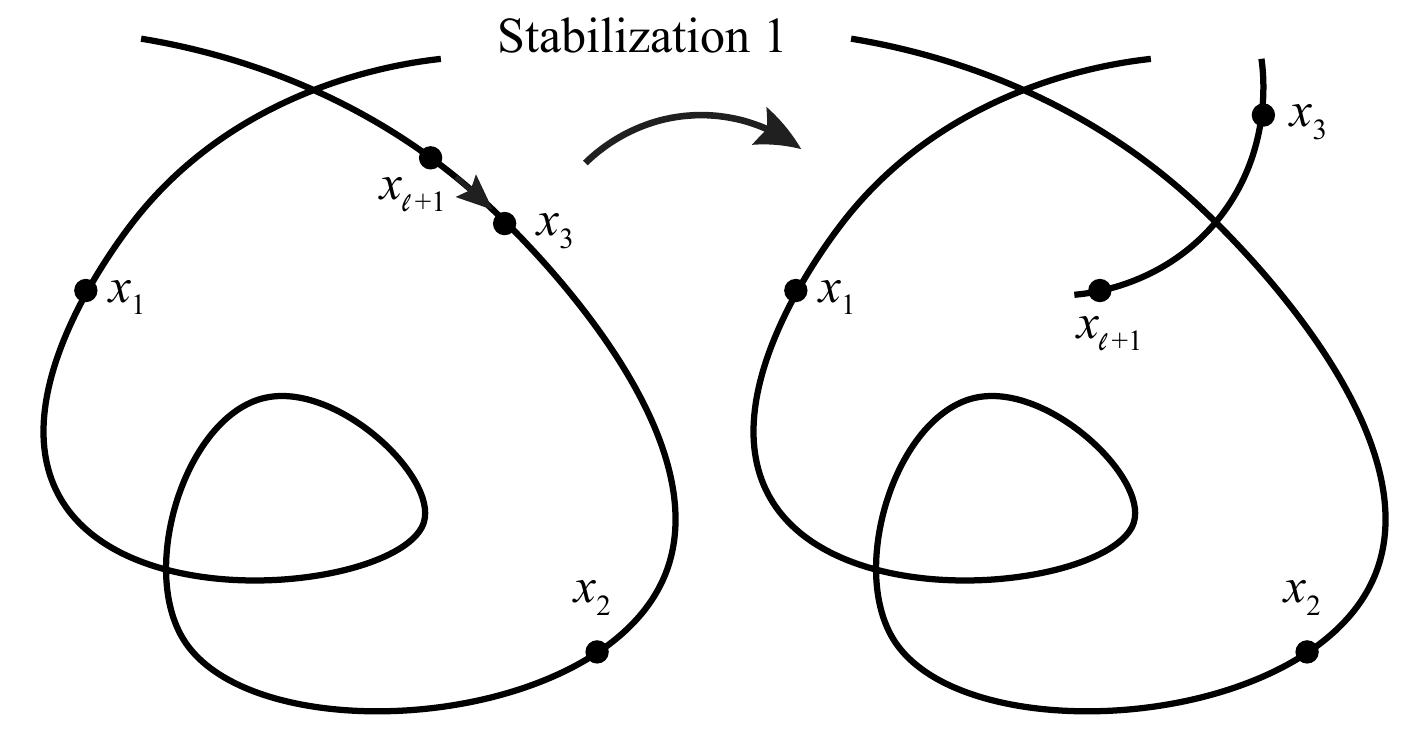}
\includegraphics[width=2.9in]{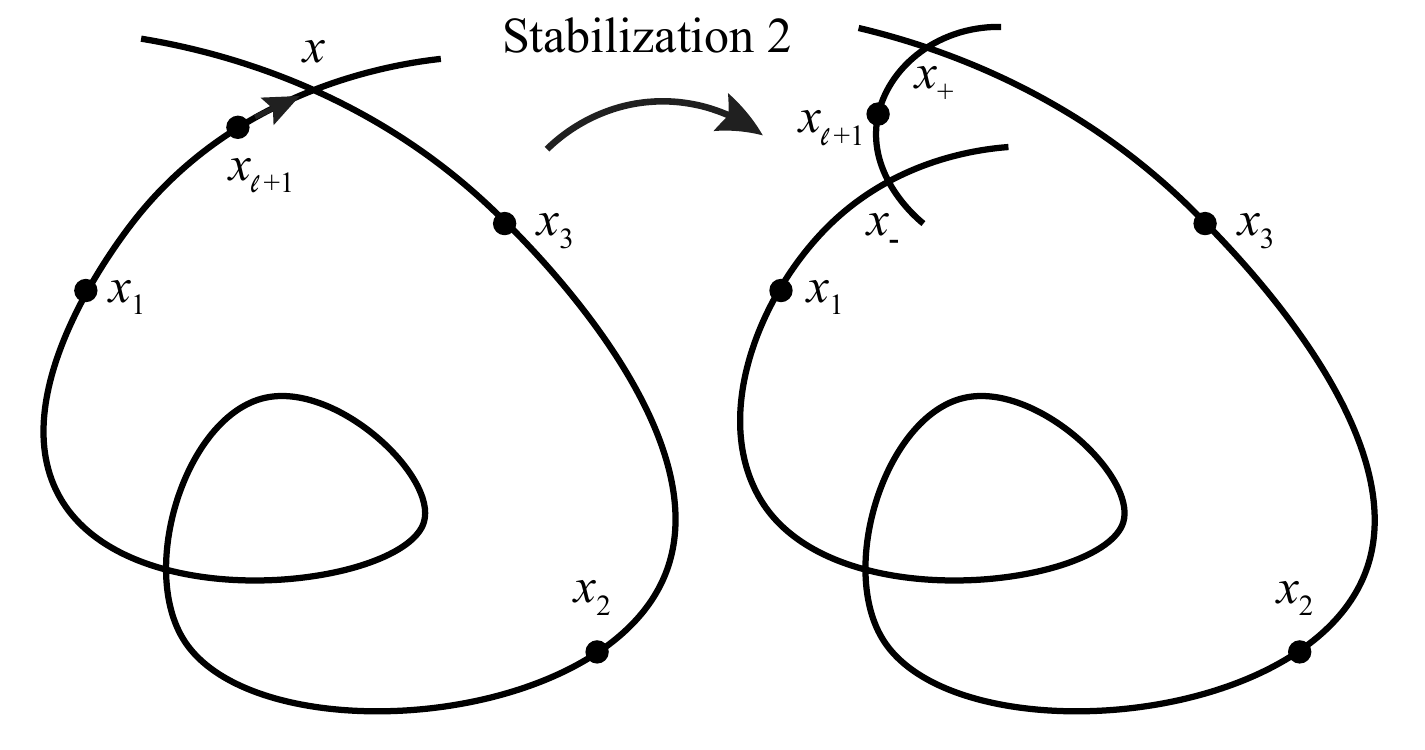}
\caption{A new point $x_{\ell+1}$ is attached on top of $x_3$ (left),
and at nodal point $x$ (right).}
\end{figure}

A universal family produces what we call
the \emph{tautological} bundles on the moduli space.
The cotangent sheaf $T^*C=\omega_C$ of each fiber of
$\pi$ is glued together to form a \emph{relative dualizing
sheaf} $\omega$ on $\Mbar_{g,\ell+1}$. 
The push-forward 
$$
\bE = \pi_*\omega
$$ 
is such
a tautological vector bundle on $\Mbar_{g,\ell}$ of 
fiber dimension 
$$
\dim H^0(C,\omega_C) = g,
$$
and is called the \emph{Hodge bundle} on 
$\Mbar_{g,\ell}$.
By assigning $x_{\ell+1} = x_i$ to 
each $[C,(x_1,\dots,x_\ell)]\in \Mbar_{g,\ell}$,
we construct a section
$$
\sigma_i :\Mbar_{g,\ell}\lrar \Mbar_{g,\ell+1}.
$$
It defines another tautological bundle
\be
\label{tautologicalL}
\bL_i=\sigma^*(\omega)
\ee
on $\Mbar_{g,\ell}$. 
The fiber of $\bL_i$ at $[C,(x_1,\dots,x_\ell)]$ is
identified with the cotangent line $T_{x_i} ^*C$.

The \emph{tautological classes} of $\Mbar_{g,\ell}$
are rational cohomology classes including
$$
\psi_i  = c_1(\bL_i)\in H^2(\Mbar_{g,\ell},\bQ)
\quad \text{and}\quad
\lambda_j = c_j(\bE)\in H^{2j}(\Mbar_{g,\ell},\bQ).
$$
In these lectures we do not consider the other classes,
such as the $\kappa$-classes.

With these notational preparations, we can now 
state an amazing theorem.

\begin{thm}[The ELSV formula \cite{ELSV}]
The simple Hurwitz numbers are expressible as
the intersection numbers of tautological 
classes on the moduli space $\Mbar_{g,\ell}$ as
follows. Let $\mu\in \bZ_+ ^\ell$ be a
positive integer vector. Then we have
\begin{equation}
\begin{aligned}
\label{eq:ELSV}
H_g(\mu)
&= 
\prod_{i=1} ^\ell 
\frac{\mu_i ^{\mu_i}}{\mu_i !}
\int_{\Mbar_{g,\ell}}
\frac{\sum_{j=0} ^g  (-1)^j \lambda_j}
{\prod_{i=1} ^\ell (1-\mu_i\psi_i)}
\\
&=
\sum_{n_1,\dots,n_\ell \ge 0}
\sum_{j=0} ^g (-1)^j
\la \tau_{n_1}\cdots \tau_{n_\ell}\lambda_j
\ra_{g,\ell}
\prod_{i=1} ^\ell 
\frac{\mu_i ^{\mu_i+n_i}}{\mu_i !}.
\end{aligned}
\end{equation}
Here we use Witten's symbol
$$
\la \tau_{n_1}\cdots \tau_{n_\ell}\lambda_j
\ra_{g,\ell}
=
\int_{\Mbar_{g,\ell}}
c_1(\bL_1) ^{n_1}\cdots
c_1(\bL_\ell) ^{n_\ell}\cdot
c_j(\bE).
$$
It is $0$ unless $n_1+\cdots+n_\ell + j = 3g-3+\ell$.
\end{thm}

\begin{rem}
It is not our purpose to give a proof of the
ELSV formula in these lectures. There are excellent
articles  about this remarkable 
formula. We refer to \cite{Liu2, OP1}.
\end{rem}

To explore the mirror partner to simple
Hurwitz numbers, we wish to compute the 
Laplace transform of the ELSV formula. 
Let us recall Stirling's formula
\begin{equation}
\label{eq:stir}
\frac{k ^{k+n}}{k!} e^{-k}\sim
\frac{1}{\sqrt{2\pi}} \; k ^{n-\half},\qquad k\gg0
\end{equation}
for a fixed $n$.
\begin{Def}
For a complex parameter $w$ with $Re(w)>0$,
we define
\begin{equation}
\label{eq:xiw}
\xi_n(w) = \sum_{k=1} ^\infty 
\frac{k ^{k+n}}{k!} e^{-k} e^{-kw}.
\end{equation}
\end{Def}
Because of Stirling's formula (\ref{eq:stir}), we expect 
that asymptotically near $w\sim 0$, 
$$
\xi_n(w) \sim \int_0 ^\infty \frac{1}{\sqrt{2\pi}}
 x^{n-\half}e^{-xw}dx.
$$
To illustrate our strategy of computing the 
Laplace transform, let us first compute
$$
f_n(w) = \int_0 ^\infty x^n e^{-xw}dx.
$$
We notice that 
\be
\label{diffrec}
-\frac{d}{dw}f_n(w) = f_{n+1}(w). 
\ee
Therefore, if we know $f_0(w)$, then we can
calculate all $f_n(w)$ for $n>0$. 
Of course we have
$$
f_0(w) =\frac{1}{w}.
$$
Therefore, we immediately conclude that
\begin{equation}
\label{eq:fn}
f_n(w) = \frac{\Gamma(n+1)}{w^{n+1}},
\end{equation}
which satisfies the initial condition and the
differential recursion formula \eqref{diffrec}.
The important fact in complex analysis is that
when we derive a formula like (\ref{eq:fn}),
it holds for an arbitrary $n$, not necessarily a positive integer. In particular,
we have
$$
\xi_n(w) \sim \int_0 ^\infty \frac{1}{\sqrt{2\pi}}
 x^{n-\half}e^{-xw}dx = 
 \frac{\Gamma(n+\half)}{\sqrt{2\pi} \;w^{n+\half}}.
$$
From this asymptotic expression, we learn that
$\xi_n$ has an expansion in
$w^{-\half}$. Thus to identify the domain of 
holomorphy, we wish to find a natural coordinate
that behaves like $w^{-\half}$.

Note that for every $n>0$, the  summation in the definition 
of $\xi_n(w)$ in (\ref{eq:xiw}) can be
taken from $k=0$ to $\infty$. For $n=0$, the $k=0$
term contributes $0^0=1$ in the summation. So let us define
\begin{equation}
\label{eq:t}
t-1 = \xi_0(w) = \sum_{k=1} ^\infty
\frac{k^k}{k!}e^{-k}e^{-kw}.
\end{equation}
Then the computation of the Laplace transform
$\xi_n(w)$ is reduced to finding the 
inverse function $w = w(t)$ of (\ref{eq:t}),
because all we need after identifying the inverse is
to differentiate $\xi_0(w)$ $n$-times. 

Here we utilize the \emph{Lagrange Inversion Formula}.
\begin{thm}[The Lagrange Inversion Formula]
\label{thm:L}
Let $f(y)$ be a holomorphic function 
defined near $y=0$ such that $f(0)\ne 0$. 
Then the inverse of the function 
$$
x= \frac{y}{f(y)}
$$
is given by
$$
y=\sum_{k=1} ^\infty
\left[
\frac{d^{k-1}}{dy^{k-1}}(f(y))^k
\right]_{y=0}
\frac{x^k}{k!}.
$$
\end{thm}

We give a proof of this formula in Appendix.
For our purpose, let us consider 
the case $f(y) = e^{y-1}$. The function
\begin{equation}
\label{eq:Lambert}
x = y e^{1-y}
\end{equation}
is called the \textbf{Lambert function}.

\textbf{This is our spectral curve, the hidden curve for Hurwitz numbers!}
The difficulty is in the calculation of \eqref{eq:xiw}. Since we do not know
how to calculate it, we just introduce the symbol $y$ for \eqref{yHur}
below, \emph{pretending} that we know what it is. Then we
can calculate \emph{everything} in terms of what we do not know, i.e.,
\eqref{yHur}. Indeed,
the Lagrange Inversion Formula 
immediately tells us that
its inverse function is given by
\be
\label{yHur}
y= \sum_{k=1} ^\infty \frac{k^{k-1}}{k!}e^{-k}
x^k.
\ee
So if we substitute 
\begin{equation}
\label{eq:x}
x = e^{-w},
\end{equation}
then we have
\begin{equation}
\label{eq:y}
y=\sum_{k=1} ^\infty \frac{k^{k-1}}{k!}e^{-k}
e^{-kw} = \xi_{-1}(w).
\end{equation}
The differential of the Lambert function gives
$$
dx = (1-y)e^{1-y} dy.
$$
Therefore, we have
\begin{equation}
\label{eq:wxy}
-\frac{d}{dw}=x\frac{d}{dx}=
\frac{y}{1-y}\frac{d}{dy}.
\end{equation}
Since 
$$
t-1 = \xi_0(w)=-\frac{d}{dw}\xi_{-1}(w) = 
\frac{y}{1-y},
$$
we conclude that
\begin{equation}
\label{eq:yint}
y=\frac{t-1}{t}.
\end{equation}
As a consequence, we complete the calculation:
\begin{equation}
\label{eq:wxyt}
-\frac{d}{dw}=x\frac{d}{dx}=
\frac{y}{1-y}\frac{d}{dy}
= t^2(t-1)\frac{d}{dt}.
\end{equation}
We also obtain a formula for $w$ in terms of $t$,
since $e^{-w}=ye^{1-y}$.
\begin{equation}
\label{eq:wint}
w=-\frac{1}{t}-\log\left(1-\frac{1}{t}\right)
=\sum_{m=2} ^\infty \frac{1}{m}\;\frac{1}{t^m}.
\end{equation}
Notice that near $w=0$, we have $t\sim \sqrt{2w}$,
as we wished!
Now we can calculate  $\xi_n(w)$ in terms of 
$t$ for every $n\ge 0$. 

\begin{Def}
As a function in $t$, we denote 
\begin{equation}
\label{eq:xihat}
\hxi_n(t) = \xi_n(w(t)).
\end{equation}
\end{Def}

\begin{thm}[Polynomiality]
For every $n\ge 0$, 
$\hxi_n(t)$ is a \textbf{polynomial} in $t$ of degree
$2n+1$. For $n>0$ it has an expansion 
\begin{equation}
\label{eq:hxi-expansion}
\hxi_n(t)=
(2n-1)!! t^{2n+1} - \frac{(2n+1)!!}{3}\; t^{2n}+\cdots
+a_n t^{n+2} + (-1)^n n! \;t^{n+1},
\end{equation}
where $a_n$ is defined by
$$
a_{n} =-\big[(n+1)a_{n-1} +(-1)^n n!\big]
$$
and is identified as the sequence A001705 or A081047 of the
\emph{On-Line Encyclopedia of Integer Sequences}.
\end{thm}

\begin{proof}
It is a straightforward calculation of
$$
\hxi_n(t) = t^2(t-1)\frac{d}{dt}\hxi_{n-1}(t)
=\left(t^2(t-1)\frac{d}{dt}\right)^n (t-1).
$$
\end{proof}

\begin{rem} J.~Zhow informed the author that all other coefficients of
$\hxi_n(t)$
 had been identified.

\end{rem}

\begin{thm}[Laplace transform of simple Hurwitz
numbers]
The Laplace transform of simple Hurwitz numbers
is given by
\begin{equation}
\begin{aligned}
\label{eq:LTofHgmu}
\mathbf{H}_{g,\ell}(t) = \mathbf{H}_{g,\ell}(t_1,t_2,\dots,t_\ell) &=
\sum_{\mu\in\bZ_+ ^\ell}
H_g(\mu) e^{-|\mu|}e^{-(\mu_1w_1+\cdots+\mu_\ell
w_\ell)}
\\
&=
\sum_{n_1,\dots,n_\ell \ge 0}
\sum_{j=0} ^g (-1)^j
\la \tau_{n_1}\cdots \tau_{n_\ell}\lambda_j\ra_{g,\ell}
\prod_{i=1} ^\ell \hxi_{n_i}(t_i).
\end{aligned}
\end{equation}
This is a polynomial of degree
$3(2g-2+\ell)$. Its highest degree terms form a
homogeneous polynomial
\begin{equation}
\label{eq:Fgltop}
\mathbf{H}_{g,\ell} ^{\text{top}}(t)
=
\sum_{n_1+\cdots+n_\ell =3g-3+\ell}
\la \tau_{n_1}\cdots \tau_{n_\ell}\ra_{g,\ell}
\prod_{i=1} ^\ell (2n_i-1)!!\;t_i ^{2n_i+1},
\end{equation}
and the lowest degree terms also form  a
homogeneous polynomial
\begin{equation}
\label{Fgllow}
\mathbf{H}_{g,\ell} ^{\text{lowest}}(t) 
=
\sum_{n_1+\cdots+n_\ell =2g-3+\ell}
(-1)^{3g-3+\ell}
\la \tau_{n_1}\cdots \tau_{n_\ell}\lambda_g\ra_{g,\ell}
\prod_{i=1} ^\ell n_i !\; t_i ^{n_i+1}.
\end{equation}
\end{thm}

\begin{rem}
We note that there is no a priori reason for 
the Laplace transform of $H_g(\mu)$ to be
a polynomial. Because it is a polynomial,
we obtain a polynomial generating function 
of linear Hodge integrals
$\la \tau_{n_1}\cdots \tau_{n_\ell}\lambda_j\ra_{g,\ell}$.
We utilize this polynomiality in the concluding subsection below.
\end{rem}

\begin{rem}
The motivation for the author to work on
Catalan numbers  \cite{DMSS,OM4} was to find an analytically
simpler example of enumeration problem for which a similar
polynomiality holds. The coordinate change for the Catalan case was 
obtained from looking for an analogy of \eqref{eq:yint}.
\end{rem}

\begin{rem}
The existence of the polynomials $\hxi_n(t)$
in (\ref{eq:LTofHgmu}) is significant,
because it reflects the ELSV formula (\ref{eq:ELSV}).
Indeed, Eynard  predicts that 
this is the general structure of the Eynard-Orantin
formalism.
\end{rem}

\subsection{Two quantum curves for the Lambert curve}
\label{FugueQC}

How about the quantum curve associated with the Lambert curve?
Since $x=y e^{1-y}$ does not give an algebraic curve,
the argument we used for the case of Catalan generating function,
\eqref{xz} and \eqref{xzdiff},
does not apply here. 
Yet the counterpart of the \emph{quantum curve} \eqref{qxz}
exists for Hurwitz numbers. Indeed, we discover \emph{two}
differential equations that recover the Lambert curve
via the semi-classical limit.

Our particular choice of the Lambert function \eqref{eq:Lambert} is for the purpose
to make the subsequent calculations less cumbersome,
by reducing the appearance of  powers of $e$. More traditional 
choice is
\be
\label{classicalLambert}
x = y e^{-y}.
\ee 
The difference is only in the constant multiplication $x\longmapsto ex$. 
For the calculation of quantum curves, we use the classical one.
This works better for quantum curves, because the differential operator we need
is $D:= x\frac{d}{dx}$, which is invariant under the $\bC^*$-action. 

\begin{thm}[\cite{BHLM,MSS, MS2015}]
As a straightforeward analogy of the formula \eqref{CatalanPsi} for 
Catalan numbers, let us define
\be
\label{HurwitzZ}
Z(t,\hbar):=
\exp\left(\sum_{g\ge0, n>0} \frac{1}{n!}\, 
\hbar ^{2g-2+n}\, \mathbf{H}_{g,n}(t,t,\dots,t)\right).
\ee
Then it satisfies the following  two equations, one is a partial 
differential equation, and
the other a difference-differential equation:
\begin{align}
\label{partial}
&\left(\frac{\hbar}{2} D^2 -\left(1+\frac{\hbar}{2}\right)D
-\hbar \,\frac{\partial}{\partial \hbar}\right)   Z\big(t(x),\hbar\big)= 0,
\\
\label{difference}
&
\left(
\hbar D - x e^{\hbar  D}
\right) Z\big(t(x),\hbar\big) = 0.
\end{align}
Here, $D = x\frac{d}{d x}$, and the variable 
$t$ in \eqref{HurwitzZ} is considered to be a function in $x$
by the relations
$$
y = \frac{t-1}{t} \quad \text{and}\quad x = y e^{-y}.
$$
\end{thm}

The semi-classical limit calculations can be applied to these partial differential
and difference-differential equations. The result is the same as
replacing 
$
\begin{cases}
\hbar D\longmapsto y\\
x\longmapsto x
\end{cases}
$
for \eqref{difference}, which recovers the Lambert curve. For the
first equation \eqref{partial}, we need to perform the actual
WKB analysis to obtain the semi-classical limit, as done in \cite{MS2015}.
The quantum curve
\eqref{difference} was also  independently obtained in \cite{Zhou4}.

An important aspect arising from \eqref{HurwitzZ} is the emergence of the
KP $\tau$-function right in this formula. Indeed, $Z(t,\hbar)$ is 
a principal specialization of the KP $\tau$-function. 
We refer to \cite{MS2015} for more detail.

We are now ready to 
compute the Laplace transform 
of the cut-and-join equation itself. The result turns out
to be a simple polynomial recursion formula. 
Here again there is no a priori reason for the result to be
a polynomial relation, because the cut-and-join
equation (\ref{eq:caj}) 
contains unstable geometries, and
they contribute non-polynomial terms after 
the Laplace transform.

\begin{rem}
We  remark here that the Laplace transform of
the cut-and-join equation is equivalent to 
the Eynard-Orantin topological recursion formula
\cite{EO1}
based on the Lambert curve 
(\ref{eq:Lambert}) as the spectral curve 
of the theory. 
This fact solves the Bouchard-Mari\~no 
conjecture \cite{BM} 
of Hurwitz numbers \cite{EMS, MZ}, and establishes
the Lambert curve as the remodeled $B$-model
corresponding to simple Hurwitz numbers
through mirror symmetry.
\end{rem}

The unexpected power \cite{MZ} of the topological recursion type
formula appearing in our context is the following.
\begin{enumerate}
\item It restricts to the top degree terms, and recovers 
the Dijkgraaf-Verlinde-Verlinde formula, or the 
Virasoro constraint condition, for the
$\psi$-class intersection numbers on $\Mbar_{g,\ell}$
\cite{W1991}.
\item It also restricts to the lowest degree terms,
and recovers the $\lambda_g$-conjecture of Faber that
was proved in \cite{FP1,FP2} in a totally different method.
Our proof is straightforward.
\end{enumerate}
In other words, we obtain a straightforward, simple
proofs of the Witten conjecture and Faber's 
$\lambda_g$-conjecture from the Laplace transform of
the cut-and-join equation.
We note that the Laplace transform contains the
information of the large $\mu$ asymptotics. Therefore,
our  proof \cite{MZ} of the Witten conjecture uses the 
same idea of Okounkov and 
Pandharipande \cite{OP1}, yet it is much simpler 
because we do not have to use any
 of the asymptotic analyses
of matrix integrals, Hurwitz numbers, and graph 
enumeration.

The proof of the $\lambda_g$-conjecture  using the
topological 
recursion is still somewhat mysterious. Here again 
the complicated combinatorics is  
wiped out and we have a transparent proof.

Let us now state the Laplace transform of the 
cut-and-join equation.

\begin{thm}[\cite{MZ}]
\label{thm:main}
The polynomial generating functions of the
linear Hodge integrals $\mathbf{H}_{g,\ell}(t)$ satisfy the
following topological recursion type formula
\begin{multline}
\label{eq:main}
\left(
2g-2+\ell +\sum_{i=1} ^\ell
\frac{1}{t_i}
D_i
\right)
\mathbf{H}_{g,\ell}(t_1,t_2,\dots,t_\ell)
\\
=
\sum_{i< j}
 \frac{  t_i ^2(t_j-1)D_i
    \mathbf{H}_{g,\ell-1}\big(t_{[\ell;\hat{j}]}\big)
    -
      t_j ^2(t_i-1)D_j
    \mathbf{H}_{g,\ell-1}\big(t_{[\ell;\hat{i}]}\big)}{t_i-t_j}
\\
+
\sum_{i=1} ^\ell
\left[
D_{u_1}D_{u_2}
\mathbf{H}_{g-1,\ell+1}\big(u_1,u_2,t_{[\ell;\hat{i}]}\big)
\right]_{u_1=u_2=t_i}
\\
+
\half
\sum_{i=1} ^\ell
\sum_{\substack{g_1+g_2 = g\\
J\sqcup K= [\ell;\hat{i}]}} ^{\rm{stable}}
D_i
\mathbf{H}_{g_1,|J|+1}(t_i,t_J)\cdot 
D_i
\mathbf{H}_{g_2,|K|+1}(t_i,t_K) ,
\end{multline}
where $D_i = t_i ^2 (t_i-1)\frac{\partial }{\partial t_i}$.
As before, $[\ell]=\{1,\dots,\ell\}$ is the index set,
and $[\ell;\hat{i}]$ is the index set obtained by 
deleting $i$ from $[\ell]$. 
The last summation is taken over all partitions
$g=g_1+g_2$  of the genus $g$ and disjoint union  decompositions
$J\sqcup K= [\ell;\hat{i}]$
satisfying the stability conditions
$2g_1-1+|J|>0$ and $2g_2-1+|K|>0$. For 
a subset $I\subset [\ell]$ we write $t_I = (t_i)_{i\in I}$. 
\end{thm}

The biggest difference between the cut-and-join
equation (\ref{eq:caj}) and 
the Laplace transformed formula 
(\ref{eq:main})
is the restriction to stable geometries in
the latter. In the case of the cut-and-join equation,
the cut case contains $g_1=0$ and $I=\emptyset$.
Then $H_{g_2}(J,\beta)$ has the same complexity
of $H_g(\mu)$. Thus the cut-and-join equation
is simply a relation among Hurwitz numbers, 
not a \emph{recursive} formula.

The new feature of our (\ref{eq:main})
is that it is a genuine recursion formula
about linear Hodge integrals. Indeed, we can 
re-write the formula as follows.

\begin{multline}
\label{eq:lhirecursion}
\sum_{n_{[\ell]}}\la \tau_{n_{[\ell]}}\Lam_g ^\vee(1)\ra_{g,\ell}
\left(
 (2g-2+\ell)\hxi_{n_{[\ell]}}(t_{[\ell]})
+
\sum_{i=1} ^\ell
\frac{1}{t_i}\;
 \hxi_{n_i+1}(t_i)\hxi_{[\ell;\hat{i}]}(
t_{[\ell;\hat{i}]})
\right)
\\
=
\sum_{i<j}
\sum_{m,n_{[\ell;\hat{i}\hat{j}]}}
\la \tau_m\tau_{n_{[\ell;\hat{i}\hat{j}]}}
\Lam_{g} ^\vee (1)\ra_{g,\ell-1}
\hxi_{n_{[\ell;\hat{i}\hat{j}]}}(t_{[\ell;\hat{i}\hat{j}]})
\frac{\hxi_{m+1}(t_i) \hxi_0(t_j) t_i ^2 
-
\hxi_{m+1}(t_j) \hxi_0(t_i) t_j ^2}{t_i-t_j}
\\
+
\half\sum_{i=1}^\ell
\sum_{n_{[\ell;\hat{i}]}}\sum_{a,b}
\Bigg(
\la \tau_a\tau_b\tau_{n_{[\ell;\hat{i}]}}\Lam_{g-1} ^\vee
(1)\ra_{g-1,\ell+1}\\
+
\sum_{\substack{g_1+g_2 = g\\
I\coprod J = [\ell;\hat{i}]}} ^{\rm{stable}}
\la \tau_a\tau_{n_{I}}\Lam_{g_1} ^\vee
(1)\ra_{g_1,|I|+1}
\la \tau_b\tau_{n_{J}}\Lam_{g_2} ^\vee
(1)\ra_{g_2,|J|+1}
\Bigg)
\hxi_{a+1}(t_i)\hxi_{b+1}(t_i)\hxi_{n_{[\ell;\hat{i}]}}(
t_{[\ell;\hat{i}]}),
\end{multline}
where $[\ell]=\{1,2\dots,\ell\}$ is the index set, 
and for a subset 
$I\subset [\ell]$, we denote
$$
t_I = (t_i)_{i\in I},\quad
	 n_I = \{\, n_i\,|\, i\in I\,\},\quad 
	\tau_{n_I} = \prod_{i\in I}\tau_{n_i},
	\quad
		\hxi_{n_I}(t_I) = \prod_{i\in I}
		\hxi_{n_i}(t_i).
$$
We also use a convenient notation
$$
\Lambda_{g} ^\vee (1) = 1-\lambda_1+\lambda_2-\cdots+
(-1)^g\lambda_g.
$$
It is now obvious that in (\ref{eq:lhirecursion}),
the complexity $2g-2+\ell$ is reduced 
exactly by $1$ on the right-hand side. Thus we
can compute linear Hodge integrals one by one
using this formula.

The Deligne-Mumford stack $\overline{\mathcal{M}}_{g,\ell}$
is defined as the moduli space of \emph{stable} curves satisfying the
stability condition
$2-2g-\ell <0$.  However, Hurwitz numbers 
are well defined for \emph{unstable} geometries
 $(g,\ell) = (0,1)$ and $(0,2)$. 
It is an elementary exercise (of tree counting, see \cite{Ltext}) to 
show that
\be
\label{01Hur}
H_0\big((d)\big) = \frac{d^{d-1}}{d!}.
\ee
We note that \textbf{this is the type $(0,1)$-Hurwitz number of degree $d$, 
and the new variable $y=y(x)$ of \eqref{yHur},
or the spectral curve, is  its
 generating
function.}

The ELSV formula remains true for unstable cases
by \emph{defining}
\begin{align}
\label{eq:01Hodge}
&\int_{\overline{\cM}_{0,1}} \frac{1}{1-k\psi}
=\frac{1}{k^2},\\
\label{eq:02Hodge}
&\int_{\overline{\cM}_{0,2}} 
\frac{1}{(1-\mu_1\psi_1)(1-\mu_2\psi_2)}
=\frac{1}{\mu_1+\mu_2}.
\end{align}
In terms of simple Hurwitz numbers,
we have
$$
H_0\big((\mu_1,\mu_2)\big)
=\frac{\mu_1^{\mu_1}}{\mu_1 !}\cdot
\frac{\mu_2^{\mu_2}}{\mu_2 !}\cdot
\frac{1}{\mu_1+\mu_2}.
$$

From these expressions we can actually compute
$\mathbf{H}_{0,1}(t)$ and $\mathbf{H}_{0,2}(t_1,t_2)$. Since these
computations are quite involved, we refer to 
\cite{EMS,MZ}. What happens often in mathematics
is what we call a \emph{miraculous cancellation}.
In our situation, when we honestly compute
all terms appearing in the Laplace transform
in the cut-and-join equation (\ref{eq:caj}),
somewhat miraculously, all non-polynomial
terms cancel out, and the rest becomes an
effective recursion formula (\ref{eq:lhirecursion}).

\subsection{New proofs of Witten-Kontsevich 
 and the $\lambda_g$ formulas}

Now let us move to proving the Witten conjecture
and the $\lambda_g$-formula using our 
recursion (\ref{eq:lhirecursion}). Although 
these important formulas have been proved a long
time ago, we present simpler proofs here just to 
illustrate the power of the topological recursion type formula.

The DVV formula for the
Virasoro constraint condition on the $\psi$-class
intersections  reads 
\begin{multline}
\label{eq:DVV}
\la \tau_{n_{[\ell]}}\ra _{g,\ell}=
\sum_{j\ge 2} 
\frac{(2n_1+2n_j-1)!!}{(2n_1+1)!! (2n_j -1)!!}
 \la \tau_{n_1+n_j-1}
 \tau_{n_{[\ell;\hat{1}\hat{j}]}}
\ra _{g,\ell-1}
\\
+
\frac{1}{2} \sum_{a+b=n_1-2}
\left(
\la \tau_a\tau_b\tau_{n_{[\ell;\hat{1}]}}
\ra _{g-1,\ell+1}
+
\sum_{\substack{g_1+g_2=g\\
J\sqcup K= [\ell;\hat{1}]}} ^{
\text{stable}}
\la \tau_a\tau_{n_J}\ra _{g_1,|J|+1}\cdot
\la \tau_b\tau_{n_K}\ra _{g_2,|K|+1}
\right)
\\
\times
\frac{(2a+1)!!(2b+1)!!}{(2n_1+1)!!}.
\end{multline}
Here $[\ell;\hat{1}\hat{j}]
=\{2,3,\dots,\hat{j},\dots,\ell\}$, and
for a subset $I\subset [\ell]$ we write
$$
n_I = (n_i)_{i\in I}\qquad{\text{and}}\qquad
\tau_{n_I}=\prod_{i\in I}\tau_{n_i}.
$$

\begin{prop}
\label{prop:WK}
The DVV formula {\rm(\ref{eq:DVV})} is exactly the 
relation among the top degree coefficients
of the recursion {\rm(\ref{eq:main})}.
\end{prop}

\begin{proof}
Choose $n_{[\ell]}$ so that 
$|n_{[\ell]}|= n_1+n_2+\cdots +n_\ell=3g-3+\ell$.
The degree of the left-hand side of (\ref{eq:main}) is $3(2g-2+\ell) +1$. 
So we
 compare the coefficients of $t_1 ^{2n_1+2}\prod_{j\ge 2} t_j ^{2n_j+1}$
in the recursion formula. 
The contribution from the left-hand side of (\ref{eq:main}) is 
$$
\la \tau_{n_{[\ell]}}\ra _{g,\ell}
(2n_1+1)!!\prod_{j\ge 2} (2n_j-1)!!.
$$
The contribution from the first line of the right-hand side comes from
\begin{multline*}
\sum_{j\ge 2}
\la\tau_m\tau_{n_{[\ell;\hat{1}\hat{j}]}}\ra_{g,\ell-1}
(2m+1)!! 
 \frac{t_1^2 t_j t_1 ^{2m+3}
-t_j^2t_1t_j ^{2m+3}}{t_1-t_j}
\\
=
\sum_{j\ge 2}
\la\tau_m\tau_{n_{[\ell;\hat{1}\hat{j}]}}\ra_{g,\ell-1}
(2m+1)!! t_1t_j
 \frac{ t_1 ^{2m+4}
-t_j ^{2m+4}}{t_1-t_j}
\\
=
\sum_{j\ge 2}
\la\tau_m\tau_{n_{[\ell;\hat{1}\hat{j}]}}\ra_{g,\ell-1}
(2m+1)!!
\sum_{a+b=2m+3}t_1 ^{a+1}t_j ^{b+1},
\end{multline*}
where $m=n_1+n_j-1$. The matching term in this
formula is $a=2n_1+1$ and $b=2n_j$. Thus we extract
as the coefficient of $t_1 ^{2n_1+2}\prod_{j\ge 2} t_j ^{2n_j+1}$
$$
\sum_{j\ge 2}
\la\tau_{n_1+n_j-1}\tau_{n_{[\ell;\hat{1}\hat{j}]}}\ra_{g,\ell-1}
(2n_1+2n_j-1)!!\prod_{k\ne 1,j}(2n_k-1)!!.
$$
The contributions of the second and the third lines 
of the right-hand side  of (\ref{eq:main}) are
\begin{multline*}
\half
\sum_{a+b=n_1-2}
\left(
\la \tau_a \tau_b \tau_{L\setminus \{1\}}\ra_{g-1,\ell+1}
+\frac{1}{2}
\sum_{\substack{g_1+g_2=g\\
J\sqcup K= [\ell;\hat{1}]}} ^{
\text{stable}}
\la \tau_a\tau_{n_J}\ra _{g_1,|J|+1}\cdot
\la \tau_b\tau_{n_K}\ra _{g_2,|K|+1}
\right)
\\
\times
(2a+1)!!(2b+1)!!\prod_{j\ge 2}(2n_j-1)!!.
\end{multline*}
We have thus recovered the Witten-Kontsevich theorem
\cite{DVV, K1992, W1991}.
\end{proof}

The $\lambda_g$ formula \cite{FP1, FP2} is
\begin{equation}
\label{eq:lam-g}
\la \tau_{n_{[\ell]}}\lambda_g\ra_{g,\ell}
=\binom{2g-3+\ell}{n_{[\ell]}}b_g,
\end{equation}
where 
\begin{equation}
\label{eq:multinomial}
\binom{2g-3+\ell}{n_{[\ell]}}=
\binom{2g-3+\ell}{n_1,\dots,n_\ell}
\end{equation}
is the multinomial coefficient, and 
$$
b_g = \frac{2^{2g-1}-1}{2^{2g-1}}\;
\frac{|B_{2g}|}{(2g)!}
$$
is a coefficient of the series
$$
\sum_{j=0} ^\infty b_j s^{2j}
= \frac{s/2}{\sin(s/2)}.
$$

\begin{prop}
\label{prop:lam-g}
The lowest degree terms of the topological recursion {\rm{(\ref{eq:main})}}
proves the combinatorial factor of the
$\lambda_g$ formula
\begin{equation}
\label{eq:lam-g-combinatorial}
\la \tau_{n_{[\ell]}}\lambda_g\ra_{g,\ell}
=\binom{2g-3+\ell}{n_{[\ell]}}
\la \tau_{2g-1}\lambda_g\ra_{g,1}.
\end{equation}

\end{prop}

\begin{proof}
Choose  $n_{[\ell]}$ subject to $|n_{[\ell]}|=2g-3+\ell$. 
We compare the coefficient of  the terms of
 $\prod_{i\ge 1}
t_i ^{n_i+1}$ in (\ref{eq:main}), which has
degree $|n_{[\ell]}|+\ell =2g-3+2\ell$. The left-hand side contributes
\begin{multline*}
(-1)^{2g-3+\ell}(-1)^g
\la\tau_{n_{[\ell]}}\lambda_g\ra_{g,\ell}
\prod_{i\ge 1} n_i!
\left(
2g-2+\ell - \sum_{i=1} ^\ell (n_i+1)
\right)
\\
=
(-1)^{\ell}(-1)^g
\la\tau_{n_{[\ell]}}\lambda_g\ra_{g,\ell}
(\ell-1)
\prod_{i\ge 1}n_i!.
\end{multline*}
The lowest degree terms of the
first line of the right-hand side are
\begin{equation*}
(-1)^g
\sum_{i< j}\sum_m
\la \tau_m\tau_{n_{[\ell;\hat{i}\hat{j}]}}
\lambda_g\ra _{g,\ell-1}
(-1)^{m}(m+1) !
 \frac{  
 t_i ^{m+4}
    -
    t_j^{m+4}}{t_i-t_j}
    (-1)^{2g-3+\ell-n_i-n_j}
    \prod_{k\ne i,j}n_k! t_k ^{n_k+1}.
\end{equation*}
Since $m=n_i+n_j-1$, the coefficient of  $\prod_{i\ge 1}
t_i ^{n_i+1}$ is
\begin{equation*}
-(-1)^g(-1)^{2g-3+\ell}\sum_{i< j}
\la \tau_{n_i+n_j-1}\tau_{n_{[\ell;\hat{i}\hat{j}]}}
\lambda_g\ra _{g,\ell-1}
\binom{n_i+n_j}{n_i}
\prod_{i\ge 1}n_i! .
\end{equation*}
Note that the lowest degree coming from the second and the third
lines of the right-hand side of (\ref{eq:main}) is 
$|n_{[\ell]}|+\ell + 2$, which
is higher than the lowest degree of the left-hand side. 
Therefore, we have 
obtained a recursion equation with respect to $\ell$
\begin{equation}
\label{eq:lam-g-recursion}
(\ell-1)\la \tau_{n_{[\ell]}}\lambda_g \ra_{g,\ell}=
\sum_{i< j}
\la \tau_{n_i+n_j-1}\tau_{n_{[\ell;\hat{i}\hat{j}]}}
\lambda_g\ra _{g,\ell-1}
\binom{n_i+n_j}{n_i}.
\end{equation}
The solution of the recursion equation 
(\ref{eq:lam-g-recursion})
is the multinomial coefficient (\ref{eq:multinomial}).
\end{proof}

\begin{rem}
Although the polynomial recursion 
(\ref{eq:main}) determines all linear Hodge integrals,
the closed formula
$$
b_g = \la \tau_{2g-2}\lambda_g\ra_{g,1}\qquad g\ge 1
$$ 
does not directly follow  from it.
\end{rem}

\subsection{From a tree counting to  the Lambert curve}
\label{Coda}

As in Section~\ref{FugueQC}, but different from \eqref{yHur}, let us use
the classical convention 
\be
\label{ynew}
y(x)=\sum_{k=1}^\infty \frac{k^{k-1}}{k!}x^k
\ee
for the function $y=y(x)$. Then the Lagrange Inversion Theorem
gives us the classical Lambert curve $x = y e^{-y}$, as pointed out in \eqref{classicalLambert}. We take the exterior derivative of this expression
$dx= (1-y)e^{-y}dy$. Then we obtain a differential equation 
for $y(x)$:
\be
\label{ydiff}
Dy = \frac{y}{1-y}, \qquad D = x\,\frac{d}{dx}.
\ee
Since it is a nonlinear differential equation, it is  not 
a Picard-Fuchs type equation such as the one \eqref{xzdiff}
for the Catalan case. However, as the Picard-Fuchs equation \eqref{xzdiff}
leads to the spectral curve $x=z+1/z$, 
the differential equation \eqref{ydiff} actually 
determines the Lambert curve, as we see below.

In this Subsection, we deduce Differential Equation \ref{ydiff}
from a purely combinatorial nature of the tree counting, and solve
it to identify the Lambert curve,
without 
appealing to the analysis of the Lagrange theorem. 

We learn from the excellent textbook of Lov\'asz et al.\ \cite{Ltext} that
the $(0,1)$ Hurwitz number of degree $d$,
$H_0\big((d)\big) =\frac{d^{d-1}}{d!}$ of \eqref{01Hur},
is the number of \emph{rooted trees} on $d$ nodes, 
counted in the stack sense. It means the reciprocal of the
order of the automorphism 
group of each tree is used as a weight in counting. 
Cayley's Theorem says that the number of all \emph{node-labeled} trees
on $d$ nodes is $d^{d-2}$, which is of course a positive integer. 
Since $d$ nodes are labeled in $d!$ different ways, 
the ratio $\frac{d^{d-2}}{d!}$ is the ``number'' of \emph{unlabeled}
trees. Since it is not an integer for $d\ge 2$, this ``counting'' is not
the count of elements of a set. We are counting \emph{objects}
in a category, and automorphisms are taken into account. 

\begin{quest}
If you have two isomorphic objects, then you count it as one. 
For example, there are two groups of order $6$, the cyclic group $C_6$
and the permutation group $S_3$. The dihedral group of a triangle $D_3$
is not counted because it is isomorphic to $S_3$.
Now a question: Suppose we have an object that 
has a non-trivial automorphism group of order $2$. 
Do we count it as one, or a \textbf{half}?
\end{quest}

\begin{rem}[On categorical counting]
\label{categorical}
\bite
\item
The textbook \cite{Ltext} mentioned above also talks about the set-theoretical count of 
the number of trees. There is no exact formula for that number. 
What becomes an important question in set-theoretical count is the \emph{asymptotic
behavior} of the number. 

\item
When we count,
it is always the best practice to count labeled objects first, because then 
we have a clear definition of the objects we are counting. 
For example, in \eqref{Fgn Catalan}, $C_{g,n}(\mu_1,\dots,\mu_n)$
is the number of \emph{cell graphs} with \emph{labeled}
vertices and no local rotation symmetries around each vertex are allowed.
We then obtain  $C_{g,n}(\mu_1,\dots,\mu_n)$ as an integer. 
But in \eqref{Fgn Catalan}, we have the denominator $\mu_1\cdots\mu_n$,
which is exactly the order of the product group of the rotation group at each vertex. This ratio is therefore \emph{not} equal to the ``number'' of 
actual isomorphism classes of a cell graph. 

\item 
Identifying the automorphism group of a large cell graph, even a tree, 
is a computationally complex task. Often the \emph{categorical count}
leads to a beautiful formula. We employ this point of view
 everywhere in these lectures.

\item
But if your interest is really the actual set-theoretical count, then what do
you do? Since there is no exact formula expected, the question is how you 
obtain the asymptotic analysis of the formula. 
\item \textbf{Bingo!} Yes, you got it. \textbf{This is why we are
calculating the Laplace transform!} 
In both Catalan and Hurwitz cases, the Laplace transform 
leads to the topology of the moduli spaces $\cM_{g,n}$ and $\Mbar_{g,n}$. 
The idea here is similar to the Ehrhart polynomials and the Weil conjecture.
\eite
\end{rem}

So we have $\frac{d^{d-2}}{d!}$ \emph{unlabeled} trees on $d$ nodes
in our categorical count. 
To make a tree a \emph{rooted} tree, we need to pick a node and declare
that it is the root. We have $d$ such choices. Therefore, the number
of rooted trees on $d$ nodes is 
$$
d\, \frac{d^{d-2}}{d!}=\frac{d^{d-1}}{d!}.
$$
Suppose now we have two rooted trees. We can join these two roots
with a new edge. The result is a new tree, where there is no root any more.
But we have a particular edge, which did not exist in any of the two original
trees. Thus we obtain a \textbf{based tree}, i.e., a tree 
in which a particular edge is chosen
and declared it to be the \emph{base} of the tree.

\begin{figure}[htb]
\label{figtree}
\center{\includegraphics[width=3.4in]{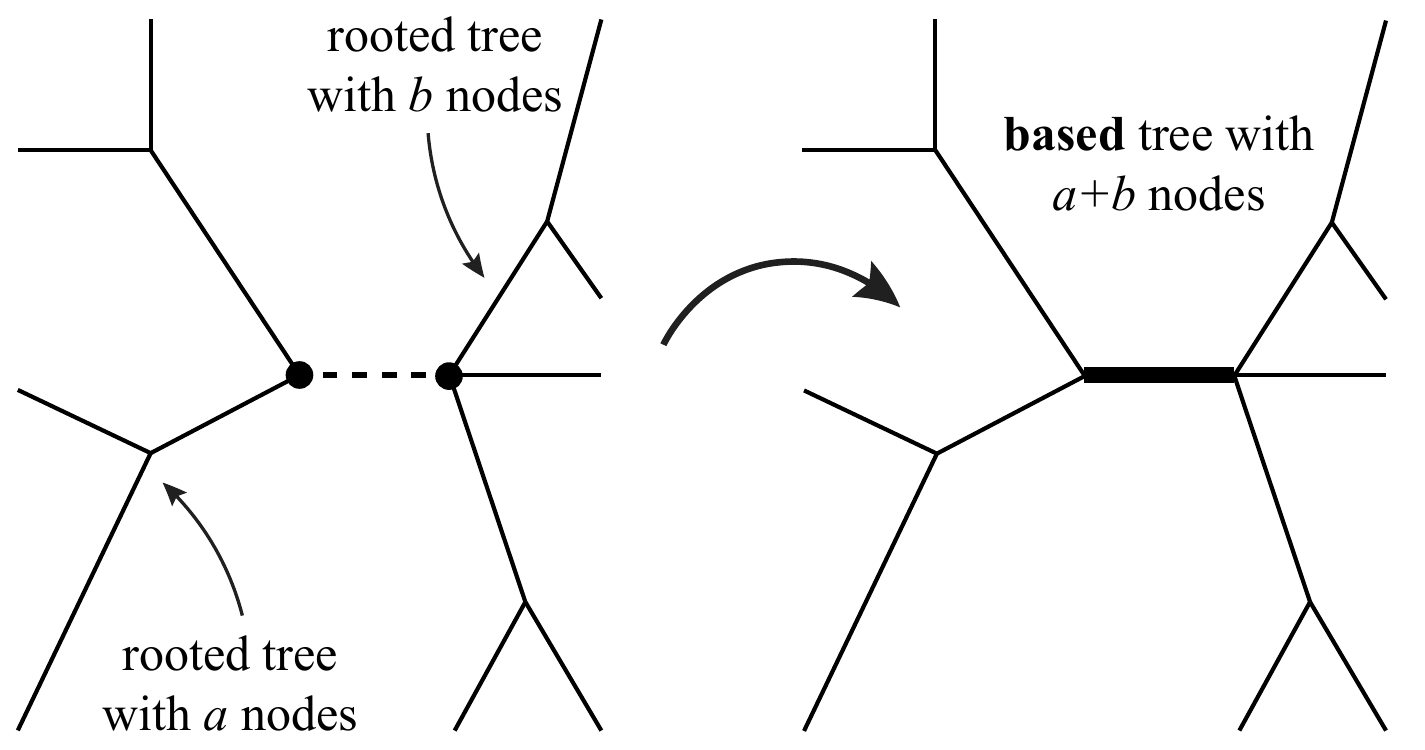}}
\caption{Construction of a \textbf{based} tree from two 
\textbf{rooted} trees.}
\end{figure}

The process if reversible. When you have a tree, and an edge has
a name ``base'' on it, then we can simply remove it. Being a tree,
the process produces two disjoint trees. 
The two ends of the removed edge will become the roots of the
two trees. A tree on $n$ nodes has $n-1$ edges. 
So we obtain a bijective counting formula
\be
\label{treecount}
\half \sum_{\substack{a+b=n\\a\ge 1,b\ge 1}} \frac{a^{a-1}}{a!}
\cdot 
 \frac{b^{b-1}}{b!}
 =
 (n-1)\,\frac{n^{n-2}}{n!}, \qquad n\ge 2.
\ee
The half on the left hand side compensates the double count of
interchanging the two parts. 
In terms of the generating function, or the Laplace transform $y(x)$ 
of \eqref{ynew}, the tree counting formula
\eqref{treecount}, after multiplying $nx^n$, produces the following
differential equation with $D=x(d/dx)$:
$$
D\left(\half y^2\right) = xD\left(\frac{y}{x}\right) \quad
\Longleftrightarrow \quad
yy' = \frac{xy'-y}{x}
\quad
\Longleftrightarrow \quad
Dy = \frac{y}{1-y}.
$$
Thus we recover \eqref{ydiff},  and this is exactly what we need
throughout Section~\ref{sect:Hurwitz}. We can also 
directly solve this differential equation and find the
classical Lambert curve, avoiding Lagrange's Inversion Formula 
all together:
$$
(1-y)\;\frac{y'}{y} = \frac{1}{x} \quad \Longrightarrow \quad
x = ye^{-y}.
$$
Here, we use the initial values $y(0)=0$ and $y'(0)=1$ from 
\eqref{ynew}. 


\section{Spectral curves in Higgs bundles and opers}
\label{sect:var}

The goal of this section is to review the biholomorphic
correspondence between the moduli space of spectral curves
and that of opers for any given smooth projective algebraic curve $C$
of genus $g(C)>1$, mentioned in Introduction. Since the passage
goes through $\hbar$-family of deformations of a differential operator,
it has a counter intuitive feature.

\subsection{Moduli spaces of Higgs bundles and character varieties}

Hitchin's introduction \cite{H1} of spectral curves in the study of 
the cotangent bundle of the moduli spaces of vector bundles on a
 smooth base curve has considerably 
expanded the world of \emph{spectral curves}. 
Spectral curves have appeared independently
 in integrable systems and random matrix theory.
With  Hitchin's work, they appear at the heart of 
algebraic geometry \cite{BNR}. 

Looking at the formalism of Lax \eqref{KP}, we immediately see the following:
deformations that the system of PDEs are making
preserve eigenvalues of the Lax operator. This is due to the right-hand side
of the equation, which is a commutator. Because of this feature,
the soliton equation type integrable systems are studied through the 
unified idea of \emph{isospectral deformation theory}.
The spectral curve $\Spec(A_L)$ of \eqref{specL} is therefore 
the fundamental invariant of the evolution equation. 

Hitchin's perspective is to deform these spectral curves. For a pair of 
a vector bundle and a curve $(E,C)$, one can construct a pair of 
a spectral curve and a line bundle on it, $(\pi:\Sigma\rar C, \cL)$,
such that $E\isom \pi_* \cL$. When we move the line bundle on the 
Jacobian, the construction ``covers'' the moduli space of 
vector bundles on $C$ of a fixed topology, by
``recovering'' the vector bundle as $\pi_*\cL$. 
But there is no canonical choice of $\Sigma$. 
Then, why don't  we consider all possibilities?

From the point of view of \eqref{spectral} and the goal of 
constructing differential operators, we now place our spectral curve 
in the cotangent bundle $T^*C$ of $C$ as
\be
\label{Hitchinspectral}
\xymatrix{
\Sigma \ar[dr]_{\pi}\ar[r]^{i} 
&{T^*C}\ar[d]^{\pi}
\\
&\;C	.}
\ee
Such a curve arrises as the \emph{spectrum} of a matrix
$\phi: E\lrar E\tensor K_C$, which Hitchin named a
\emph{Higgs field}. It is a matrix of $1$-forms, $\phi\in H^0(C, \End(E)\tensor  K_C)$.
The spectral curve is the set of 
eigenvalues of $\phi$, i.e., 
$$
\Sigma=\left\{ (z,\lam)\in T^*C\;\left|\;z\in C,\; \lam\in T^*_z C, \;
\det\big(\lam I - \phi(z)\big)=0\right\}\right. .
$$
Notice that the infinitesimal deformation 
space of $E$ is $H^1(C,\End(E))$, which is dual to the space
of Higgs fields:
\be
\label{Serre}
H^1(C,\End(E)) \isom H^0(C, \End(E)\tensor  K_C)^*.
\ee
Hitchin was  naturally led to considering the moduli space of 
pairs $(E,\phi)$, or 
the \emph{Higgs bundles}, to simultaneously deal with
all possible deformations of $E, \Sigma$, and $\cL$ over a fixed
base curve $C$.

Since vector bundles on curves of genus $0$ and $1$ behave
differently from the  curves of \emph{general type} $g=g(C)>1$, 
let us restrict ourselves to  the latter case for now. We also
restrict our attention to vector bundles $E$ of rank $n$ with the fixed
trivial determinant $\det(E)\isom \cO_C$. It is natural 
to restrict the endomorphism sheaf
for such a bundle to be \emph{traceless}, because of the
same reason of $\exp: sl_n(\bC)\lrar SL_n(\bC)$. 
We denote the sheaf of traceless endomorphism by $\End_0(E)$,
which is a rank $n^2-1$ locally free module over $\cO_C$.

The Riemann-Roch formula
\be
\label{RR}
\dim H^0(C,\End_0(E)) - \dim H^1(C,\End_0(E)) = -(g-1)(n^2-1)
\ee
tells us that the \emph{expected} dimension of the moduli space
of  vector bundles of the  trivial determinant 
is $(g-1)(n^2-1)$. Although the 
identity map of $E$ into itself is a non-trivial endomorphism, it
has a non-zero trace, hence it is not a section of $\End_0(E)$.
When do we have a vector bundle with non-trivial endomorphisms?
For example,  $E=\cO_C(m)\dsum \cO_C(-m)$ for a large $m>0$
has endomorphisms because
$$
\dim H^0(C,\End_0(E))=\dim \Hom(\cO_C(-m),\cO_C(m)) +1 = 2m-g+2.
$$
It then contributes to the dimension of the space of
infinitesimal deformations of $E$. Although these vector bundles appear
as  \emph{objects} 
of the moduli \textbf{stack} of vector bundles with trivial determinants,
they need to be excluded from the moduli \textbf{space}.
The \emph{slope stability condition} 
\be
\label{stable}
\frac{\deg(E)}{\rank(E)}> \frac{\deg(F)}{\rank(F)}
\ee
for every proper vector subbundle
$F\subset E$ is imposed for the construction of moduli spaces
to avoid such appearances of large endomorphisms.
For our case, since $E$ has $\deg(E) = 0$, it cannot have any 
subbundles of non-negative degrees. And the 
moduli space $\cS \cU (C,n)$ is smooth at a stable vector bundle 
$E$ with the tangent and cotangent spaces given by
$$
\begin{cases}
T_E \cS \cU (C,n) = H^1(C,\End_0(E))\\
T^*_E \cS \cU (C,n) =H^0(C,\End_0(E)\tensor K_C).
\end{cases}
$$

The Higgs field $\phi\in H^0(C,\End_0(E)\tensor K_C)$ functions as
a marking of the vector bundle $E$. We define an automorphism 
of $(E,\phi)$ to be an automorphism of $E$ that fixes  $\phi$. 
Thus a Higgs field makes the pair $(E,\phi)$ more stable. So
the slope stability of $E$ is modified for a Higgs bundle $(E,\phi)$, 
requiring that \eqref{stable} holds only for \emph{$\phi$-invariant}
subbundles $F$, meaning that $\phi$ maps $F$ to $F\tensor K_C$.

A concrete example better explains this effect. Since $\deg (K_C) = 2g-2$,
the canonical bundle has a square root $K_C^\half$. Actually, the $\pm$ sign choice
redundancy exists here, so there are $\left| H^0\big( C,\bZ\big/2\bZ\big)\right|=
2^{2g}$ different choices of the square root. Any one of them is called a
\emph{spin structure} of $C$, and it appears in Riemann's work
as a \emph{$\theta$-characteristic}. Choose a spin structure 
$K_C^\half$ on $C$ and  define $E=K_C^\half\dsum K_C^{-\half}$.
Clearly $E$ is not a stable vector bundle. Since
$\End(E) = K_C\dsum \cO_C\dsum \cO_C\dsum K_C^{-1}$, 
a non-zero quadratic differential 
$$
q\in H^0(C,K_C^{\tensor 2}) = \Hom\big(K_C^{-\half},K_C^{\half}
\tensor K_C\big)
$$ defines a non-trivial traceless Higgs field of $E$.
Similarly, the identity map 
$$
1\in H^0(C,K_C^{-1}\tensor K_C) = \Hom\big(K_C^{\half},K_C^{-\half}
\tensor K_C\big)
$$ is also a traceless Higgs field of $E$. Adding together, we have a non-trivial traceless Higgs field of $E$ constructed by
\be
\label{rank2section}
\phi = \begin{bmatrix}
0&&q\\
\\
1&&0
\end{bmatrix}:
\left(
\begin{matrix}
K_C^\half\\
\dsum\\
K_C^{-\half}
\end{matrix}
\right)
\lrar
\left(
\begin{matrix}
K_C^\half\\
\dsum\\
K_C^{-\half}
\end{matrix}
\right)
\tensor K_C.
\ee
None of the subbundles $K_C^{\pm\half}$ of $E$ are
invariant under this $\phi$. Hence $(E,\phi)$ is stable.
Being a split bundle, $E$ has a non-trivial automorphism 
for every choice of $1$-form $\omega\in H^0(C,K_C)$,
$$
u = \sqrt{-1}\begin{bmatrix}
1&&\omega\\
\\
0&&-1
\end{bmatrix}:
\left(
\begin{matrix}
K_C^\half\\
\dsum\\
K_C^{-\half}
\end{matrix}
\right)
\lrar
\left(
\begin{matrix}
K_C^\half\\
\dsum\\
K_C^{-\half}
\end{matrix}
\right),\qquad \det u = 1.
$$
But $u$ does not fix $\phi$, because
$u^*(\phi):=u^{-1}\circ \phi \circ u = \begin{bmatrix}
\omega&\omega^2-q\\
-1&-\omega
\end{bmatrix}\ne \phi
$. This means $u$ is not an automorphism of the Higgs bundle
$(E,\phi)$. We rather consider  $(E,\phi)$ and $(E,u^*(\phi))$
are \emph{isomorphic} as Higgs bundles.

Let us denote by $\cM(C,G)$ the moduli space of  
isomorphism classes of stable
$SL_n(\bC)$-Higgs bundles we have been discussing. Here,
$G$ is meant to be $G=SL_n(\bC)$, which can be generalized 
to other Lie groups. If $E$ is itself a stable vector bundle, then
$(E,\phi)$ is stable as a Higgs bundle for any Higgs field 
$\phi$. So we see that the cotangent bundle
of the moduli of stable bundles is included, as an open dense 
subset, in the moduli space of 
stable Higgs bundles:
$$
T^*\cS\cU(C,n)\subset \cM(C,G).
$$
The \emph{complex symplectic structure} on $\cM(C,G)$ is
constructed from this embedding by analyzing the codimension
of the complement of $T^*\cS\cU(C,n)$.

Here, we note the dimension, $\dim_\bC \cM\big(C, SL_n(\bC)\big)
=2(g-1)(n^2-1).$ There is another 
complex symplectic manifold of the same dimension, which appears to be
very different from the context of Higgs bundles. Let 
$$
\Hom\big(\pi_1(C), SL_n(\bC)\big)
$$
denote the space of representations of the fundamental group of $C$
into $SL_n(\bC)$. Slightly modifying the 
homology bases \eqref{homology}, we can choose a
homotopy basis for $\pi_1(C)$ and give a presentation 
\be
\label{pi1}
\pi_1(C) = \big\langle \a_1,\a_2,\dots,\a_g,\b_1,\b_2,\dots,\b_g\;\big|\;
[\a_1,\b_1]\cdot [\a_2,\b_2] \cdots [\a_g,\b_g]=1\big\rangle,
\ee
where $[\a,\b] = \a\b\a^{-1}\b^{-1}$ is the multiplicative commutator.

\begin{figure}[htb]
\includegraphics[width=3in]{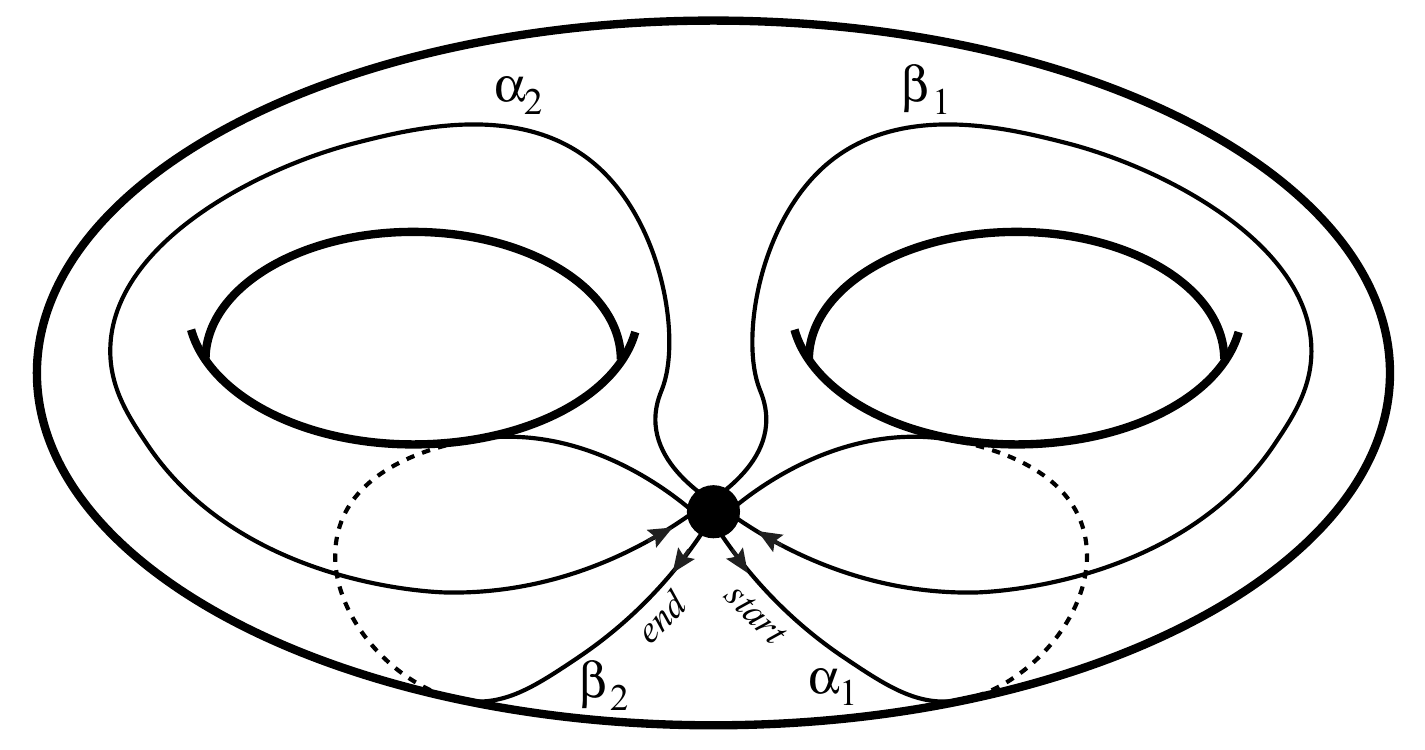}
\caption{A homotopy basis for $\pi_1(C)$ of a curve of genus $2$. $[\a_1,\b_1]\cdot
[\a_2,\b_2]=1$.}
\end{figure}

Since every representation of the fundamental group
 $\rho:\pi_1(C)\lrar SL_n(\bC)$
is determined by its values $a_i=\rho(\a_i)\in SL_n(\bC)$ and 
$b_i=\rho(\b_i)\in SL_n(\bC)$, and since
the commutator relation 
\be
\label{commutator}
[a_1,b_1]\cdot [a_2,b_2] \cdots [a_g,b_g]=1
\ee
is a polynomial 
equation with respect to the entries of the standard $n\times n$
matrix representation of $SL_n(\bC)$,
the representation space is realized as an affine algebraic variety
$$
\Hom\big(\pi_1(C), SL_n(\bC)\big) =
\big\{(a_1,b_1),\dots,(a_g,b_g)\in SL_n(\bC)^{2g}\;\big|\;
[a_1,b_1] \cdots [a_g,b_g]=1\big\}.
$$
For many purposes, we identify two representations when they 
have the same \emph{character}. In our geometric situation, 
two representations $\rho_1$ and $\rho_2$ are identified
if there is a group element $h\in SL_n(\bC)$ so that 
$\rho_1 = h^{-1}\rho_2 h$. This identification induces a
conjugation action of $SL_n(\bC)$ on the space of representations.
Thus we define the \textbf{character variety} 
of the fundamental group of $\pi_1(C)$ into a complex  algebraic 
group $G$ as the categorical 
quotient
\be
\label{charvar}
\rchi \big(\pi_1(C), G\big) = \Hom\big(\pi_1(C), G\big)
\big/\!\!\big/ G,
\ee
understanding that the $G$-action on the representation space is
by conjugation. The \emph{Geometric Invariant Theory} construction
\cite{MFK}
of this quotient makes the character variety a scheme
of dimension  $2(g-1)\dim_\bC G$, where $\dim_\bC G$ is subtracted
twice because of the commutator relation \eqref{commutator} and
the quotient of the conjugation action. 
This quotient is also identified as the \textbf{symplectic quotient}
(see the added chapter in the 3rd edition of \cite{MFK} by Kirwan), 
which makes the character variety a symplectic variety. 
For our case of $G=SL_n(\bC)$, we have
$$
\dim_\bC \rchi \big(\pi_1(C), SL_n(\bC)\big) = 2(g-1)(n^2-1).
$$

It was Narasimhan and Seshadri \cite{NS} who noticed the homeomorphism 
\be
\label{NS}
\rchi \big(\pi_1(C), SU(n)\big) \overset{\sim}{\lrar} \cS\cU(C,n),
\ee
where the character variety is a \emph{real} variety of 
dimension $2(g-1)(n^2-1)$, and $\cS\cU(C,n)$ is a complex
variety of dimension $(g-1)(n^2-1)$. Yet another space of this dimension
is the \emph{base}, whose dimension is calculated via Riemann-Roch again:
\be
\label{base}
B: = \bigoplus_{\ell=2}^n H^0\big(V, K_C^{\tensor \ell}\big),
\quad \dim_\bC \left(\sum _{\ell=2}^n \big(\ell(2g-2) -(g-1)\big)\right) =(g-1)(n^2-1).
\ee
This is the space we mentioned in Section~\ref{theme2},
\eqref{HitchinBase}. Now our \emph{numerology} is complete:
\begin{multline*}
\dim_\bR \rchi \big(\pi_1(C), SU(n)\big)=
2\dim_\bC \cS\cU(C,n) = 2 \dim_\bC B \\= \dim_\bC \cM\big(C,SL_n(\bC)\big)
= \dim_\bC \rchi \big(\pi_1(C), SL_n(\bC)\big) = 2(g-1)(n^2-1).
\end{multline*}
The complexification of Narasimhan-Seshadri \eqref{NS}
 has been established in its vast generality
(\cite{Cor1988, Don1987, H1,Sim1992,Sim1994a,Sim1994b, Sim1997}, see also
\cite{Mochizuki}).
In our context, it is a curve case of the
homeomorphism of Simpson, or \emph{non-Abelian
Hodge correspondence} 
\be
\label{nah}
\rchi \big(\pi_1(C), G\big) \approx \cM(C,G)
\ee
for a complex Lie group $G$. These are hyperk\"ahler manifolds 
and their complex structures differ by a hyperk\"ahler rotation. 
The complex structure of the
character variety has nothing to do with the complex structure of the
curve $C$.

Hitchin's discovery \cite{H1} includes the algebraically completely
integrable system on the moduli space $\cM\big(C,SL_n(\bC)\big)$ of 
Higgs bundles through the characteristic polynomial map to the base 
space $B$,
\be
\label{charmap}
\mu_H: \cM\big(C,SL_n(\bC)\big) \owns (E,\phi) \longmapsto \det(\lam I - \phi)\in B.
\ee
It is a generically Abelian variety fibration, and
each fiber  is a complex Lagrangian of the 
complex symplectic manifold $\cM\big(C,SL_n(\bC)\big)$. 
Since $T^*\cS\cU(C,n)\subset \cM\big(C,SL_n(\bC)\big)$,
each cotangent space $T_E^*\cS\cU(C,n)$ at a stable vector bundle $E$
is also a complex Lagrangian of  $\cM\big(C,SL_n(\bC)\big)$. 

There is another complex Lagrangian, which plays a
critically important role in our current investigation. 
It is obtained by a particular section 
$\kappa: B\hookrightarrow \cM\big(C,SL_n(\bC)\big)$
of the map $\mu_H$. Let us review the construction 
of this section, called the \emph{Hitchin section}, 
 utilizing Kostant's 
\emph{principal three-dimensional subgroup}
(TDS)
\cite{Kostant}. We first choose a spin structure 
$K_C^\half$, and take an arbitrary point
$
\mathbf{q}=(q_2,q_3,\dots,q_n)
\in B$
of the  base $B$.
Define three $n\times n$ matrices in $sl(n,\bC)$ by
\be
\label{TDS}
\begin{aligned}
&X_- := \big[ \sqrt{s_{i-1}} \delta_{i-1,j}\big]_{ij}
=\begin{bmatrix}
0&0&\cdots&0&0\\
\sqrt{s_1}&&&&0\\
0&\sqrt{s_2}&&&0\\
\vdots&&\ddots&&\vdots\\
0&0&\cdots&\sqrt{s_{n-1}}&0
\end{bmatrix},
\qquad  X_+ := X_-^t,
\\
& H :=[X_+,X_-],
\end{aligned}
\ee
where $s_i := i(n-i)$.
$H$ is a diagonal matrix whose 
$(i,i)$-entry is $s_i-s_{i-1}
= n - 2i+1$, with $s_0=s_r=0$. 
The Lie subalgebra generated by $X_+,X_-,H$ is isomorphic to 
$sl(2,\bC)$, and  is called  
a \emph{principal TDS} of $SL(r,\bC)$. 

\begin{Def}
\label{Hitchin section}
The \emph{Hitchin section} is defined as the set
$\left\{\big(E_0,\phi(\mathbf{q})\big)\;\big|\;\mathbf{q}\in B\right\}$
consisting of Higgs bundles, with the fixed vector bundle 
\be
\label{E0}
E_0 := \left(K_C^\half\right)^{\tensor H}
=\bigoplus_{i=1}^n \left(K_C^\half\right)^{\tensor (n-2i+1)}
=\bigoplus_{\ell=0}^{n-1} K_C^{-\frac{n-1}{2}}\tensor K_C^{\tensor \ell},
\quad(\ell = n-i)
\ee
and varying Higgs fields  with parameter $\mathbf{q}$:
\be\label{phi q}
\phi(\mathbf{q}): = 
X_- + \sum_{\ell = 2}^n q_\ell X_+ ^{\ell-1}.
\ee
\end{Def}

Each Higgs bundle $\big(E_0,\phi(\mathbf{q})\big)$
as constructed is stable.
The Hitchin section 
$$
\kappa: B\owns \bq
\longmapsto \big(E_0,\phi(\mathbf{q})\big)
\in  \cM\big(C,SL_n(\bC)\big)
$$
is a biholomorphic map from 
$B$ to $\kappa(B)\subset \cM\big(C,SL_n(\bC)\big)$. 
The example $\left(K_C^\half\dsum K_C^{-\half}, \phi\right)$
of \eqref{rank2section} is the rank $2$ case of the
Hitchin section.

\subsection{Higher order linear
differential operators on a curve $C$}
\label{subsect:higherorder}

\begin{quest}
What is a linear differential operator of order $n\ge 2$ globally
defined on a compact smooth algebraic curve $C$?
\end{quest}

If $\o$ is a global holomorphic 
$1$-form on $C$, then 
$
(d+\o)\psi=0
$
is a first order linear differential equation globally defined on $C$.
But we have no globally defined second order differential operator,
since $d^2 = 0$. Once again let us go back to 
\eqref{nth order}, and consider the second order case first:
\be\label{2nd}
\left(\frac{d^2}{dz^2} -q(z)\right)\psi(z) = 0.
\ee
We need the following:
\bite
\item To find a  coordinate system on the curve $C$ that allows 
the expression \eqref{2nd} stays globally  the same; and
\item To show that every complex structure of $C$ admits  such a 
coordinate system.
\eite
We consider $\psi(z)$ of \eqref{2nd} to be a section of some
unknown  line bundle $L$ with a transition function $e^{g(w)}$ such that
$e^{-g(w)} \psi(w) = \psi(z)$. The second order differential operator acts on $L$  
and change its section  into  a section of $K_C^{\tensor 2}\tensor L$. 
Therefore, we wish to solve the equation
\be\label{z=w}
\left(\left(
\frac{d}{dz}\right)^2-q(z)\right)\psi(z) = 
\left(\frac{dw}{dz}\right)^2e^{-g(w)} \left(\left(
\frac{d}{dw}\right)^2-q(w)\right)\psi(w).
\ee
Here, the unknowns are the coordinate change $w=w(z)$ as a function in $z$, 
and
the  function $g(w)$.  It is natural to assume that 
the $0$-th order term $q$ of the differential operator
satisfies $q(z)dz^2 = q(w)dw^2$,  i.e.,  
it is  a quadratic differential $q\in K_C^{\tensor 2}$.
With these considerations, we can simplify \eqref{z=w} and obtain
\be
\begin{aligned}
\label{2nd calculation}
&
\left(\frac{d}{dz}\right)^2
=
\left(\frac{dw}{dz}\right)^2
e^{-g(w)}\left(\frac{d}{dw}\right)^2e^{g(w)}
\\
\Longleftrightarrow \quad
&\frac{d}{dz}\left(\frac{dw}{dz}\frac{d}{dw}\right)
= 
\left(\frac{dw}{dz}\right)^2\left(\left(\frac{d}{dw}\right)^2
+ 2g_w(w)\frac{d}{dw} + g_{ww}(w)+g_w(w)^2\right)
\\
\Longleftrightarrow \quad
&w'' \frac{d}{dw} = 2(w')^2g_w\frac{d}{dw} + g_{ww} + g_w^2
\\
\Longleftrightarrow \quad
&\begin{cases}
2g_w w'= \frac{w''}{w'}\\
g_{ww} + g_w^2=0\;
\end{cases}
\Longleftrightarrow \quad
\begin{cases}
g(w)'= \half(\log w')'\\
g_{ww} + g_w^2=0\;,
\end{cases}
\end{aligned}
\ee
where $w' = \frac{dw}{dz}$ and $g_w=\frac{dg}{dw}$.
The first line of the final equations implies
$$
e^{-g(w)}= \sqrt{\frac{dz}{dw}}
$$
because $g(w)=0$ when $w(z)= z$.
Therefore,
the line bundle is $L\isom K_C^{-\half}$. And the second line, 
$g_{ww} + g_w^2=0$, after substituting 
$g_w= \half \frac{w''}{(w')^2}$ into it,
gives us exactly $s_z(w)=0$,
where 
 \be
 \label{schwarz}
 s_z(w):=\left(\frac{w''(z)}{w'(z)}\right)'-\frac{1}{2}\left(
\frac{w''(z)}{w'(z)}\right)^2
 \ee
is the   \textbf{Schwarzian derivative}.
We learn from Gunning \cite{Gun} that every complex structure
of $C$ admits a \emph{projective coordinate system},
and that the transition function of this coordinate system
\be
\label{projective}
w(z) = \frac{az+b}{cz+d}, \qquad 
\begin{pmatrix}
a&b\\
c&d
\end{pmatrix}\in SL(2, \bC),
\ee
satisfies the equation $s_z(w)=0$.

The calculation of \eqref{2nd calculation} can be 
extended to an operator $P$ of order $n$. The comparison of  the $(n-1)$-order terms in
\be
\label{n}
\left(\frac{d}{dz}\right)^n
=
\left(\frac{dw}{dz}\right)^n
e^{-g(w)}\left(\frac{d}{dw}\right)^ne^{g(w)}
\ee                                                                                                                                                                                                                                                                                                                                                                                                                                                                                                                                                                                                                                                                                                                                                                                                                                                                                                                                                                                                                                                                                                                                                                                                                                                                                                                                                                                                                                                                           
gives
$$
\binom{n}{2}(w')^{n-2}w''\left(\frac{d}{dw}\right)^{n-1} 
=n (w')^n g_w \left(\frac{d}{dw}\right)^{n-1}
\Longleftrightarrow 
\quad g(w)' = \frac{n-1}{2}(\log w')'.
$$
Hence the line bundle on which $P$ acts is $K_C^{-\frac{n-1}{2}}$.
And when we use the projective coordinate system, from \eqref{projective}
we find  $\frac{dw}{dz} = \frac{1}{(cz+d)^2}$. Therefore, 
the transition function for the canonical sheaf is $(cz+d)^2$, and with a 
choice of a spin structure on $C$, we can use $cz+d$ as the transition function
of $K_C^\half$. 
Then \eqref{n} should give 
\be
\begin{aligned}
\label{Kn}
&\left(\frac{d}{dz}\right)^n
=
\left(\frac{dw}{dz}\right)^n
e^{-g(w)}\left(\frac{d}{dw}\right)^ne^{g(w)}
\\
\Longleftrightarrow \quad
& \left(\frac{d}{dz}\right)^n
=
(cz+d)^{-n-1}\left(\frac{d}{dw}\right)^n(cz+d)^{-n+1}
\\
\Longleftrightarrow \quad
&\left(\frac{d}{dw}\right)^n
=
(cz+d)^{n+1}\left(\frac{d}{dz}\right)^n(cz+d)^{n-1},
\end{aligned}
\ee
which is easy to check by induction on $n$. 

The differential equation $P\psi = 0$ makes sense globally on $C$ 
if $\psi$ is a section of $K_C^{-\frac{n-1}{2}}$, and $P$ changes it
to a section of $K_C^{\tensor n}\tensor K_C^{-\frac{n-1}{2}}
= K_C^{\frac{n+1}{2}}$. In terms of the projective coordinates, 
we start with a section $\psi\in \tensor K_C^{-\frac{n-1}{2}}$, which satisfies
$\psi(w) = (cz+d)^{-(n-1)}\psi(z)$. Then the third line of \eqref{Kn} gives
\begin{align*}
\psi^{(n)}(w):= \left(\frac{d}{dw}\right)^n \psi(w)
&=
(cz+d)^{n+1}\left(\frac{d}{dz}\right)^n(cz+d)^{n-1}
(cz+d)^{-(n-1)}\psi(z)
\\
&=(cz+d)^{n+1}\left(\frac{d}{dz}\right)^n\psi(z) = (cz+d)^{n+1}\psi^{(n)}
(z).
\end{align*}
Thus the $n$-derivative $\psi^{(n)}$ globally makes sense
as a section of  $K_C^{\frac{n+1}{2}}$ in this
projective coordinate system.

In conclusion, we have verified that a locally given 
$n$-th order differential operator $P$ of \eqref{nth order}
can globally be extended on the curve $C$
if we use a projective coordinate system on 
$C$,  and if $P$ is considered to be a $\bC$-linear map
$P:K_C^{-\frac{n-1}{2}}\lrar K_C^{\frac{n+1}{2}}$.

\subsection{Spectral curves and opers}
\label{subsect:oper}

\begin{quest}
An $n$-th oder globally defined  linear differential equation $P\psi=0$ on $C$ is
equivalent to a system of $n$ first order differential
equations on every local neighborhood $U\subset C$.
Does this equivalence globally hold? And if so, what is the 
condition for the system of first order differential equations?
\end{quest}

The \emph{opers} \cite{BD, BZF} give an answer. 
Our starting point is the 
local connection $\nabla_z$ of \eqref{introconn}. Since $\psi\in 
K_C^{-\frac{n-1}{2}}$, and its $n$-th derivative 
$\psi^{(n)}\in K_C^{\frac{n+1}{2}}$, the vector $\Psi$ of
\eqref{introconn} on which $\nabla_z$ acts is a local section of 
$$
\bigoplus_{\ell = 0}^{n-1} K_C^{-\frac{n-1}{2}}\tensor K_C^{\tensor \ell},
$$
which is the same as $E_0$ of \eqref{E0}. And the matrix 
expression of $\nabla_z$ of  \eqref{introconn} exhibits a similarity with 
the Higgs field $\phi(\mathbf{q})$ of \eqref{phi q}. Here, 
\emph{similarity} means that the connection matrix of \eqref{introconn}
and the matrix of \eqref{phi q} have the same characteristic polynomial, 
which is a point of the \emph{base} $B$.
However, we learn from Atiyah \cite{Atiyah} that
 there is no holomorphic connection on this particular vector 
 bundle $E_0$. 
We need to identify the vector bundle $E$ such that the
differential operator $\left(P, K_C^{-\frac{n-1}{2}}\right)$ is equivalent to 
a connection $(E,\nabla)$, globally on $C$.  The \emph{equivalence} 
we want to achieve here is the obvious
local equivalence between  \eqref{nth order} and \eqref{introconn}. 
In the local expression
\be
\label{Plocal}
P = \left(\frac{d}{dz}\right)^n + 
\sum_{k=2}^n a_k \left(\frac{d}{dz}\right)^{n-k},
\ee
the coefficients are $a_k\in H^0\big(C,K_C^{\tensor k}\big)$.
Therefore, the dimension of the space of globally defined differential 
operator of order $n$ of the  shape \eqref{Plocal}
is equal to $\dim B = (g-1)(n^2-1)$.

For the purpose of analyzing the situation, 
let us go back to the consideration of the second order case \eqref{2nd} again. 
The operator $P=(d/dz)^2+q(z)$ depends only on $q\in H^0\big(C,K_C^{\tensor 2}\big)=B$ in this case. 
As in Section~\ref{subsect:higherorder}, we use local coordinates
$w$ and $z$ connected by a fractional linear transformation 
\eqref{projective}. The spin structure $K_C^\half$ we choose corresponds
to the transition function $\xi=\xi_{wz} = cz+d$.

\begin{prop}[Section~3.1 of \cite{OM5}]
\label{2oper}
Choose an element $\hbar\in H^1(C,K_C)\isom \bC$. Then
the Higgs bundle on the Hitchin section
\be
\label{2Higgs}
(E_0,\phi) = \left(K_C^\half\dsum K_C^{-\half}, 
\begin{bmatrix} 
&q\\
1&
\end{bmatrix}
\right)
\ee
automatically selects a unique $\hbar$-connection $\big(E^\hbar, \nabla^\hbar\big)$
on the extension bundle 
\be
\label{ext}
\begin{CD}
0@>>> K_C^\half @>>>E^\hbar@>>>K_C^{-\half}@>>>0
\end{CD}
\ee
determined by $\hbar\in \Ext^1\big(K_C^{-\half}, K_C^{\half}
\big)$. Here, the $\hbar$-connection is given by 
\be
\label{hbar-conn}
\nabla^\hbar = \hbar\,d + \begin{bmatrix} 
&q\\
1&
\end{bmatrix}.
\ee
\end{prop}

\begin{rem}
The surprise here is in the connection matrix
\eqref{hbar-conn}. It is identical to 
the Higgs field $\phi$ of \eqref{2Higgs}! But this is exactly what 
we wanted.  
\end{rem}

\begin{proof}
The transition function for the vector bundle $E_0$ is
$\begin{bmatrix} 
\xi_{wz}\\
&\xi_{wz}^{-1}
\end{bmatrix}
$, and the Higgs field $\phi$ transforms
$$
\begin{bmatrix} 
&q(w)\\
1&
\end{bmatrix}dw = 
\begin{bmatrix} 
\xi_{wz}\\
&\xi_{wz}^{-1}
\end{bmatrix}
\begin{bmatrix} 
&q(z)\\
1&
\end{bmatrix}dz
\begin{bmatrix} 
\xi_{wz}\\
&\xi_{wz}^{-1}
\end{bmatrix}^{-1},
$$
because $\phi$ is a matrix valued $1$-form. 
We note that $dw/dz = \xi_{wz}^{-2}=1/(cz+d)^2$.
As calculated in \cite[Section~3.1]{OM5}, 
the extension group $\Ext^1\big(K_C^{-\half}, K_C^{\half}
\big)$ is generated by the cohomology class $\frac{d}{dz}\xi_{wz} = c$,
which is the image of the line bundle $K_C^\half$ under the cohomology 
map
$$
\begin{CD}
\cdots @>>> H^1\big(C, \cO_C^*\big) @>>> H^1\big(C,K_C\big)@>>>
\cdots 
\end{CD}
$$
associated with the short exact sequence $0\lrar \bC^*\lrar 
\cO_C^*\lrar K_C\lrar 0$. The extension \eqref{ext}
is then defined by a matrix 
$\begin{bmatrix} 
\xi_{wz}&\hbar\, c\\
&\xi_{wz}^{-1}
\end{bmatrix}
$, and the $\hbar$-connection $\nabla^\hbar$ satisfies the desired transition relation
$$
\hbar \,d+\begin{bmatrix} 
&q(w)\\
1&
\end{bmatrix}dw = 
\begin{bmatrix} 
\xi_{wz}&\hbar \, c\\
&\xi_{wz}^{-1}
\end{bmatrix}
\left(\hbar \,d+\begin{bmatrix} 
&q(z)\\
1&
\end{bmatrix}dz
\right)
\begin{bmatrix} 
\xi_{wz}&\hbar \, c\\
&\xi_{wz}^{-1}
\end{bmatrix}^{-1}.
$$
\end{proof}

The precise cancellation here makes us feel miraculous. 
Definitions of a few terminologies are due here. Deligne's $\hbar$-connection
$\nabla^\hbar$ is a linear differential operator 
$\nabla^\hbar:E\lrar E\tensor K_C[[\hbar]]$ on a vector bundle $E$
that satisfies 
$$
\nabla^\hbar(fs) = s\tensor \hbar\,df + f\nabla^\hbar (s).
$$
The restriction $\nabla^\hbar\big|_{\hbar = 0}$ is a Higgs field, 
and $\nabla^\hbar\big|_{\hbar = 1}$ is a connection in the usual sense.
The extension $E^\hbar$ of \eqref{ext} comes with a natural 
$3$-term \emph{filtration}
$0\subset \cF^1 \subset E^\hbar$ that satisfies a  
condition $\cF^1 \isom \big(E^\hbar\big/\cF^1\big) \tensor K_C$. 
More generally, a connection $(\nabla, E)$ on a vector bundle
$E$  of rank $n$ is an \textbf{oper} if there is a full length
filtration
$$
0=\cF^n\subset \cF^{n-1}\subset \cF^{n-2}\subset\cdots\subset \cF^0=E
$$
that satisfies the \textbf{Griffiths transversality} 
$$
\nabla\big|_{\cF^i}: \cF^i\lrar \cF^{i-1}\tensor K_C, \quad i=1, 2, \dots, n,
$$
and the $\cO_C$-module isomorphism of graded objects
$$
\overline{\nabla}: \cF^i\big/\cF^{i+1}\overset{\sim}{\lrar}
\big(\cF^{i-1}\big/\cF^{i}\big)\tensor K_C, \quad i=1, 2, \dots, n-1,
$$
induced by the connection $\nabla$. The concept can be
extended to the case of $\hbar$-connections.

The discovery of \cite[Section~3.1]{OM5} is a full generalization
of Proposition~\ref{2oper} to the case of the Hitchin section
on $\cM\big(C,SL_n(\bC)\big)$.

\begin{thm}[\cite{OM5}, Theorem 3.8, Theorem 3.10, Theorem 3.11]
Let $C$ be a smooth projective algebraic curve of genus $g(C)>1$.
Choose a projective coordinate system on $C$, and a spin
structure $K_C^\half$. Consider the characteristic polynomial map
\eqref{charmap} on $\cM\big(C,SL_n(\bC)\big)$.
Recall that every point $p\in B$ of
the base $B$ represents a spectral curve $\Sigma\subset T^*C$. 
Then every spectral curve uniquely determines an $\hbar$-family
of opers $\big(E^\hbar, \nabla^\hbar\big)$, which corresponds to an $\hbar$-family of globally
defined differential operators $P^\hbar$ of order $n$. 
The semi-classical limit of $P^\hbar$ recovers the spectral curve 
$\Sigma$. Moreover, the map from $p\in B$ to 
the corresponding connection $\big(E^{\hbar=1}, \nabla^{\hbar=1}\big)$
is a biholomorphic map between $B$ and the moduli space
of opers in the character variety $\rchi\big(\pi_1(C),SL_n(\bC)\big)$, 
which is a complex Lagrangian.
\end{thm}

\section{In search of the geometry that knows the irrationality of 
$\zeta(3)$}
\label{zeta3}

\subsection{The Mystery Formula} 
\be
\begin{aligned}
\label{mystery}
(-1)^n (n!)^2 &\sum_{\substack{\ell+m=n\\
\ell,m\ge 0}}
\frac{(2\ell+m)! (\ell+2m)!}
{(\ell !)^5  (m!)^5}\bigg(1+(m-\ell)\big(H_{2\ell+m} 
+2H_{\ell + 2m} -5 H_m\big)\bigg)
\\
&=\sum_{\ell=0}^n\binom{n}{\ell}^2 \binom{n+\ell}{\ell}^2, \qquad 
n\ge 0.
\end{aligned}
\ee
Ap\'ery's astonishing discovery of \cite{A} announced in 1978, 
that proves
the irrationality of a special value $\zeta(3)$ of the Riemann zeta function, 
caused a huge sensation (see \cite{vdP}). It was followed by an almost
immediate, and unexpected, geometric surprise \cite{B, BP} that 
the generating function of the key integer sequence $A_n$  of 
Ap\'ery,
the second line of \eqref{mystery}, solves a 
Picard-Fuchs  differential equation 
associated with a particular family of K3 surfaces. 
Since last decade, there has been a renewed sensation
\cite{Gal, GKS, GZ, MS2024, Zagier} on this topic.
This time it is due to the identification that the  Borel-Laplace transform of the
same generating function   is  \textbf{identical} to the 
generating function of the 
genus $0$,  $1$-marked point, degree $d\ge 2$
Gromov-Witten invariants of  a 
Fano $3$-fold described below, after adjusting the exponential factor 
$e^{5t}$.
  This is the content of  the Mystery Formula \eqref{mystery}.

Following
\cite{CCGK}, let us introduce the 
Fano $3$-fold $V_{12}$:

\begin{Def}[\cite{CCGK}, p.139. Fano $3$-fold $V_{12}$]   
We consider a rank $3$ vector bundle 
$\cE=(\cU^*\tensor \det \cU^*)\dsum \det \cU^*$ defined on the
$6$-simensional
Grassmannian $Gr(2,5)$, where $\cU$ is the universal bundle.
Take  a generic section
$
\sigma
\in H^0\big(Gr(2,5), \cE)
$
of  $\cE$.  The Fano $3$-fold $V_{12}$ is the zero locus $[\sig]_0$ of this
section. The genericness condition assures  smoothness of the zero locus. 
\end{Def}

\begin{thm}[\cite{CCGK}, p.141] Let $X$ be a non-singular model
for the Fano $3$-fold $V_{12}$, and $-K_X$ its anticanonical divisor.
Define the anticanonical degree $d\ge 2$, type $(g,n)=(0,1)$, Gromov-Witten invariant of $X$ by
\be
\label{gam}
\gam_d:=\sum_{\substack{\beta\in H_2(X,\bZ)\\
d=\la \beta, -K_X\ra}}\int_{\Mbar_{0,1}(X, \b)} ev^*(pt) \psi^{d-2},
\ee
where 
\bite
\item
$\Mbar_{0,1}(X,\b)$ is the moduli stack of holomorphic maps
$f:(C,\infty)\lrar X$ from a genus $0$ curve $C$ with a marked point
$\infty\in C$ to $X$ such that the homology class of $f(C)$ is equal to $\beta$
in $H_2(X,\bZ)$;
\item $ev: \Mbar_{0,1}(X,\b)\owns f\longmapsto
f(\infty)\in X$ is the evaluation map;
\item $ev^*: H^\bullet (X,\bQ)
\lrar H^\bullet \big(\Mbar_{0,1}(X, \b),\bQ)$ is the induced cohomology 
map; and 
\item $\psi = c_1(\bL)$ is the first Chern class of the tautological 
line bundle $\pi: \bL\lrar \Mbar_{0,1}(X,\b)$
whose fiber at  $f\in \Mbar_{0,1}(X, \b)$ is $T^*_{f(\infty)} f(C)$.
\eite
Then the generating function of these type $(0,1)$ Gromov-Witten invariants 
is given by
\be
\begin{aligned}
\label{GWFano}
GW_{0,1}(X):=
\sum_{d=0}^\infty \gam_d\; t^d 
= \; e^{-5t} 
&\sum_{n=0}^\infty 
(-1)^n n! 
\sum_{\substack{\ell+m=n\\
\ell,m\ge 0}}
\frac{(2\ell+m)! (\ell+2m)!}
{(\ell !)^5  (m!)^5}
\\
\times &\bigg(1+(m-\ell)\big(H_{2\ell+m} 
+2H_{\ell + 2m} -5 H_m\big)\bigg) \;t^n.
\end{aligned}
\ee
\end{thm}

\begin{rem}
From the expected dimension $\dim \Mbar_{0,1}(X, \b)=d+1$,  we 
normalize $\gam_0=1$ and $\gam_1=0$. Only $\gam_2$ has a direct 
geometric counting interpretation.
The exponential factor is chosen to assure $\gam_1=0$.
\end{rem}

We recall that topological recursion 
for simple Hurwitz numbers and 
higher genus Catalan numbers, as we have seen in earlier sections,
 are uniformly formulated on the
corresponding spectral curves, and the spectral curves are exactly 
the generating functions of the $(g,n)=(0,1)$-invariants. 
The role of topological recursion is to produce all $(g,n)$ invariants
from the spectral curve through a universal recursive procedure. 

From the definition (see below), the generating function of the Ap\'ery sequence 
$$
A(t) = \sum_{n=0}^\infty A_n t^n
$$
automatically satisfies a linear ODE with regular singular points. 
There is a  $3$-dimensional
 integral expression of $\zeta(3)$ making it a \textbf{period} 
 in the sense of Kontsevich-Zagier \cite{KZ},
which provides an illuminating  geometric understanding
of   Ap\'ery's proof. 
And
this integration formlula, modified with Ap\'ery sequences, 
  makes the direct passage toward the 
Gauss-Manin connection on the K3 fibration.
Right after Ap\'ery's discovery, it was pointed out that
the differential equation for $A(t)$ is identified to the Picard-Fuchs 
differential equation associated with a $1$-dimensional 
family of  K3 surfaces \cite{B, BP}, as mentioned above.

Now, researchers are wondering (e.g.,~\cite{Zagier}), isn't this pair, a Fano $3$-fold and
a family of K3 surfaces, a \emph{mirror symmetric pair}?
Our interest is in the following question, hoping that something
in the line of Section~\ref{subsect:Catalan} may come up:

\begin{quest}
\bite
\item Does Ap\'ery's $A(t)$ define a \textbf{spectral curve}?
\item If so, does the
topological recursion formulated on this spectral curve
generate all Gromov-Witten invariants of $V_{12}$ for arbitrary $(g,n)$? 
\item 
What is its quantum curve? 
\item What geometry does this 
quantum curve tell us?
\item Is the asymptotic solution  to the quantum curve,
considered as a Schr\"odinger equation,
at the irregular singular point a $\tau$-function of 
an integrable system, such as the KP equations?

\eite
\end{quest}

Nothing precise is known at this moment. The \emph{Quantum 
Differential Equation} (see \cite{I}) for the Fano variety $V_{12}$ is identified
in \cite{CCGK} (see below, Theorem~\ref{QDE}). The quantum 
curve in question is an $\hbar$-deformation family of QDE. 
This point relates to the  work \cite{DFKMMN} on Gaiotto conjecture 
\cite{Gai2014} discussed in Section~\ref{sect:var}.

\subsection{A quick review of Ap\'ery's ideas}
\label{subsect:Apery}

Ap\'ery proposed a recursion formula
\be
\label{Apery}
(n+1)^3 u_{n+1} -(34n^3+51n^2+27n+5)u_n +n^3 u_{n-1}=0, \qquad n\ge 1.
\ee
Following Ap\'ery \cite{A}, 
let us define two solutions of \eqref{Apery} and their generating functions: 
\begin{align}
\label{A}
& \{A_n\}_{n=0}^\infty\;, \qquad A_0=1, \quad A_1=5, 
\qquad A(t):=\sum_{n=0}^\infty A_n t^n, \qquad \a(t):=\sum_{n=0}^\infty 
\frac{A_n}{n!} t^n,\\
 \label{B}
&\{B_n\}_{n=0}^\infty\;, \qquad B_0=0, \quad B_1=6,
\qquad B(t):=\sum_{n=0}^\infty B_n t^n,\qquad \b(t):=\sum_{n=0}^\infty 
\frac{B_n}{n!} t^n\;.
\end{align}
The space of solutions of \eqref{Apery} has dimension $2$, spanned by 
$ \{A_n\}_{n=0}^\infty$
and  $\{B_n\}_{n=0}^\infty$. First few terms of $B_n$ are:
$0,6,\frac{351}{4},\frac{62531}{36}, \frac{11424695}{288},
\frac{35441662103}{36000},\frac{20637706271}{800}, 
\frac{963652602684713}{1372000},\dots.$ 
Closed formulas for Ap\'ery's  sequences are established:
\begin{align}
\label{An}
A_n &=\sum_{\ell=0}^n\binom{n}{\ell}^2 \binom{n+\ell}{\ell}^2, \quad n\ge 0,
\\
\label{Bn}
B_n &=\sum_{\ell=0}^n\binom{n}{\ell}^2 \binom{n+\ell}{\ell}^2
\left(\sum_{m=1}^n\frac{1}{m^3} 
- \sum_{m=1}^\ell (-1)^m\frac{1}{2m^3 \binom{n}{m}\binom{n+m}{m}}
\right), \quad n\ge 1.
\end{align}
Partial sums of $\zeta(3)=\sum_{m=1}^\infty\frac{1}{m^3}$ are appearing
here.
Notice  that $\{A_n\}$ is a natural basis for the space of solutions
because it is unique, up to a constant multiplication, that satisfies
\eqref{Apery} for \textbf{all} $n\ge 0$. 

Re-interpreting the result of  \cite{CCGK, GZ, Zagier}, 
the Mystery Formula \eqref{mystery} has a geometric
expression:
\be
\label{a=GW}
\a(t) = e^{5t}\cdot GW_{0,1}(X).
\ee
Notice that $GW_{0,1}(X)$ does not have a linear term in $t$, while 
$A_1=5$. This requires the adjustment of $e^{5t}$ in \eqref{a=GW}.
Some of  these amazing formulas are proved (relatively) easily
through differential equations established in \cite{B, BP,CCGK,GZ}.
Direct calculations remain to be very hard.

To see how fast the sequence determined by \eqref{Apery} grows or decreases, 
put $u_n \sim n^\a C^n$ with  constants $C$ and $\a$.
For large $n$, we then have
$$
\frac{(n+1)^{3+\a}}{n^{3+\a}}  C^2
-\left(34+\frac{51}{n}\right)C + \frac{(n-1)^\a}{n^\a} \sim 0.
$$
At $n\rar \infty$ we find  $C^2-34C+1=0$. Define $C=17+12\sqrt{2}$. 
Then the other solution of this quadratic equation is $C^{-1}$. 
Since $C^2=34C-1$, the sub order terms give us
$$
\frac{3+\a}{n}C^2 -\frac{51}{n}C -\frac{\a}{n}\sim 0
\Longrightarrow \a = -\frac{51C-3}{34C-2}=-\frac{3}{2}.
$$
Therefore, every solution of \eqref{Apery} has a growth behavior
\be
\label{growth}
u_n \sim \const\; n^{-\frac{3}{2}} \; C^{\pm n}.
\ee
In particular, both $A_n$ and $B_n$ grow 
exponentially fast, $\sim n^{-\frac{3}{2}}C^n$,  as $n\rar\infty$.
Since the space of solutions has dimension $2$, 
there is a unique solution, again up to a constant multiple, 
that decays exponentially. This means there exists a
unique constant $\zeta\in \bR$ such that 
$$
\lim_{n\rar \infty} \big(A_n\zeta-B_n\big) = 0.
$$
Since $\{A_n\}$ is unique up to constant multiple, and $\{B_n\}$ is canonical
up to constant multiple
by choosing $B_0=0$ to 
make it linearly independent with $\{A_n\}$, the sequence \eqref{Apery}
\emph{knows} this constant $\zeta$.

The other surprise, the relation to the Gauss-Manin connection
of \cite{BP}, starts with 
the integral expression
$$
\zeta(3) = \frac{1}{2}\int_0^1\int_0^1\int_0^1 \frac{dxdydz}{1-z+xyz}.
$$
Then it is proved that
\be
\label{B/A}
A_n\zeta(3) -B_n = \frac{1}{2}\int_0^1\int_0^1\int_0^1
\left(\frac{xyz(1-x)(1-y)(1-z)}{1-z+xyz}\right)^n
 \frac{dxdydz}{1-z+xyz}. 
\ee
These integral expressions show that $A_n\zeta(3) -B_n$ is a
\emph{period} of \cite{KZ} for every $n\ge 0$, that includes $\zeta(3)$. 
Since
$$
0< \left(\frac{xyz(1-x)(1-y)(1-z)}{1-z+xyz}\right)^n<1
$$
in the domain of integration $(x,y,z)\in (0,1)^3$, we immediately see that 
$$
\lim_{n\rar \infty} \big(A_n\zeta(3)-B_n\big) = 0.
$$
The mysterious constant $\zeta$ is indeed $\zeta(3)$. 
We know  $A_n$ and $B_n$ grow exponentially
fast, while the linear combination $A_n\zeta(3)-B_n$
converges to $0$. It is also a solution of the 
recursion \eqref{Apery}. Therefore, it decreases 
exponentially fast!
\be
\label{z(3)estimate}
A_n\zeta(3)-B_n \sim \const \; n^{-\frac{3}{2}}C^{-n}
\Longrightarrow \left|\zeta(3)-\frac{B_n}{A_n}\right|\sim \const \; C^{-2n},
\quad n\rar \infty.
\ee

Let us recall a criterion for irrationality of a real number.
\begin{lem}
\label{irr}
A positive real number $r>0$ is irrational if there is an $\epsilon>0$
such that  infinitely many relatively prime 
positive pairs of integers $p,q$ satisfy 
$$
0<\left|r-\frac{p}{q}\right|< \frac{1}{q^{1+\e}}.
$$
\end{lem}

This means if a real number is approximated too well by a
sequence of rationals, then it cannot be rational. 

\begin{proof}
Suppose $r = a/b$ with integers $a,b>0$. Then for any
pair of integers $p,q>0$ such that $r\ne p/q$, we have
$$
\frac{1}{q^{1+\e}} > \left|r-\frac{p}{q}\right| = 
\left|\frac{a}{b}-\frac{p}{q}\right| =\left|\frac{aq-bp}{bq}\right|
\ge \frac{1}{bq}, \quad \text{hence}\quad 0<q<b^{1/\e}.
$$
There are only a finite number of such integers. 
\end{proof}

To appeal to this Lemma, Ap\'ery estimated the denominator of $B_n$. 
Let 
$$
d_n:= \lcm[1,2,3,\dots,n].$$
 He found that  $d_n^3 B_n\in \bZ$, which follows from \eqref{B}.
The recursion \eqref{Apery}
suggests that the denominator of a solution $u_n$ starting with 
integral initial values $u_0, u_1\in \bZ$ is about $(n!)^3$. 
The denominator of $B_{30}$ is  $\sim 1.9\times 10^{33}$, while
$\lcm[1,\dots,30]^3$ divided by this denominator is $6630$. 
So the above estimate is quite accurate. 

\begin{rem}
Of course what is truly astonishing is that the sequence $A_n$ is an 
\textbf{integer}
sequence. Zagier \cite{Zagier} (see also \cite{KZ}) discusses how special and
rare such a  phenomenon is. 
\end{rem}

Now we know that the denominator of the rational number 
${B_n}/{A_n}$ is estimated to be 
$$
d_n^3A_n \sim  n^{-\frac{3}{2}}d_n^3 C^n.
$$
Does $\zeta(3)$ 
satisfy the criterion of Lemma~\ref{irr}? 
The Prime Number Theorem tells us the following: 
\begin{align*}
d_n = \lcm[1,2,\dots,n] &= \prod_{\substack{p: \text{ prime}\\
2\le p \le n}} p^{\left\lfloor\frac{\log(n)}{\log(p)}\right\rfloor}
\lesssim \prod_{\substack{p: \text{ prime}\\
2\le p \le n}} p^{\frac{\log(n)}{\log(p)}}
= \prod_{\substack{p: \text{ prime}\\
2\le p \le n}}e^{\log(n)} \\
&= n^{\pi(n)}\gtrsim n^\frac{n}{\log(n)}= e^{n}.
\end{align*}
Since the asymptotic inequalities are in the opposite directions,
the above estimate is  crude. However, since $C=17+12\sqrt{2}=33.9705\dots$ 
is large enough compaired to $e^3
=20.0855\dots$,
we have
$C^n\gg n^{-\frac{3}{2}}d_n^3$ for very large $n$.
In particular, there is  an $\e>0$ such that 
$$
\left|\zeta(3)-\frac{B_n}{A_n}\right|\sim \const \; C^{-2n} <
\frac{1}{\left(n^{-\frac{3}{2}}d_n^3 C^n\right)^{1+\e}}, \qquad n\gg 1.
$$
This completes the proof of the irrationality of  $\zeta(3)$.\hfill$\square$

\subsection{The differential equations behind the scene}
We start with two linear differential equations,
\begin{align}
\label{P}
&Pu=0, \qquad P= D^3-t(34D^3+51D^2+27D +5)+t^2(D+1)^3\\
\label{Q}
&Qu=0, \qquad Q=D^4-t(34D^3+51D^2+27D +5)+t^2(D+1)^2,
\end{align}
where $D=t \frac{d}{dt}$, and $t\in \bP^1$. 
To analyze these operators, let us first find solutions to the equation
$Pu(t)=0$. Put $u(t)=\sum_{n=0}^\infty u_n t^n$. Since
$Dt^n=nt^n$,   comparing the terms of $t^{n+1}$,
we find  
the same $3$-term recursion equation of Ap\'ery \eqref{Apery}
$$
(n+1)^3 u_{n+1} -(34n^3+51n^2+27n+5)u_n +n^3 u_{n-1}=0, \qquad n\ge 0.
$$
We notice here that $Pu=0$ for a formal
power series implies \eqref{Apery} for \emph{all} $n\ge 0$, 
not just $n\ge 1$ as in \eqref{Apery}. Therefore, a solution to \eqref{Apery} \emph{does not}
always give a solution of $Pu=0$. The recursion relation has two 
linearly independent solutions determined by the initial values
$u_0$ and $u_1$. The differential equation $Pu=0$ imposes another 
condition $u_1=5u_0$, hence the space of  solutions of \eqref{P}
analytic at the origin is $1$-dimensional.

Thus we find that the generating functions of  \eqref{A} and  \eqref{B} satisfy
\be
\label{ABdiff}
\begin{cases}
PA(t)=0
\\ Q\a(t)=0,
\end{cases} \qquad
\begin{cases}
PB(t)=6t
\\ Q\b(t)=6t.
\end{cases}
\ee
We are considering this pair of differential equations
to 
correspond to the two differential equations \eqref{xzdiff} and \eqref{qxz},
since the first equation has only regular singular points, while the second 
one has only one  irregular singular point at infinity and another regular singularity.
For these singularity analysis, it is useful to rewrite the operators in
terms of the usual expression:
$$
P=  \big(t^5-34t^4+t^3\big)\left(\frac{d}{dt}\right)^3 +
\big(6t^4-153t^3+3t^2\big)\left(\frac{d}{dt}\right)^2+
\big(7t^3-112t^2+t\big)\frac{d}{dt}
+t^2-5t,
$$
\begin{align*}
Q=t^4\left(\frac{d}{dt}\right)^4+
\big(-34t^4+6t^3\big)\left(\frac{d}{dt}\right)^3 &+
\big(t^4-153t^3+7t^2\big)\left(\frac{d}{dt}\right)^2
\\
& +
\big(3t^3-112t^2+t\big)\frac{d}{dt}
+t^2-5t.
\end{align*}
It is easy to see that Eqn.\eqref{Q} has the unique regular singular
point at $t=0$, and the unique irregular singular point at $t=\infty$. 
Its solution $\a(t)$ is the unique entire solution up to a constant factor. 
With a little calculation we  find that Eqn.\eqref{P} has 
regular singular points at $\{0, C^{-1}, C, \infty\}$, 
and no other singular points.

\begin{prob}
Find the $\hbar$-deformation family of $Q$, and 
identify the spectral curve through the semi-classical limit of this family. 
This spectral curve should be our \textbf{hidden curve}. 
\end{prob}

The growth order \eqref{growth} determines
the radius of  convergence of $A(t)$ and $B(t)$, which is
$C^{-1}=17-12\sqrt{2}$.
Both $A(t), B(t)$ and the linear combination $A(t)\zeta(3)-B(t)$
solve
\be
\label{D-1}
(D-1)Pu(t)=0,
\ee
and the radius of convergence of $A(t)\zeta(3)-B(t)$ is $C$.

\begin{thm}[\cite{CCGK}, p.141. Quantum Differential Equation]
The generating function $GW_{0,1}(X)$ of the type $(0,1)$ Gromov-Witten
invariants of $X$, a smooth model of Fano $3$-fold $V_{12}$, 
atisfies a \textbf{Quantum Differential Equation}.
\end{thm}

The state-of-the-art result \cite{GZ, Zagier} tells us that
$$
GW_{0,1}(X) = e^{-5t}\; \a(t), 
$$
and hence the  Quantum Differential Equation is
\be
\label{QDE}
\big(e^{-5t}\circ Q\circ e^{5t}\big) \;GW_{0,1}(X) = 0.
\ee

If the correspondence between opers and spectral curves work
for irregular singular differential operators, then $e^{-5t}\circ Q\circ e^{5t}$
should have a canonical $\hbar$-deformation family, and its
semi-classical limit should be the spectral curve corresponding to
$A(t)$.

Summing up the integral expression \eqref{B/A} with $t^n$
for all $n\ge 0$, 
we obtain
\be
\begin{aligned}
\label{period}
A(t)\zeta(3)-B(t) &=  
 \frac{1}{2}\int_0^1\int_0^1\int_0^1
 \sum_{n=0}^\infty
\left(\frac{xyz(1-x)(1-y)(1-z)}{1-z+xyz}\right)^n t^n
 \frac{dxdydz}{1-z+xyz}\\
 &=
  \frac{1}{2}\int_0^1\int_0^1\int_0^1
  \frac{dxdydz}{1-z+xyz-txyz(1-x)(1-y)(1-z)}.
 \end{aligned}
\ee
Beukers and Peters \cite{BP} considered the family of algebraic varieties
defined in $\bC^3$ by the equation
$$
1-z+xyz-txyz(1-x)(1-y)(1-z) = 0
$$
as a parameter $t\in \bP^1$ moves. They noticed that this is a 
family of K3 surfaces in $\bP^3$ after projectivize the formula
and blow-up all singularities. The family degenerates at
$t\in \{0, C, C^{- 1}, \infty\}$.

\section{Concluding Remark}

Our observation is that the generating function of all $g=0$, $n=1$ invariants
for a particular counting problem,
such as Catalan numbers or tree counts as discussed earlier, determines
a spectral curve. In terms of complex analysis over the spectral curve, 
we can generate all $(g,n)$ invariants through a recursive mechanism.
The $(0,1)$ invariants themselves satisfy a recursion, which 
is used to determine the analytic formula for the spectral curve.
And in terms of the rightly chosen  coordinate of the spectral curve,
the generating functions of the $(g,n)$ invariants turn out to be polynomials.
The asymptotic behavior of these polynomials then recover 
the Witten-Kontsevich formula, Virasoro constraint conditions,
and particular counting problems including the $\lam_g$-formula
and the Euler characteristic of $\cM_{g,n}$.

The fact that such mechanism exists makes the starting counting problem
very special. The method presented in these lectures are not meant to be
a subject for \emph{generalization}. Rather, it tells us how \emph{special}
these problems are, and this explains why they have 
deep connections with 
integrable systems, representation of Virasoro algebras,  moduli spaces of 
algebraic curves,  and so on.

Ap\'ery's recursions are in the sense extremely special. Again, they are not
a subject for  generalization. We are led to wondering how and why they
are so special. Once we understand these, we may be able to transplant 
the geometric situation to other equally special cases, and may obtain a new
arithmetic result. 

The Ap\'ery recursions for $\zeta(2)$ and $\zeta(3)$ 
are indeed recursions for the $(0,1)$ case of 
Gromov-Witten invariants of very special target manifolds. As such, the geometric 
origin of these recursions come from degenerations of rational curves
on these particular target spaces. In this way we \emph{understand} the 
origin of the Ap\'ery recursions from symplectic geometry. Its mirror partner
is the Gauss-Manin connection found in this  context, which
generate infinite sequence of \emph{periods}. The speed of the 
asymptotic behavior of these sequences translates into the irrationality of
$\zeta(2)$ and $\zeta(3)$.

It reminds us of the new proof of the Lawrence conjecture \cite{Lawrence}
due to Habiro,
on the integrality of certain $3$-manifold invariants over the ring generated by 
a root of unity  using his powerful \emph{universal} invariants
\cite{Habiro}. When considered for more general link invariants, one 
understands  the amazing \emph{symplectic nature} of the counting problem
(see for example, \cite{Cristina}). The role that the universal invariants 
play, in particular, in the context of arithmeticity \cite{GSWZ}, suggests
the parallel mirror symmetric point of view in these developments.

$$
\xymatrix{&\boxed{\text{Periods}}\ar[dr] \ar[dl]
\\
\boxed{\text{Symplectic Counting}}
&\ar[l]\boxed{\text{Mirror Symmetry}}\ar[r]&\boxed{
\text{Connections on Families}}	
\\
&\ar[ul]\boxed{\text{Quantization}}\ar[ur]
}
$$


\begin{appendix}

\section{The Lagrange Inversion Formula}

In  Appendix we give a brief proof of the
Lagrange Inversion Formula. For more 
detail, we refer to \cite{WW1927}.

\begin{thm}
Let $x=f(y)$ be a holomorphic function in $y$
defined on a neighborhood of $y=b$. Let $f(b) = a$,
and suppose $f'(b)\ne 0$. Then the inverse function
$y=y(x)$ is given by the following expansion
near $x=a$:
\begin{equation}
\label{eq:LIT}
y-b = \sum_{k=1} ^\infty
\frac{d^{k-1}}{dy^{k-1}}
\left.
\left(
\frac{y-b}{f(y)-a}
\right)^k
\right|_{y=b}
\frac{(x-a)^k}{k!}.
\end{equation}
\end{thm}

\begin{proof}
Let us recall the Cauchy integration formula
$$
\phi(s) = \frac{1}{2\pi i} \oint \frac{\phi(t)dt}{t-s},
$$
where $\phi(t)$ is a holomorphic function 
defined on a neighborhood of $t=s$, and 
the integration contour is a small simple loop 
inside this neighborhood
counterclockwisely rotating around the point $s$.
Since $x=f(y)$ is one-to-one near $y=b$, 
for a point $s$ close to $b$, we have
\begin{align*}
\frac{1}{f'(s)} 
&= \frac{1}{f'\big(f^{-1}(f(s)\big)}
\\
&=
\frac{1}{2\pi i}\oint \frac{df(t)}
{f'\big(f^{-1}(f(t)\big)\big(f(t)-f(s)\big)}
\\
&=
\frac{1}{2\pi i}\oint \frac{f'(t)dt}
{f'(t)\big(f(t)-f(s)\big)}
\\
&=
\frac{1}{2\pi i}\oint \frac{dt}
{f(t)-f(s)}.
\end{align*}
Therefore, assuming that $s$ is close enough to $b$,
we compute
\begin{align*}
y-b
&=
\int_b ^y 1\cdot ds 
=
\int_b ^y  
\left(\frac{1}{2\pi i}\oint \frac{f'(s)dt}
{f(t)-f(s)}
\right)
ds
\\
&=
\int_b ^y  
\left(\frac{1}{2\pi i}\oint \frac{f'(s)dt}
{\big(f(t)-a\big)-\big(f(s)-a\big)}
\right)
ds
\\
&=
\frac{1}{2\pi i}\int_b ^y\oint 
\frac{\frac{f'(s)}{f(t)-a}}
{1-\frac{f(s)-a}{f(t)-a}}
dtds
\\
&=
\frac{1}{2\pi i}\int_b ^y
\sum_{n=0} ^\infty
\oint
\frac{f'(s)}{f(t)-a}
\left(
\frac{f(s)-a}{f(t)-a}
\right)^n
dsdt
\\
&=
\frac{1}{2\pi i}\int_{f(b)} ^{f(y)}
\sum_{n=0} ^\infty
\oint
\frac{1}{f(t)-a}
\left(
\frac{f(s)-a}{f(t)-a}
\right)^n
df(s)dt
\\
&=
\frac{1}{2\pi i}\sum_{n=0} ^\infty
\oint
\frac{1}{\big(f(t)-a\big)^{n+1}}\cdot
\frac{\big(f(y)-a\big)^{n+1}}{n+1}dt
\\
&=
\frac{1}{2\pi i}\sum_{k=1} ^\infty
\oint
\frac{dt}{\big(f(t)-a\big)^k}
\cdot \frac{(x-a)^k}{k}
\\
&=
\frac{1}{2\pi i}\sum_{k=1} ^\infty
\oint
\frac{1}{(y-b)^k}
\left(
\frac{y-b}
{f(y)-a}
\right)^k
dy
\cdot \frac{(x-a)^k}{k}
\\
&=
\sum_{k=1} ^\infty
\left.
\frac{d^{k-1}}{dy^{k-1}}
\left(
\frac{y-b}
{f(y)-a}
\right)^k
\right|_{y=b}
\cdot \frac{(x-a)^k}{k(k-1)!}.
\end{align*}
\end{proof}

The following formula is a straightforward
application of the above
Lagrange Inversion Theorem.

\begin{cor} Let $f(y)$ be a holomorphic function
defined in a neighborhood of $y=0$. If $f(0)\ne 0$,
then the inverse function of 
$$
x=\frac{y}{f(y)} 
$$
 is given by
\begin{equation}
\label{eq:lif}
y=\sum_{k=1} ^\infty
\left.
\frac{d^{k-1}}{dy^{k-1}}
\big(f(y)\big)^k
\right|_{y=0}
\frac{x^k}{k!}.
\end{equation}
\end{cor}

\end{appendix}


\providecommand{\bysame}{\leavevmode\hbox to3em{\hrulefill}\thinspace}

\bibliographystyle{amsplain}

\end{document}